\documentclass[12pt]{article}
\usepackage[small,compact]{titlesec}
\usepackage{authblk}
\usepackage{amsmath,amssymb,amsfonts,mathabx,setspace,cases}
\usepackage{graphicx,caption,epsfig,subfigure,epsfig,wrapfig}
\usepackage{url,color,verbatim,algorithmic}
\usepackage[sort,compress]{cite}
\usepackage[ruled,boxed]{algorithm}
\usepackage[scaled]{helvet}
\usepackage[T1]{fontenc}
\usepackage[bookmarks=true, bookmarksnumbered=true, colorlinks=true,   pdfstartview=FitV,
linkcolor=blue, citecolor=blue, urlcolor=blue]{hyperref}
\usepackage{enumitem}
\usepackage{stmaryrd}
\usepackage{makecell,booktabs}
\usepackage{threeparttable}

\def\R{\mathbb{R}}
\def\P{\mathbb{P}}
\def\calE{\mathcal{E}}
\def\calF{\mathcal{F}}
\def\calN{\mathcal{N}}
\def\calH{\mathcal{H}}

\def\calS{\mathcal{S}}
\def\calX{\mathcal{X}}
\def\spaceX{\mathbb{X}}
\def\spaceY{\mathbb{Y}}

\def\mathspan{\mathrm{span}}
\def\Gbar{ {\overline{G}} }

\def\Abar{ {\overline{A}} }
\def\bbar{ {\overline{b}} }
\def\LGbar{ {\mathcal{L}_{\overline{G}}}  }

\def\btheta{\boldsymbol{\theta}}

\def\calA{\mathcal{A}}

\newcommand{\E}{\mathbb{E}}
\newcommand{\Ebracket}[1]{\mathbb{E}\left[{#1}\right]}

\newcommand{\innerp}[1]{\langle{#1}\rangle}

\newcommand{\kl}[1]{\mathrm{KL}\left(#1\right)}
\newcommand{\tr}{\mathrm{Tr}}
\newcommand{\var}{\mathrm{Var}}

\definecolor{mygrey}{gray}{0.75}

\newcommand{\RNum}[1]{\uppercase\expandafter{\romannumeral #1\relax}}

\newtheorem{theorem}{Theorem}

\newtheorem{assumption}[theorem]{Assumption}

\newtheorem{corollary}[theorem]{Corollary}

\newtheorem{definition}[theorem]{Definition}
\newtheorem{example}[theorem]{Example}

\newtheorem{lemma}[theorem]{Lemma}

\newtheorem{proposition}[theorem]{Proposition}
\newtheorem{remark}[theorem]{Remark}

\newenvironment{proof}[1][Proof]{\noindent\textbf{#1.} }{\ \rule{0.5em}{0.5em}}

\numberwithin{equation}{section}
\numberwithin{theorem}{section}

\title{Minimax rates for learning kernels in operators} 
\author[1]{Sichong Zhang}
\affil[1]{Department of Applied Mathematics and Statistics, Johns Hopkins University, Baltimore, USA. 
}
\author[2]{Xiong Wang}
\author[2]{Fei Lu} 
\affil[2]{Department of Mathematics, Johns Hopkins University, Baltimore, USA. 
}
\date{}
\begin{document}
\maketitle
\vspace{-6mm}
Learning kernels in operators from data lies at the intersection of inverse problems and statistical learning, providing a powerful framework for capturing non-local dependencies in function spaces and high-dimensional settings. In contrast to classical nonparametric regression, where the inverse problem is well-posed, kernel estimation involves a compact normal operator and an ill-posed deconvolution. To address these challenges, we introduce adaptive spectral Sobolev spaces, which unify Sobolev spaces and reproducing kernel Hilbert spaces, automatically discarding non-identifiable components and controlling terms with small eigenvalues. Within this framework, we establish the minimax convergence rates for the mean squared error under both polynomial and exponential spectral decay regimes. 
 Methodologically, we develop a tamed least squares estimator achieving the minimax upper rates via controlling the left-tail probability for eigenvalues of the random normal matrix; and for the minimax lower rates, we resolve challenges from infinite-dimensional measures through their projections.

\setcounter{tocdepth}{1}
 \tableofcontents

\section{Introduction}
 Kernels are effective in capturing nonlocal dependency, making them indispensable for designing operators between function spaces or tackling high-dimensional problems. Thus, the problem of learning kernels in operators arises in diverse applications, from identifying nonlocal operators in partial differential equations in \cite{bucur2016_NonlocalDiffusion,pang2020_NPINNsNonlocal,LangLu22,you2022data,you2024nonlocal,yu2024nonlocal} to image and signal processing in \cite{caflisch1981inverse,gilboa2009nonlocal,lou2010image}.

In such settings, one seeks to recover a kernel function $\phi$ in the forward operator $R_\phi: \spaceX\to\spaceY $ from noisy data of random input-output pairs $\{(u^{m},f^{m} )\}_{m=1}^{M}$, where
\begin{equation}\label{eq:R_phi_g}
f(x) = R_\phi[u](x) + \varepsilon(x)\text{ and } R_\phi[u](x) = \int_{\calS} \phi(s) g[u](x,s)ds, \quad x\in \calX.	
\end{equation}
Here, $\spaceX$ and $\spaceY$ are problem-specific function spaces, the functional $g$ is given, the set $\calX$ can be either a finite set or a domain in $\R^d$, $\calS\subset\R^d$ is a compact set, and $\varepsilon$ is observation noise that can be non-Gaussian; see Section \ref{sec:abstract_settings} for detailed model settings. In particular, the equation is interpreted in the weak sense when $\spaceY$ is infinite-dimensional.

The operator $R_\phi[u]$ can be nonlinear in $u$, but it is linear in the kernel $\phi$, and the output depends on $\phi$ nonlocally. Examples include integral operators with $g[u](x,s)= u(x-s)$ that is ubiquitous in science and engineering, the nonlocal operators with $g[u](x,s)= u(x+s)+u(x-s)- 2u(x)$ in nonlocal diffusion models \cite{you2024nonlocal,du2012_AnalysisApproximation}, and the aggregation operators with $g[u](x,s) = \partial_x[ u(x+s)u(x)] - \partial_x[ u(x-s)u(x)]$ in mean-field equations \cite{carrillo2019aggregation,LangLu22}; see Examples \ref{example:int_opt}--\ref{example:Aggr_opt} for details.

The problem of recovering the kernel $\phi$ from given data is at the intersection of statistical learning and inverse problems. In essence, it is a deconvolution from multiple function-valued input-output data pairs. The deconvolution renders it a severely ill-posed inverse problem, while the randomness of the data endows the problem with a statistical learning flavor. Thus, it is close to functional linear regression (FLR) \cite{hall2007methodology, yuan2010reproducing, wahba1990spline}, inverse statistical learning (ISL) \cite{BM18}, nonparametric regression \cite{Gyorfi06a, CuckerSmale02}, and classical inverse problem of solving the Fredholm equations of the first kind \cite{engl1996regularization, hansen1998rank}.

A fundamental question is the minimax convergence rate as the number of independent input-output pairs grows. In particular, it is crucial to understand how the severe ill-posedness inherent in the deconvolution affects the minimax rate and to clarify the connections between statistical learning and inverse problems.  

Building on the above pioneering work,  this study addresses the above question by establishing the minimax convergence rates for the ill-posed settings of polynomial and exponential spectral decays. We prove the minimax rates in a framework based on adaptive \emph{spectral Sobolev spaces} that connects inverse problems and statistical learning. This framework is rooted in the observation that the large-sample limit of statistical learning is a deterministic inverse problem, where the associated normal operator plays a key role. In classical nonparametric regression, the normal operator is the identity operator, whereas in learning kernels in operators, as well as in FLR and ISL, it is a compact operator. By exploiting the spectral decay of this normal operator, we construct adaptive spectral Sobolev spaces that discard the non-identifiable components in the null space of the normal operator. In particular, these Sobolev spaces include the reproducing kernel Hilbert spaces (RKHS) of the normal operator. Thus, this approach unifies Sobolev spaces and RKHS, thereby providing a robust theoretical foundation for statistical learning in ill-posed settings.

\subsection{Main results}

Our main result is the minimax convergence rate for estimating the kernel $\phi$ as the number of samples $M$ increases. With the default function space of learning being $L^2_\rho:= L^2(\calS,\mathcal{B}(\calS),\rho)$, we quantify the smoothness of the kernel by adaptive \emph{spectral Sobolev spaces}, 
$$H_\rho^\beta = \LGbar^{\beta/2}(L^2_\rho), \quad \beta\geq 0, $$
 where $\LGbar:L^2_\rho\to L^2_\rho$ is the normal operator of regression defined by $
\innerp{\LGbar\phi,\phi}_{L^2_\rho} = \E[ \| R_\phi[u]\|_\spaceY^2 ]
$
for all $ \phi\in L^2_\rho$. These Sobolev spaces are adaptive to the distribution of $u$ and the forward operator $R_\phi[u]$, and they automatically discard the non-identifiable components in the null space of the normal operator. In particular, the space $H^\beta_\rho$ with $\beta=1$ is the RKHS associated with the normal operator's integral kernel $\Gbar$. These spaces are unifying generalizations of the source sets in inverse problems (see, e.g., \cite[Eq.(3.29)]{engl1996regularization} and \cite[Eq.(2.5)]{BM18}), the periodic Sobolev spaces in spline regression (see, e.g., \cite[Chapter 2]{wahba1990spline}), and the RKHSs in functional linear regression in \cite{hall2007methodology,yuan2010reproducing,balasubramanian22unified}.

We establish the minimax convergence rates when the normal operators have either polynomial or exponential spectral decay. Let $\beta>0$ and denote $\phi_*$ the true kernel. 
\begin{itemize}
    \item When the spectral decay is polynomial, i.e., $\lambda_n\asymp n^{-2r}$ with $r> 1/4$, the minimax rate is  
    \[
       \inf_{\widehat \phi_M}\sup_{\phi_*\in H^\beta_\rho(L)}   \E_{\phi_*}\big[\|\widehat{\phi}_M-\phi_*\|_{L^2_\rho}^2\big] \asymp M^{-\frac{2\beta r}{2\beta r+2r+1}},
    \]
    where the infimum $\inf_{\widehat \phi_M}$ runs over all estimators $\widehat \phi_M$ using data $\{(u^{m},f^{m} )\}_{m=1}^{M}$.
      \item When the spectral decay is exponential, i.e., $\lambda_n\asymp e^{-rn}$ with $r> 0$, the minimax rate is 
    \[
       \inf_{\widehat \phi_M}\sup_{\phi_*\in H^\beta_\rho(L)}   \E_{\phi_*}\big[\|\widehat{\phi}_M-\phi_*\|_{L^2_\rho}^2\big] \asymp M^{-\frac{\beta}{\beta+1}}.
    \]
\end{itemize}
When the spectral decay is polynomial, the above minimax rate aligns with those reported in functional linear regression in \cite{hall2007methodology,yuan2010reproducing} and inverse statistical learning in \cite{BM18,Helin2024}. Although these studies employ different settings and methods, their minimax rates coincide because, in each case, the inverse problem in the large sample limit involves a compact normal operator (see Section \ref{sec:relatedwork} for detailed comparisons). In contrast, when the spectral decay is exponential, to the best of our knowledge,  our work is the first to establish an optimal minimax rate. Remarkably, this rate is independent of the decay speed exponent $r$ and depends solely on the smoothness exponent $\beta$, a result enabled by the spectral Sobolev space.

\vspace{2mm}\noindent\textbf{Main contributions.} 
The primary contribution of this study is to establish optimal minimax rates for learning operator kernels within a unifying framework that bridges inverse problems and statistical learning via adaptive spectral Sobolev spaces.

A major methodological contribution of this study is the introduction of a \emph{tamed least squares estimator} (tLSE) that achieves the minimax upper rate for ill-posed statistical learning problems. Unlike previous work \cite{hall2007methodology,yuan2010reproducing,BM18} that focused on RKHS-regularized estimators, we establish the minimax upper rate using a tLSE, which mitigates the impact of small eigenvalues through a cutoff strategy. Originally introduced in \cite{wang2023optimal} for learning interaction kernels in interacting particle systems, where the inverse problem is well-posed in the large-sample limit, the tLSE framework is extended in this paper to address ill-posed settings. In particular, our approach features two key technical innovations: (i) a tight bound for the sampling error derived via singular value decomposition, and (ii) a relaxed PAC-Bayesian inequality to bound the left-tail probability of the eigenvalues of the random normal matrix under a mild fourth-moment condition.

Furthermore, this study introduces technical innovations to address infinite-dimensional (non-Gaussian) noises when establishing minimax lower rates. We derive the minimax lower rate using Assouad's method \cite{assouad1983deux,yu97Assouad,hall2007methodology}, reducing the estimation problem to hypothesis testing on a hypercube through binary coefficients in the eigenfunction expansion. Crucially, to handle distribution-valued noise in the infinite-dimensional output space, we control the total variation distance via the Kullback-Leibler divergence between restricted measures on filtrations, employing the monotone class theorem as detailed in Section \ref{sec:tv_meas_inftyD}.

\subsection{Related work}\label{sec:relatedwork}
The learning of kernels in operators is closely related to inverse problems and their statistical variants, functional linear regression, and classical nonparametric regression.

\vspace{2mm}\noindent\textbf{Functional linear regression.} Minimax rates are well-established for functional linear regression in \cite{hall2007methodology,yuan2010reproducing,balasubramanian22unified}, where the task is to estimate the slope function $\phi$ and the inception $\alpha$ in the model 
$	Y_i = \alpha + \int \phi(s)X_i(s)ds + \varepsilon_i
$ 
from data $\{(X_i, Y_i)\}_{i=1}^M$, where $\{\varepsilon_i\}$ are i.i.d.~$\R$-valued noise. 
 The minimax-optimal estimators are typically constructed via RKHS regularization with user-selected RKHSs, under the assumption that the covariance operator (equivalent to our normal operator) is strictly positive definite. Our work extends the setting from scalar-on-function to function-on-function regression. We quantify the smoothness of $\phi$ by the spectral Sobolev spaces defined through the normal operator, and these spaces automatically provide RKHSs for the learning. Importantly, our tLSE offers an alternative approach to RKHS regularization for establishing the minimax upper rate.
    
Minimax rates for prediction accuracy (also called excess/prediction risk) have been established in \cite{cai_yuan2012} for scalar-on-function regression and in \cite{SunDuWangMa2018,Dette_Tang2024} for function-on-function regression in the form 
$	Y_i(x) = \alpha(x) + \int \phi(x,y)X_i(y)dy + \varepsilon_i(x)
$. 
We note that the predictor error is for forward estimation, which contrasts with the inverse problem of learning kernels in operators.

\vspace{2mm}\noindent\textbf{Inverse problems and their statistical variants.} Classical ill-posed inverse problem solves $\phi$ in the model $A\phi(x_i) + \varepsilon_i =f(x_i)$ from discrete data $\{f(x_i)\}_{i=1}^M$, where $A$ is a given compact operator and $\{x_i\}$ are deterministic meshes. Due to the vastness of the literature, we direct readers to \cite{engl1996regularization,hansen1998rank,wahba1990spline}, among others, for comprehensive reviews. A prototype example is the Fredholm equations of the first kind, which corresponds to Model \eqref{eq:R_phi_g} with $M=1$. The convergence of various regularized solutions has been extensively studied when the mesh refines, including the RKHS-regularized estimators in \cite{wahba1973convergence,wahba1977practical}. When $\{x_i\}$ are random samples, the problem is called inverse statistical learning in \cite{BM18} and statistical inverse learning problems in \cite{Helin2024}, where the minimax rate has been established. In these problems, due to the limited information from the data, it is natural to consider the estimator in the spectral Sobolev space of the normal operator $A^*A$, as illustrated in \cite{engl1996regularization,BM18, Helin2024}. Our study adopts this idea by using the normal operator from the inverse problem in the large sample limit. Additionally, learning kernels in operators can be viewed as estimating $\phi$ in the model $A_i\phi + \varepsilon_i =f_i$ from data $\{(A_i,f_i)\}$, a statistical learning inverse problem.

\vspace{2mm}\noindent\textbf{Minimax rates for nonparametric regression.} 
In classical nonparametric regression, where the goal is to estimate $f(x)= \E[Y|X=x]$ from samples $\{(X_i,Y_i)\}$, the minimax rate is a well-studied subject with a range of established tools (see, e.g., \cite{Gyorfi06a,tsybakov2008introduction,CuckerSmale02,van2000asymptotic,Wainwright2019} for comprehensive reviews). Common techniques for establishing lower bounds include the Le Cam, Assouad, and Fano methods, while upper bounds are typically proved using empirical process theory combined with covering arguments and chaining techniques or RKHS-regularized estimators \cite{deVito2005learning,caponnetto07,rosasco2010learning}. These approaches benefit from the well-posed nature of the classical regression problem, where the inverse problem in the large sample limit is characterized by an identity normal operator. Consequently, universal Sobolev spaces or H\"older classes are naturally employed to quantify function smoothness. 

In contrast, our study addresses an ill-posed regression problem, where the normal operator is compact. This setting motivates the use of adaptive spectral Sobolev spaces, as summarized in Table \ref{tab:SL-vs-kernel}. Although classical methods extend to well-posed problems, such as learning interaction kernels for particle systems (see, e.g., \cite{LZTM19pnas,LMT21,LMT21_JMLR}), they do not directly apply to our framework for establishing upper bounds. To overcome this challenge, we build upon the tLSE method introduced in \cite{wang2023optimal}. In our adaptation, the bias-variance trade-off is closely linked to the spectral decay of the normal operator. Thus, our study extends the tLSE method into a versatile tool for proving minimax upper rates for both well-posed and ill-posed statistical learning inverse problems.

\begin{table}[htbp]
\caption{Comparison of well-posed and ill-posed statistical learning problems: the normal operators in the large sample limit, the Sobolev spaces, the dominating orders of the bias and variance terms.
}\label{tab:SL-vs-kernel}
\begin{threeparttable}
\renewcommand{\arraystretch}{1.2}
\begin{tabular}{c||c|c|c|c}
\Xhline{1.2pt} 
\begin{tabular}[c]{@{}c@{}}Learning problem \\ (at $M = \infty$)\end{tabular} & \begin{tabular}[c]{@{}c@{}}Normal\\ operator\end{tabular} & \begin{tabular}[c]{@{}c@{}}Sobolev\\ space\end{tabular} &\, \,  Bias \, \, & Variance \, \,  \\ \Xhline{1pt} 
\textbf{Well-posed}  & $I$   & $H^\beta $     & $n^{-\frac{2\beta}{d}}$   & $n/M$\\ \hline
\textbf{Ill-posed}   & \begin{tabular}[c]{@{}c@{}}$\LGbar$ \\ (compact)\end{tabular}     & $H^\beta_\rho = \LGbar^{\frac{\beta}{2}}\left(L_\rho^2\right)$   & $\lambda_n^\beta$ & \, \,  $\begin{cases}  n^{1+2r}/M & \text{if } \lambda_n \asymp n^{-2r},\\    e^{rn}/M & \text{if } \lambda_n \asymp e^{-rn}\\    \end{cases}$ \\ \Xhline{1.2pt} 
\end{tabular}
 \begin{tablenotes}[para,flushleft]
{\footnotesize Here, $H^{\beta} $ is the classical periodic Sobolev space associated with the operator $\left(-\Delta\right)^{-1}$ on $[0,2\pi]^d$, which has a spectral decay of order $n^{-2/d}$. The space $H_\rho^\beta$ is defined through the normal operator $\LGbar$, whose eigenvalues are $\{\lambda_n\}_{n\geq 1}$. In the bias-variance tradeoff, the dominating order in the bias depends on the smoothness quantified by the Sobolev space, and the dominating order of the variance depends on the spectral decay of the normal operator.}
  \end{tablenotes}
  \end{threeparttable}

\end{table}

The rest of the paper is organized as follows. Section \ref{sec:spaces} introduces the model settings and defines the function spaces that are adaptive to this learning problem. Section \ref{sec:uprate} proves the minimax upper rates, which are achieved by the tLSE. Section \ref{sec:lower} presents the proofs for the lower minimax rates. We postpone technical proofs to the Appendix. 

\vspace{2mm}\noindent\textbf{Notations.} 
Hereafter, we denote the pairing between a Hilbert space $\spaceY$ and its dual action $z$ by $\innerp{z,y}$ with $y\in \spaceY$, and use $\innerp{\cdot,\cdot}_\spaceY$ to denote the inner product of $\spaceY$ as a Hilbert space. The underlying probability space in this study is complete and is denoted by $(\Omega,\calF,\P)$. We denote by $\P_\phi$ the distribution of the data from the model with kernel $\phi$.  $\E$ and $\E_\phi$ denote expectations with respect to $\P$ and $\P_\phi$, respectively.  
We simplify the notation by using $L^2_\rho:= L^2(\calS,\mathcal{B}(\calS),\rho)$ and $L^2(\Omega):= L^2(\Omega,\calF,\P)$ for the spaces of square-integrable functions and random variables, respectively. We denote by $\phi_*= \sum_{k=1}^\infty\theta_k^*\psi_k\in L^2_\rho$ the true kernel, where $\{\psi_k\}$ is an orthonormal basis of $L^2_\rho$.

 \section{Function spaces of learning}\label{sec:spaces}

\subsection{Abstract model settings}\label{sec:abstract_settings}
Consider the problem of estimating the parameter $\phi\in L^2_\rho$ in the operator equation  
\begin{align}\label{eq:model_general}
f= R_\phi[u] + \varepsilon  
\end{align} 
from data consisting of random sample input-output pairs $\{(u^m,f^m)\}_{m=1}^M$. Here, $\calS\subset\R^d$ is a compact set and $\rho$ is a Borel measure on $\calS$. Suppose $\spaceX$ is a Banach space and $\spaceY$ is a separable Hilbert space. We make the following assumptions about the forward operator $R_\phi:\spaceX\to \spaceY$, the distribution of the input $u$, and the noise $\varepsilon$. 
\begin{assumption}[Forward operator and input distribution]
\label{assum:model}
The forward operator is linear in the parameter, and the normal operator is compact:    
\begin{itemize}
        \item \emph{Linearity.} $R_\phi[u]$ is linear in $\phi$, i.e., $R_{\phi+\psi}[u]=R_{\phi}[u]+ R_{\psi}[u] $, $\forall \phi,\psi\in L^2_\rho$ and $u\in \spaceX$.   
        \item \emph{Spectral decay.} The \textbf{normal operator} $\LGbar: L^2_{\rho}\to L^2_{\rho}$ defined by
		  \begin{equation}\label{eq:normal_opt}
			\begin{aligned}
				\langle\LGbar\phi,\psi \rangle_{L_\rho^2} 
		         &= \E [\langle R_\phi[u], R_\psi[u]\rangle_{\spaceY} ],\quad  \forall \phi,\psi\in L^2_\rho,
		  	\end{aligned}
		  \end{equation}
		     is nonnegative, self-adjoint, compact, and has its positive eigenvalues $\{\lambda_k\}_{k=1}^\infty$ decaying either polynomially or exponentially, i.e., there exist $b\geq a>0$ such that    
		 \begin{enumerate}[label=$(\mathrm{A\arabic*})$]
		 \item \emph{ \bf Polynomial decay: } $ ak^{-2r}\leq \lambda_k\leq bk^{-2r}$ with $r>1/4$; or  \label{assum:polydecay}
		 \item \emph{ \bf Exponential decay:} $ a\exp(-rk)\leq \lambda_k\leq b\exp(-rk)$ with $r>0$. 
		 \label{assum:expdecay}
		 \end{enumerate}
		 Note that in either case, $\LGbar$ is Hilbert-Schmidt with $\sum_{k=1}^\infty \lambda_k^2<+\infty$. 
\end{itemize}
\end{assumption}

By the linearity of the operator, the learning of the kernel is a linear regression problem. The normal operator comes from the variational inverse problem in the large sample limit (see Section \ref{sec:inversionOperator}), and it is a self-adjoint compact operator, which we prove for Model \eqref{eq:R_phi_g} in Section \ref{sec:deconv}. 

In particular, the spectral decay condition is commonly used for deterministic ill-posed inverse problems (see, e.g., \cite{engl1996regularization,hansen1998rank}) and statistical inverse problems (see, e.g.,\cite{hall2007methodology,BM18,yuan2010reproducing}). It quantifies the ill-posedness of the inverse problem.  

\begin{assumption}[Conditions on the noise.]
\label{assump:noise-X-general}
The noise $\varepsilon$ is independent of $u$, and it is a linear map $\varepsilon: \spaceY\to L^2(\Omega)$, $y \mapsto \innerp{\varepsilon,y}$, satisfying the following two conditions.
\begin{enumerate}[label=$(\mathrm{B\arabic*})$]
 \item  \label{assum:noise_up} It is centered and square-integrable, i.e., for all $y\in \spaceY$, $\E[\innerp{\varepsilon,y}]=0$ and 
 \begin{equation}\label{ineq:noise_up}
     \E\left[\innerp{\varepsilon,y}^2\right] \leq\sigma^2\|y\|_\spaceY^2 
 \end{equation}
 for some $\sigma>0$ that is uniform for $y\in\spaceY$. 
 \item\label{assum:noise_low} For some orthonormal basis $\{y_i\}$ of $\spaceY$, the distribution of $\left(\innerp{\varepsilon,y_1},\cdots,\innerp{\varepsilon,y_N}\right)$ has a probability density function $p_N$ in $\R^N$ (with respect to the Lebesgue measure). Moreover, $p_N$ satisfies
 \begin{equation}\label{ineq:noise_low}
     \kl{p_N,p_N(\cdot +v)}= \int_{\R^N}\log\left(\frac{p_N(x)}{p_N(x+v)}\right)p_N(x)\,dx \leq \frac{\tau}{2}\|v\|^2,
 \end{equation} 
 for a constant $\tau>0$ that is uniform for all $N$ and $v\in\R^N.$ 
 \end{enumerate}
\end{assumption}
Conditions \ref{assum:noise_up} and \ref{assum:noise_low} are used for the minimax upper and lower rates, respectively.

The noise can be either Gaussian or non-Gaussian, and the space $\spaceY$ can be either finite- or infinite-dimensional. When $\spaceY$ is finite-dimensional, the linear map induced by a Gaussian random variable satisfies both conditions, i.e., $    \left(\innerp{\varepsilon,y_1},\cdots,\innerp{\varepsilon,y_N}\right)\sim\calN(0,\sigma^2\mathrm{Id}_N)$ satisfies \eqref{ineq:noise_up} and \eqref{ineq:noise_low} with $\tau = 1/\sigma^2$. A non-Gaussian example is the logistic distribution, that is, the random vector $\left(\innerp{\varepsilon,y_1},\cdots,\innerp{\varepsilon,y_N}\right)$ has i.i.d.~marginal components with probability density function $p(x) = e^{-x}(1+e^{-x})^{-2}$, and it satisfies \eqref{ineq:noise_up} with $\sigma^2= \pi^2/3$ and \eqref{ineq:noise_low} with $\tau = 25/6 $; see {Example \ref{eg:logistic}} in the Appendix for details. 

When $\spaceY$ is infinite-dimensional, the isonormal Gaussian process indexed by $\spaceY$ satisfies Conditions \ref{assum:noise_up}--\ref{assum:noise_low} with $\sigma=1$ and $\tau=1$. 
Recall that an isonormal Gaussian process indexed by $\spaceY$, denoted by $\varepsilon=(\innerp{\varepsilon,y},y\in \spaceY)$, is a family of centered Gaussian random variables satisfying $\E[\innerp{\varepsilon,h}\innerp{\varepsilon,g}] =\innerp{h,g}_{\spaceY}$ for all $h,g\in \spaceY$.
 In this case, Eq.\eqref{eq:model_general} is interpreted in the weak sense as  
\begin{align*}
\innerp{f,y} = \innerp{R_\phi[u],y}_\spaceY + \innerp{\varepsilon,y}, \,\, \forall y\in \spaceY. 
\end{align*} 
In particular, when $\spaceY= L^2(\calX,\mathcal{B}(\calX), \nu)$ with $\mathcal{B}(\calX)$ being the Borel sets of a domain $\calX\subseteq \R^d$ and $\nu$ being a  $\sigma$-finite measure without atoms, the $L^2(\Omega)$-valued measure $\varepsilon(A):= \innerp{\varepsilon,\mathbf{1}_A}$ for $A\in \mathcal{B}(\calX)$ is the \emph{white noise} on  $(\calX,\mathcal{B}(\calX))$; see, e.g., \cite[Section 1.1]{Nua06} and \cite{da2006introduction,hoffmann2008nonlinear,Hu2016analysis}. For example, when $\spaceY= L^2([0,1])$, the white noise $\varepsilon$ is formally the derivative of the standard Brownian motion $B(x)$ (which is called a generalized stochastic process in \cite[Section 3.1]{oksendal2013_sde}), and Eq.\eqref{eq:model_general} is formally a stochastic differential equation 
$ f(x) = \dot{X}(x) = R_\phi[u](x) + \dot{B}(x)$,  
which is interpreted as $X(x) = X(0)+ \int_0^x R_\phi[u](z)dz + B(x)$ for $x\in [0,1]$.

Importantly, in practice, discretization connects finite- and infinite-dimensional spaces. For example, when $\calX$ is an interval, a partition $\calX = \bigcup_{i=1}^N A_i$ connects the above infinite-dimensional space $\spaceY= L^2(\calX,\mathcal{B}(\calX), \nu)$ with a finite-dimensional observation space $\widetilde \spaceY= L^2(\widetilde \calX,\widetilde\nu)$ when one takes $\widetilde \calX= \{x_i\}_{i=1}^N$ and $\widetilde\nu(x_i) = \nu(A_i)$, where the $\{A_i\}$ are pairwise disjoint intervals and $x_i\in A_i$ for each $1\leq i \leq N$. In particular, when $R_\phi[u]$ is a piecewise constant function on the partition and when the test functions are $\{\mathbf{1}_{A_i}\}_{i=1}^N$, the above weak form equation leads to a discrete model 
\begin{align*}
f(x_i):= \innerp{f,\mathbf{1}_{A_i}} = \innerp{R_\phi[u],\mathbf{1}_{A_i}}_\spaceY + \innerp{\varepsilon,\mathbf{1}_{A_i}} =  R_\phi[u](x_i) + \varepsilon_i, \,\, 1\leq i\leq N,
\end{align*} 
where the noise has a distribution $(\varepsilon_1,\ldots,\varepsilon_N) \sim \mathcal{N}(0,\mathrm{diag}(\widetilde\nu(x_i)_{1\leq i\leq N}))$.

\subsection{Learning the convolution kernels}\label{sec:deconv}
We show in this section that the deconvolution problem of learning the kernel in Model \eqref{eq:R_phi_g} from data satisfies the abstract model settings in Section \ref{sec:abstract_settings}. 
 
 In practice, when fitting Model \eqref{eq:R_phi_g} to data, neither the measure $\rho$ nor the associated function space $L^2_\rho$ is predetermined. Instead, one may choose $\rho$ in a data-driven manner so that it reflects how the data explore the kernel $\phi$. Depending on the smoothness and boundedness of $g[u]$, various definitions of $\rho$ can be adopted to ensure that the normal operator in \eqref{eq:normal_opt} is compact. In the following, we define $\rho$ as the mean square average weight induced by $g[u]$ on the set  
$
\calS := \overline{\Big\{ s : \E \Big[\int_\calX \big|g[u](x,s)\big|^2 \,\nu(dx)\Big] > 0 \Big\}},
$
though one may alternatively choose the Lebesgue measure on $\calS$ or define $\rho$ through the mean of $|g[u](x,s)|$ as in \cite{LangLu21id}.

\begin{definition}[Exploration measure $\rho$.] \label{def:rho}
For {\rm Model} \eqref{eq:R_phi_g} with $\spaceY = L^2(\calX,\nu)$, suppose that the distribution of the data $u$ satisfies  
 \begin{equation}\label{eq:gL2}
    Z:= \E\left[\int_\calX\int_\calS | g[u](x,s)|^2 \nu(dx) \, ds \right] < +\infty.
 \end{equation}
We define an exploration measure $\rho$ by its density function given by
 \begin{equation}\label{def:rho_dot}
     \dot{\rho}(s) =\frac{1}{Z} \E\left[\int_\calX | g[u](x,s)|^2 \nu(dx) \right].
 \end{equation}
 \end{definition}
 
 With the above exploration measure $\rho$, the forward operator of Model \eqref{eq:R_phi_g} defines a square-integrable $\spaceY$-valued random variable when the volume of $\calS$ is finite. That is, for all $\phi\in L_\rho^2$, we have, by Cauchy-Schwartz inequality, 
 \begin{align*}
     \E\left[\|R_{\phi}[u]\|_\spaceY^2\right] &= \E\left[\int_\calX\left(\int_\calS\phi(s)g[u](x,s)ds\right)^2\nu(dx)\right]\\
     &\leq \E\left[\int_\calX\int_\calS\phi^2(s)g^2[u](x,s)ds\,\nu(dx)\right] {\rm{vol}}(\calS)=  {\rm{vol}}(\calS)\, Z \|\phi\|_{L_\rho^2}^2<+\infty. 
 \end{align*}
 
Furthermore, the normal operator of learning the kernel $\phi$ is a compact integral operator, as the next proposition shows (see the Appendix for its proof). 
\begin{proposition}[Compact normal operator]\label{prop:compactLG}
For {\rm Model} \eqref{eq:R_phi_g} satisfying \eqref{eq:gL2} and ${\rm{vol}}(\calS)<+\infty $, its normal operator $\LGbar$ is nonnegative, self-adjoint, compact, and has an integral kernel $\Gbar\in L^2(\rho\otimes \rho)$ with $\rho$ in {\rm Definition \ref{def:rho}}.    
\end{proposition}

The spectral decay rate of the normal operator $\LGbar$ depends on the smoothness of the integral kernel $\Gbar$ and the measure $\rho$; in particular, it depends on the dimension $d$ of $\calS\subset  \R^d$; see, e.g., \cite{carrijo2020approximation,ferreira2009eigenvalues,scetbon2021spectral}. For instance, the integral kernel $\Gbar$ in Example \ref{example:int_opt} below exhibits polynomially decaying eigenvalues (with a rate $\lambda_k \asymp k^{-\frac{2}{d}}$ when the example is generalized to $\calS= [0,1]^d$), while Gaussian kernels produce exponentially decaying eigenvalues, as illustrated in \cite[Chapter 4.3.1]{Rasmussen2006GP}.

Although it is relatively straightforward to construct integral operators with a prescribed spectral decay, it remains challenging to impose general explicit conditions on the function $g[u]$ to achieve a specific decay rate, especially when the measure $\rho$ is adaptive to the data distribution.

We consider three illustrative applications for learning kernels in operators: integral operators, nonlocal operators, and an aggregation operator. In practice, the corresponding normal operators may exhibit various types of spectral decays, including polynomial or exponential decays, depending on the data distribution (see, e.g., \cite{LuOu25,LangLu22,LAY22}). Here, we construct data distributions to achieve the desired spectral decay in the example of the integral operator. 

\begin{example}
\label{example:int_opt} 
 Let $\mathrm{supp}(\phi)\subset \calS= [0,1]$ and consider the \emph{integral operator} $R_\phi:\spaceX\to \spaceY$ defined by 
	\begin{equation}\label{eg:integral}
	R_\phi[u](x) = \int_{[-1,2]} \phi(x-y) u(y)dy = \int_{\calS}\phi(s) u(x-s)  ds, \, x\in \calX =[0,1]\, , 
	\end{equation}
	where $\spaceY = L^2(\calX,\nu)$ with $\nu$ being the Lebesgue measure on $\calX$, and 
	$
 \spaceX =\{u\in L^2([-1,2]):$ $u(x) = \sum_{k=1}^\infty \alpha_k \cos(2\pi k x),$ $\sum_{k=1}^\infty \alpha_k^2<+\infty \}
 $. Here, $g[u](x,s) = u(x-s)$ for $(x,s)\in \calX \times\calS$ and the second equality holds because the support of the kernel $\phi$ is in $\calS = [0,1]$. Let the random input functions be 
 \begin{equation}\label{eq:input_u}
 	u(x):=\sum_{k=1}^\infty X_k \cos(2\pi k x)\,,  
 \end{equation} 
 where $\{X_k\}_{k=1}^\infty$ is a sequence of independent $\calN(0,\sigma_k^2)$ random variables with $\sum_{k=1}^\infty \sigma_k^2<+\infty$. The exploration measure $\rho$ in \eqref{def:rho_dot} has density $\dot{\rho}\equiv 1$ due to the periodicity of $u$: 
   \begin{align*}
        \dot{\rho}(s) \, \propto \, \E\left[\int_0^1 | u(x-s)|^2 \nu(dx) \right]\equiv \E\left[\int_0^1 | u(x)|^2 \nu(dx) \right]
    \end{align*}
    Hence, $\Gbar(s,s')=G(s,s')$ and   
     	\begin{align*}
     		&\ \Gbar(s,s')  =  G(s,s')
     		= \int_\calX \E[u(x-s)  u(x-s') ]\, dx \\
     		= &\ \sum_{j=1}^\infty \sum_{k=1}^\infty \int_0^1  \E\left[  X_jX_k \cos(2\pi j (x-s))\cos(2\pi k (x-s'))  \right]\, dx \\
     		 = &\ \sum_{k=1}^\infty\frac{\sigma_k^2}{2} \cos(2\pi k (s-s'))=  \sum_{k=1}^\infty \frac{\sigma_k^2}{2} [\cos(2\pi ks)\cos(2\pi ks') + \sin(2\pi k s)\sin(2\pi k s')].
     		\end{align*} 
            Recall that $\{\mathbf{1}\}\cup\{\sqrt{2}\cos(2\pi k s), \sqrt{2}\sin(2\pi k s)\}_{k=1}^\infty$ is an orthonormal basis of $L^2_\rho$. Thus, the eigenvalues of $\LGbar:L^2_\rho\to L^2_\rho$ are $\lambda_{2k-1}=\lambda_{2k}= \sigma_k^2/4$ for $k\geq 1$ and $\lambda_0=0$, with eigenfunction $\mathbf{1}$. Thus, we obtain the polynomial or exponential decay if $\sigma_k$ decays accordingly. Notably, when $\sigma_k^2 = \frac{4}{(2\pi k)^{2}} $ for $k\geq 1$,  the RKHS $H_\Gbar$ is the Sobolev space with periodic functions (see, {\rm \cite[Chapter 1-2]{wahba1990spline}})
     		     		   \begin{align*}
				         W^0_{1,per} =\left \{ \phi' \in L^2([0,1]): \int_0^1 \phi(s)ds= 0, \ \phi(0) = \phi(1) \right\}.  
      \end{align*}
\end{example}

\begin{example}
\label{example:nonlocal_opt}
Consider the \emph{nonlocal operator} with radial interaction kernel:
	\begin{equation}\label{eg:nonlocal}
	R_\phi[u](x) = \int_{|y|\leq 1} \phi(|y|) [u(x+y)-u(x)] dy = \int_{[0,1]}\phi(s) g[u](x,s) ds,  \, x\in \calX=[0,1]\, 
	\end{equation}
	where $g[u](x,s) = u(x+s)+u(x-s)-2u(x)$ for $(x,s)\in \calX\times \calS $ with $\calS= [0,1]$. It arises in the nonlocal PDE $\partial_{tt}u = R_\phi[u]$ for peridynamics; see, e.g., {\rm\cite{you2024nonlocal,LAY22}}. Consider the same function spaces $\spaceX$ and $\spaceY $ and random inputs $u$ as in {\rm Example \ref{example:int_opt}}. Direct computations in the Appendix show that 
		\[G(s,s')= 2\sum_{k=1}^\infty \sigma_k^2 \,\bigl(\cos(2\pi k \,s)-1\bigr) \bigl(\cos(2\pi k \,s')-1\bigr)\,. \] 
	Clearly, $Z=\int_\calS G(s,s)ds<+\infty$, so the normal operator $\LGbar:L^2_\rho\to L^2_\rho$ is compact. 
\end{example}

\begin{example}
\label{example:Aggr_opt}
 Consider estimating $\phi:\calS=[0,1]\to \R$ in the \emph{aggregation operator} 
\begin{equation}\label{eg:aggregation}
		R_\phi[u](x) = \int_{|y|\leq 1} \phi(|y|)\frac{y}{|y|} \partial_x[ u(x+y)u(x)] dy = \int_{\calS} \phi(s) g[u](x,s) ds,\quad  x\in \calX= [0,1]
		\end{equation}
		with $g[u](x,s) = \partial_x[ u(x+s)u(x)] - \partial_x[ u(x-s)u(x)]$ for $(x,s)\in \calX\times \calS $. The operator 
 $R_\phi[u] = \nabla\cdot(u\nabla\Phi*u)$ arises in the mean-field equation $\partial_t u = \nu \Delta u + \nabla \cdot  (u\nabla\Phi*u)$ on $\R^1$ for interacting particle systems {\rm\cite{carrillo2019aggregation,LangLu21id}}, where $\Phi$ is a radial interaction potential satisfying $\phi= \Phi'$. Letting $\spaceY = L^2(\calX)$ and $\spaceX =\{u\in C_c^1([-1,2]): u(x)\geq 0, \forall x\in [-1,2] \} $. Consider the random input functions: 
$$
u(x,\omega)= 1 + \sum_{n=1}^\infty
a_n\,\zeta_n(\omega)\,\cos\bigl(2\pi n\,x\bigr),
\quad x\in[-1,2], 
$$
where $\{\zeta_n\}_{n\ge1}$ are i.i.d.~Rademacher random signs ($\P(\zeta_n=\pm1)=\tfrac12$), and $a_n>0$ with $\sum_n n a_n<1$. Then, the explicitly computed $G(s,s')$ in the Appendix shows that $Z=\int_\calS G(s,s)ds<+\infty$, so the normal operator $\LGbar:L^2_\rho\to L^2_\rho$ is compact. 
	\end{example}

  \subsection{Inverse problem in the large sample limit}
\label{sec:inversionOperator}
To understand the statistical learning problem, we start with the deterministic inverse problem in the large sample limit, which lays the foundation for defining the function spaces for learning.

In a variational inference approach, we find an estimator by minimizing the empirical loss function of the data samples $\{(u^m,f^m)\}_{m=1}^M$: 
\begin{equation}\label{eq:lossFn}
	\calE_M(\phi) := \frac{1}{M}\sum_{m=1}^M \big( \|R_\phi[u^m]\|_\spaceY^2  - 2\innerp{f^m, R_\phi[u^m]} \big) .
	\end{equation}
This loss function is the scaled log-likelihood of the data when the noise is standard Gaussian. In particular, when $\spaceY$ is infinite-dimensional and the noise is white, it is $\calE_M(\phi)= - \frac{2}{M}\log\frac{d\P_\phi}{d\P_0} $, where the Radon-Nikodym derivative $\frac{d\P_\phi}{d\P_0}$ is given by the Cameron-Martin formula; see {Remark \ref{rmk:CM-space}} in the Appendix.

By the strong Law of Large Numbers and {\rm Assumption \ref{assum:model}}, we have  
\[
\calE_\infty(\phi):=\lim_{M\to \infty}  \calE_M(\phi)  = \E[\|\R_\phi[u]\|_\spaceY^2] - 2\E[\innerp{f,R_\phi[u] }] = \innerp{\LGbar\phi,\phi}_{L^2_\rho} - 2\innerp{\LGbar\phi_*,\phi}_{L^2_\rho}, 
\]
where $\phi_*$ denotes the true kernel. Hence, the set of minimizers of $\calE_\infty$ is 
\[
\{\phi\in L^2_\rho:  \nabla \calE_\infty(\phi) =2(\LGbar\phi- \LGbar\phi_*)= 0\} = \phi_*+ \ker(\LGbar). 
\]
That is, the minimizer of $\calE_\infty$ is non-unique unless $\ker(\LGbar)= \{0\}$. However, as shown in the previous section, the null space of $\LGbar$ may have non-zero elements. 

Importantly, we can only identify the projection of $\phi_*$ in $\ker(\LGbar)^\perp$ when minimizing the loss function with infinite samples. That is, when solving $\nabla \calE_\infty(\phi) = 0$, we can only identify $\widehat \phi = \LGbar^{\dagger} (\LGbar\phi_*) = P_{\ker(\LGbar)^\perp}\phi_*$, which is the least squares estimator with minimal norm. Thus, it is crucial to restrict the estimation in $\ker(\LGbar)^\perp$. 
This motivates us to define spectral Sobolev spaces based on the normal operator in the next section.

\subsection{Spectral Sobolev spaces}\label{sec:SSS}
We introduce spectral Sobolev spaces adaptive to the model through its normal operator $\LGbar$. They quantify the ``smoothness'' of a function $\phi$ in terms of the decay of its coefficients relative to the spectral decay of $\LGbar$. This adaptability ensures that the null space of $\LGbar$, whose elements cannot be identified from the data, is excluded. These spaces arise naturally in nonparametric regression and are generalizations of the source sets in inverse problems (see, e.g., \cite[Eq.(3.29)]{engl1996regularization} and \cite[Eq.(2.5)]{BM18}) and the RKHSs in functional linear regression in \cite{hall2007methodology,yuan2010reproducing,balasubramanian22unified}. 
\begin{definition}[Spectral Sobolev spaces.]\label{def:SSS} 
Assume that the normal operator $\LGbar$ in \eqref{eq:normal_opt} is Hilbert-Schmidt. Denote $\{\psi_k\}_{k=1}^\infty$ the orthonormal eigenfunctions corresponding to the positive eigenvalues $\{\lambda_k\}_{k=1}^\infty$ in descending order. For $\beta\geq 0$ and $L>0$, define the spectral Sobolev space $H^\beta_\rho$ and class $H_\rho^\beta(L)\subset \ker(\LGbar)^\bot \subset L^2_\rho$ as 
\[
H^\beta_\rho :=\left\{\phi = \sum_{k=1}^\infty \theta_k\psi_k: \|\phi\|_{H^\beta_\rho}^2=\sum_{k=1}^\infty \lambda_k^{-\beta}\theta_k^2 <+\infty \right\}, \  H_\rho^\beta(L) := \{\phi\in H^\beta_\rho: \|\phi\|_{H^\beta_\rho}^2 \leq L^2\}. 
\]
\end{definition}

The spectral Sobolev spaces differ from the classical model-agnostic Sobolev spaces (or the H\"older spaces), which are commonly used in nonparametric regression \cite{Gyorfi06a,CuckerSmale02,tsybakov2008introduction}. Classical spaces provide a universal quantification of the smoothness independent of the model and measure $\rho$, making them suitable for problems for classical regression problems that estimate $f(x)$ in the model $Y = R_f(X)+\varepsilon$ with $R_f(x)= f(x)$ from data $\{(X^m, Y^m)\}$. For these problems, the normal operator is the identity operator, whose null space is $\{0\}$. However, classical spaces are not suitable for the learning kernel in operators since they may include nonzero elements of $\ker(\LGbar)$ that cannot be identified from the data. By construction, $H^\beta_\rho$ avoids this issue, offering a tailored alternative to these classical spaces.

The spectral Sobolev space $H^\beta_\rho$ is a generalization of the classical periodic Sobolev space (see, e.g., \cite{wahba1990spline}), as shown in Example \ref{example:int_opt_H} below. Unlike these classical spaces, $H^\beta_\rho$ adapts to the specific spectral properties of $\LGbar$, making it better suited for problems involving compact normal operators. 

\begin{example}[Periodic Sobolev spaces]
\label{example:int_opt_H} 
For {\rm Example \ref{example:int_opt}}, the normal operator has eigenvalues $\lambda_{2k-1}=\lambda_{2k} = \sigma_k^2/4>0 $ and eigenfunctions $\psi_{2k-1}(s)= \sqrt{2}\cos(2\pi k s)$ and $\psi_{2k}(s) = \sqrt{2}\sin(2\pi k s)$ for $k\geq 1$. It has a zero eigenvalue and $\ker{(\LGbar)} = \mathrm{span}\{\mathbf{1}\}$. The adaptive spectral Sobolev space is 
\[
H^\beta_\rho = \left\{\phi(s) =\sqrt{2} \sum_{k=1}^\infty (\theta_{2k-1}\cos(2\pi k s) + \theta_{2k} \sin(2\pi k s)) \in L^2([0,1]): \sum_{k=1}^\infty \lambda_k^{-\beta} \theta_k^2<+\infty \right\}.   
\]
In particular, when $\sigma_k^2=\frac{4}{(2\pi k)^{2}}$, the space $H^\beta_\rho$ with $\beta=1$ is the periodic Sobolev space $W^0_{1,per}$. 
\end{example}

Furthermore, the spectral Sobolev space $H^\beta_\rho$ is closely related to RKHS when $\LGbar$ has an integral kernel $\Gbar$ as in Proposition \ref{prop:compactLG}. In particular, when $\beta=1$, the space $H^1_\rho= \LGbar^{1/2}L^2_\rho$ is the RKHS associated with $\Gbar$. They have been used for regularization in \cite{LangLu23small}.

The space $H^\beta_\rho$ controls the decay of the coefficients of $\phi$, as the next lemma shows.  
\begin{lemma}\label{lemma:SSS}
	If $\phi = \sum_{k=1}^\infty \theta_k\psi_k\in H_\rho^\beta(L)$. Then  $\sum_{k=n}^\infty |\theta_k|^2 \leq L^2 \lambda_n^{\beta} $ for all $n\geq 1$.
	\end{lemma}
\begin{proof} It follows directly from that $\sum_{k=n}^\infty |\theta_k|^2 \leq \lambda_n^{\beta} \sum_{k=n}^\infty \lambda_k^{-\beta}|\theta_k|^2 \leq L^2 \lambda_n^{\beta}\,.$ 
\end{proof}


 \section{Minimax upper rates}
\label{sec:uprate}
In this section, we establish the minimax upper rate by demonstrating it as the upper bound for a tamed least squares estimator (tLSE). The tLSE offers a relatively straightforward proof of the minimax rate, leveraging the left-tail probability of the small eigenvalues of random regression matrices. Originally introduced in \cite{wang2023optimal} for a problem involving a coercive normal operator (i.e., where the eigenvalues of $\LGbar$ have a positive lower bound), its applicability is extended in this section to problems with compact normal operators. This innovation, combined with the results in \cite{wang2023optimal}, highlights tLSE as a versatile and powerful tool to establish minimax upper rates in general nonparametric regression, regardless of whether the normal operator is coercive or non-coercive.
  
\subsection{A tamed projection estimator}\label{sec:tLSE}
We construct the following tamed least squares estimator, which is the minimizer of the quadratic loss function in \eqref{eq:lossFn} over the hypothesis space $\calH_n=\mathrm{span}\{\psi_k\}_{k=1}^n$ when the random regression matrix is suitably well-posed and is zero otherwise.  
 
\begin{definition}[Tamed Least Square Estimator (tLSE)]\label{def:tLSE}
   Let $ \{\psi_k\}_{k=1}^\infty$ be the orthonormal eigenfunctions corresponding to the decaying eigenvalues $\{\lambda_k\}_{k=1}^\infty$ of the normal operator $\LGbar$ in \eqref{eq:normal_opt}.  The tLSE in $\calH_n= \mathrm{span}\{\psi_k\}_{k=1}^n$ is  
    \begin{equation}\label{eq:tLSE}
    \begin{aligned}
      \widehat{\phi}_{n,M} & = \sum_{k=1}^{n} \widehat \theta_k \psi_k \, \text{ with }\widehat \btheta_{n,M} = (\widehat \theta_1,\ldots,\widehat \theta_n)^\top  = \Abar_{n,M} ^{-1}\bbar_{n,M}\mathbf{1}_{\calA}, \quad \\ 
       \calA :&= \begin{cases}
				 \{ \lambda_{\min}(\Abar_{n,M}) > \lambda_n/4\} , & \text{if polynomial decay}; \\
            \{\lambda_{k}(\Abar_{n,M}) > \lambda_k/4,\,\forall k\leq n\}
, & \text{if exponential decay}, 
		 \end{cases}
	\end{aligned}
    \end{equation}
  where the normal matrix $\Abar_{n,M}$ and normal vector $\bbar_{n,M}$ are given by 
  \begin{equation}\label{eq:Ab}
	\begin{aligned}
		\Abar_{n,M}(k,l) = \frac{1}{M}\sum_{m=1}^M\innerp{R_{\psi_k}[u^{m}], R_{\psi_l}[u^{m}]}_\spaceY,  \quad 
		\bbar_{n,M}(k)  = \frac{1}{M}\sum_{m=1}^M\innerp{ f^{m},R_{\psi_k}[u^{m}]}.
	\end{aligned} 
\end{equation} 
\end{definition}

The tLSE is a ``tamed'' version of the classical least squares estimator   
\begin{equation}
	\widehat \phi_{n,M}^{lse} = \sum_{k=1}^{n} \widehat \theta_k^{lse} \psi_k, \quad \text{ with }  (\widehat \theta_1^{lse},\ldots,\widehat \theta_n^{lse})^\top = \Abar_{n,M}^{\dag}\bbar_{n,M},
\end{equation}
where $A^\dag$ denotes the Moore--Penrose inverse satisfying $A^\dag A = A A^\dag  = {\rm{Id}}_{\rm{rank}(A)}$. The LSE is unstable since the empirical normal matrix $\Abar_{n,M}$, whose smallest eigenvalue can be arbitrarily small, is ill-conditioned. In contrast, the tLSE is stable by using the LSE only when the eigenvalues of $\Abar_{n,M}$ are not too small, and it is zero otherwise.  

Additionally, the tLSE differs from the classical truncated SVD estimator (see, e.g., \cite{hansen1987_TruncatedSVD}), which stabilizes the inversion of $A_{n,M}$ by retaining only those singular values above a fixed threshold. By contrast, the tLSE is zero when an eigenvalue falls below the threshold. Although this hard‐thresholding rule is rarely optimal in practice with finitely many samples, it provides a remarkably tractable estimator for establishing sharp minimax-rate bounds.

 The next lemma shows that in the large sample limit, the tamed LSE recovers the LSE, equivalently, the projection of the true function in the hypothesis space $\calH_n$. It follows directly from the strong Law of Large Numbers, so we omit its proof. 
\begin{lemma}\label{Lem:projEst} Under {\rm Assumption \ref{assum:model}} on the model and {\rm Assumption \ref{assump:noise-X-general} \ref{assum:noise_up}} on the noise, let $\{(\lambda_k,\psi_k)\}_{k=1}^\infty$ be the eigen-pairs of the normal operator $\LGbar$ with $ \{\psi_k\}_{k=1}^n$ being orthonormal, and consider the normal matrice ${\Abar}_{n,M}$ and vector ${\bbar}_{n,M}$ in \eqref{eq:Ab}. Then, the limits ${\Abar}_{n,\infty}(k,l)=\lim_{M\to \infty} {\Abar}_{n,M}(k,l)$ and $\bbar_{n,\infty}(k)=\lim_{M\to \infty} {\bbar}_{n,M}(k)$ exist and satisfy
\begin{equation}{\label{Eq:AbEE}}
\begin{aligned}
	{\Abar}_{n,\infty}(k,l) &=  \E[\innerp{R_{\psi_k}[u],R_{\psi_l}[u]}_{\spaceY}] =\innerp{\LGbar\psi_k,\psi_l}_{L^2_\rho} = \lambda_k\delta_{kl} \,,\quad \forall\ 1\leq k,l\leq n\,;   \\
 {\bbar}_{n,\infty}(k) & =  \E[\innerp{f,R_{\psi_k}[u]}]=\innerp{\LGbar\phi,\psi_k}_{L^2_\rho} = \lambda_k\theta_k^* ,\quad \forall\ 1\leq k \leq n \,, 
\end{aligned}
\end{equation}
where $\theta_k^*$ are the coefficients of the true kernel $\phi_*=\sum_{k=1}^\infty \theta_k^* \psi_k$. Consequently,  
\begin{equation}\label{Eq:Asyp_theta}
	\btheta_n^*:=(\theta_1^*,\theta_2^*,\cdots,\theta_n^*)^\top=\Abar_{n,\infty}^{-1}\bbar_{n,\infty}. 
\end{equation}
\end{lemma}

 \subsection{Minimax upper rates}
 We prove next the minimax  upper rates under a condition concerning the fourth moment of $R_\phi[u]$. It constrains the distribution of $u$ and the forward operator $R_\phi$. 
 \begin{assumption}[Fourth-moment condition]\label{assum:4thmom} 
 There exists $\kappa\geq 1$ such that, 
        \begin{equation}\label{ineq:4thmom}
            \frac{\E[\|R_\phi[u]\|_\spaceY^4]}{(\E[\|R_\phi[u]\|_\spaceY^2])^2} = \frac{\E[\|R_\phi[u]\|_\spaceY^4]}{\innerp{\LGbar \phi,\phi}^2_{L^2_{\rho}}}\leq \kappa, \quad \forall \phi\in H_\rho^\beta. 
        \end{equation}
\end{assumption}

The fourth-moment condition holds for Gaussian processes $u$ and linear operators $R_\phi$ such as the integral operator \eqref{eg:integral} and the nonlocal operator \eqref{eg:nonlocal}. In fact, when $u$ is a centered Gaussian process and $R_\phi$ is linear in $u$,  for each $\phi\in L_\rho^2$ and $x\in \calX$, the random variable $R_\phi[u](x)$ is  centered Gaussian and $\E[R^4_\phi[u](x)]= 3 \E[R^2_\phi[u](x)]^2$. Then, by Cauchy-Schwartz inequality, 
$$\Ebracket{ R^2_\phi[u](x)R^2_\phi[u](y)}\leq  \E[R^4_\phi[u](x)]^{1/2}\E[R^4_\phi[u](y)]^{1/2} = 3 \Ebracket{ R^2_\phi[u](x)}\Ebracket{R^2_\phi[u](y)}.$$ Hence, 
  \begin{align*}
        \E[\|R_\phi[u]\|_\spaceY^4] &= \Ebracket{\left(\int_\calX R^2_\phi[u](x)\nu(dx)\right)^2}= \Ebracket{\int_\calX\int_\calX R^2_\phi[u](x)R^2_\phi[u](y)\,\nu(dx)\, \nu(dy)}\\
        &= \int_\calX\int_\calX\Ebracket{ R^2_\phi[u](x)R^2_\phi[u](y)}\,\nu(dx)\, \nu(dy)\\
        &\leq 3 \int_\calX\int_\calX\Ebracket{ R^2_\phi[u](x)}\Ebracket{R^2_\phi[u](y)}\,\nu(dx)\, \nu(dy) = 3 (\E[\|R_\phi[u]\|_\spaceY^2])^2. 
    \end{align*}
That is, the fourth-moment condition holds on $H^\beta_\rho$ for all $\beta\geq 0$ with $\kappa= 3$.

Our main result is the following minimax upper rate. 
\begin{theorem}[Minimax upper rate]\label{thm:L2rho_upper}
Under {\rm Assumptions} {\rm\ref{assum:model}}, {\rm\ref{assump:noise-X-general}}, and {\rm\ref{assum:4thmom}} on the general model \eqref{eq:model_general}, we have the following minimax upper rates for $\beta>0$. 
\begin{itemize}
    \item If the polynomial spectral decay with $r>\frac1 4$ in {\rm Assumption \ref{assump:noise-X-general}} {\rm\ref{assum:polydecay}} is satisfied, we have 
		    \begin{equation}\label{ineq:upper_main}
		\begin{aligned}
			& \limsup_{M\to\infty} \inf_{\widehat \phi\in L^2_\rho}\sup_{\phi_*\in H^\beta_\rho(L)}   \E_{\phi_*} \left[ M^{\frac{2\beta r}{2\beta r+2r+1}} \| \widehat{\phi}-\phi_* \|_{L^2_\rho} ^2 \right] \leq C_{\beta,r,L,a,b,\sigma},
		\end{aligned}
		\end{equation}
	     where $C_{\beta,r,L,a,b,\sigma} =3 \left(\frac{2\sigma^2 }{a}\right)^{\frac{2\beta r}{2\beta r + 2r +1}} (b^\beta L^2)^{\frac{2 r +1}{2\beta r + 2r + 1}} $. 
    \item If the exponential spectral decay with $r>0$ in {\rm Assumption \ref{assump:noise-X-general}} {\rm\ref{assum:expdecay}} is satisfied, we have 
    \begin{equation}\label{ineq:upper_exp}
    \limsup_{M\to\infty} \inf_{\widehat \phi\in L^2_\rho}\sup_{\phi_*\in H^\beta_\rho(L)}   \E_{\phi_*}\left[ M^{\frac{\beta}{\beta+1}} \| \widehat{\phi}-\phi_* \|_{L^2_\rho} ^2 \right] \leq C_{\beta,r,L,a,b,\sigma} 
\end{equation}
with $ C_{\beta,r,L,a,b,\sigma} = 2 \left( \frac{4\sigma^2 e^r}{a(e^r-1)} \right)^{\frac{\beta}{\beta+1} }(b^\beta L^2)^{\frac{1}{\beta+1} } $. 
\end{itemize}
 Here, $\E_{\phi_*}$ is the expectation with respect to the dataset $\{(u^{m},f^m)\}_{m=1}^M$ generated from {\rm Model} \eqref{eq:model_general} with kernel $\phi_*$, and $H^\beta_\rho(L)$ is the spectral Sobolev class in {\rm Definition \ref{def:SSS}}.
\end{theorem}

We prove the minimax upper rate by showing that the tLSE defined in \eqref{eq:tLSE} attains the convergence rate. Our analysis implements the classical bias-variance trade-off framework in three main steps:
\begin{itemize}
	\item \emph{Error Decomposition}. We split the estimation error into a bias (approximation error) term and a variance term. The bias decays as the projection dimension $n$ increases, at a rate $O(\lambda_n^{\beta})$ dictated by the function space $H^\beta_\rho(L)$, which is analogous to the $O(n^{-2\beta})$ rate in classical regression.   
    \item \emph{Variance Control.} We show that the variance term is bounded by $O(n^{1+2r}/M)$ for polynomial spectral decay and $O(e^{rn}/M)$ for exponential spectral decay, paralleling the \(O(n/M)\) rate in classical regression. This variance is further decomposed into a sampling error component and a negligible term arising from the event \(\mathcal{A}^c\) (the cutoff event for small eigenvalues), with each part controlled by Lemma \ref{lemma:samplingError} and Lemma \ref{lemma:prob_cutoff}, respectively.
    \item \emph{Optimal Dimension Selection.} Finally, we select the projection dimension $n_M$ to balance the bias and variance, thereby achieving the optimal rate. 
\end{itemize}
Our technical innovations relative to \cite{wang2023optimal} are twofold. First, instead of the standard approach that bounds the variance via the operator norm of the normal matrix (which leads to a suboptimal estimate), we derive a tight bound for the sampling error using a singular value decomposition. Second, we obtain two refined left-tail probability bounds for the eigenvalues of the normal matrix by leveraging its trace, which allows us to control the cutoff probability $\P(\mathcal{A}^c)$ using only a fourth-moment condition, without imposing additional boundedness on the basis functions. We state these results in Lemma \ref{lemma:samplingError} and \ref{lemma:prob_cutoff} below and postpone their proofs to the Appendix.

\begin{lemma}[Sampling error]\label{lemma:samplingError}
Under {\rm Assumptions} {\rm\ref{assum:model}}, {\rm\ref{assump:noise-X-general}}, and {\rm\ref{assum:4thmom}} on the general model \eqref{eq:model_general}, conditional on the event $\calA$, the sampling error of the tLSE in \eqref{eq:tLSE} satisfies 
\begin{align*}
    \E[\|\Abar_{n,M}^{-1}\bbar_{n,M}- \btheta_{n}^*\|^2\mathbf{1}_\mathcal{A}] &\leq \frac{16\kappa L^2}{ M}\lambda_n^{\beta-1}\sum_{k=1}^n\lambda_k +  \begin{cases}
 	\frac{4\sigma^2}{M} \frac{n}{\lambda_n},  & \text{if polynomial decay;} \\
 	\frac{4\sigma^2}{M} \sum_{k=1}^n\lambda_k^{-1}, & \text{if exponential decay,}
 \end{cases}
\end{align*}
	where $\Abar_{n,M}$ and $\bbar_{n,M}$ are defined in \eqref{eq:Ab} and $\btheta_n^*$ in \eqref{Eq:Asyp_theta}. 
\end{lemma}

\begin{lemma}[Probability of cutoff]\label{lemma:prob_cutoff}
Under {\rm Assumptions} {\rm\ref{assum:model}} and {\rm\ref{assum:4thmom}} on the general model \eqref{eq:model_general}, the probability of cutoff $\P(\calA^c)$, under polynomial or exponential spectral decay, is controlled by the following left-tail probability bounds: 
\begin{align}\label{ineq:LeftProb4}
{\P(\calA^c)\leq } 
\begin{cases} 
 	\frac{\kappa\lambda_1^2}{M} + \exp\left(n\log\left(\frac{20(\lambda_1+1)}{\lambda_n}\right) - \frac{ M\lambda_n}{4\kappa(\lambda_n+\lambda_1)}\right)\text{; or}\\ 
 	\frac{n\kappa\lambda_1^2}{M} + \sum_{k=1}^n\exp\left(k \log\left(\frac{20(\lambda_1+1)}{\lambda_k}\right) - \frac{ M\lambda_k}{4\kappa(\lambda_k+\lambda_1)}\right), \text{ respectively.}
 \end{cases}	
\end{align}
\end{lemma}

\begin{proof}[Proof of Theorem \ref{thm:L2rho_upper}]
It suffices to prove that tLSE defined in \eqref{eq:tLSE}  converges at the minimax upper rate. We start from the variance-bias decomposition: 
\begin{align*}
	\E_{\phi_*}[ \| \widehat{\phi}_{n,M}-\phi_* \|_{L^2_\rho} ^2 ] &= \E[ \| \widehat{\phi}_{n,M}-\phi^*_{\calH_n} \|_{L^2_\rho} ^2 ] + \|\phi^*_{\calH_n^\bot}\|_{L^2_\rho}^2 
	= \E[\|\widehat \btheta_{n,M}- \btheta_{n}^*\|^2] + \|\phi^*_{\calH_n^\bot}\|_{L^2_\rho}^2\,, 
\end{align*}
where $\phi^*_{\calH_n}=\sum_{k=1}^n\theta_k^*\psi_k$ and $\phi^*_{\calH_n^\bot}=\sum_{k=n+1}^\infty\theta_k^*\psi_k$ are projections of the true kernel $\phi_*= \sum_{k=1}^\infty \theta_k^* \psi_k$ on $\calH_n$ and its orthogonal complement $\calH_n^{\bot}$, respectively, and $\btheta_{n}^*=(\theta_1^*,\ldots, \theta_n^*)^\top$. 

The bias term (i.e., the 2nd term) is bounded above by the smoothness of the true kernel in $H_\rho^\beta(L)$. That is, by Lemma \ref{lemma:SSS}, we have 
\begin{equation*}
	\|\phi^*_{\calH_n^\bot}\|_{L^2_\rho}^2=\sum_{k={n}+1}^\infty |\theta_k^*|^2\leq L^2 \lambda_{n+1}^\beta\leq L^2 \lambda_{n}^\beta\,. 
\end{equation*}
Next, we split the variance term into two parts: 
\begin{align*}
        \E[\|\widehat \btheta_{n,M}- \btheta_{n}^*\|^2]&=\E[\|\Abar_{n,M}^{-1}\bbar_{n,M}- \btheta_{n}^*\|^2\mathbf{1}_\mathcal{A}] + \P(\mathcal{A}^c)\|\btheta_{n}^*\|^2\\
        &\leq \E[\|\Abar_{n,M}^{-1}\bbar_{n,M}- \btheta_{n}^*\|^2\mathbf{1}_\mathcal{A}] + \P(\mathcal{A}^c)\lambda_1^\beta L^2.  
    \end{align*}
    Here, the first term comes from the sampling error, the second term is the probability of cutoff error, and they are bounded by Lemma \ref{lemma:samplingError} and Lemma \ref{lemma:prob_cutoff}. 

Lastly, we select $n$ adaptive to $M$ according to the spectral decay. 

\vspace{2mm}\noindent\textbf{Polynomial spectral decay.} Consider first the case where $a n^{-2r}\leq \lambda_n\leq bn^{-2r}$ and $\mathcal{A}=\{\lambda_{\min} (\Abar_{n,M})>\lambda_n/4 \}$. Note that $\lambda_n^{\beta-1}\sum_{k=1}^n\lambda_k\leq n\lambda_n^{-1} (\lambda_1\lambda_n^\beta)= n\lambda_n^{-1}o(1) $ since $\lambda_n\asymp n^{-2r}$. Lemma \ref{lemma:samplingError} implies 
\begin{align*}
    \E[\|\Abar_{n,M}^{-1}\bbar_{n,M}- \btheta_{n}^*\|^2\mathbf{1}_\mathcal{A}] &\leq  \frac{4\sigma^2n}{M\lambda_n} + \frac{16\kappa L^2}{ M}\lambda_n^{\beta-1}\sum_{k=1}^n\lambda_k \\
    & \leq (4\sigma^2+o(1)) \frac{n}{M\lambda_n}    \leq  \left(\frac{4\sigma^2}{a}+o(1)\right)\frac{n^{1+2r}}{M}. 
\end{align*}

Recall that the bias term satisfies $\|\phi^*_{\calH_n^\bot}\|_{L^2_\rho}^2 \leq L^2\lambda_n^\beta\leq b^\beta L^2 n^{-2\beta r}$, we select the optimal $n$ by minimizing the trade-off function $
g(n):= \frac{4\sigma^2 n^{1+2r}}{aM}+ \frac{b^\beta L^2}{n^{2\beta r}}, 
$ 
which gives 
$ n_M=\left\lceil\left( \frac{ab^\beta \beta r L^2 }{2\sigma^2(1+2r)} M\right)^{\frac{1}{2\beta r+2r+1}}\right\rceil$. 
 
 Then, we get 
\begin{align*}
		&      \E[\|\Abar_{n,M}^{-1}\bbar_{n,M}- \btheta_{n}^*\|^2\mathbf{1}_\mathcal{A}] + \|\phi^*_{\calH_n^\bot}\|_{L^2_\rho}^2
\leq  \left(\frac{4\sigma^2}{a}+o(1)\right)\frac{n_M^{1+2r}}{M}+b^\beta L^2   n^{-2\beta r}_{M} \\
	\leq\, &  \left[ \frac{4\sigma^2}{a}\left( \frac{ab^\beta \beta r L^2 }{2\sigma^2(1+2r)} \right)^{\frac{1+2r}{2\beta r+2r+1}}  + b^\beta L^2 \left( \frac{ab^\beta \beta r L^2 }{2\sigma^2(1+2r)} \right)^{\frac{-2\beta r}{2\beta r+2r+1}} +o(1)\right]  M^{-\frac{2\beta r}{2\beta r+2r+1}} \\
     =\, & \left[\left(\frac{2\sigma^2}{a}\right)^{\frac{2\beta r}{2\beta r + 2r +1}} (b^\beta L^2)^{\frac{2 r +1}{2\beta r + 2r + 1}} h\left(\frac{\beta r}{1+2r}\right)+o(1)\right] M^{-\frac{2\beta r}{2\beta r+2r+1}} \\
     \leq \, &   (C_{\beta,r,L,a,b,\sigma}+o(1)) M^{-\frac{2\beta r}{2\beta r+2r+1}}, 
\end{align*}
where $h(x) = 2 x^{\frac{1}{2x+1}} + x^{-\frac{2x}{2x+1}}$ and $C_{\beta,r,L,a,b,\sigma} =3 \left(\frac{2\sigma^2}{a}\right)^{\frac{2\beta r}{2\beta r + 2r +1}} (b^\beta L^2)^{\frac{2 r +1}{2\beta r + 2r + 1}} $ by the fact that $\sup_{x>0} h(x) =  h(1)=3$ since 
$
h'(x) = -\frac{2\log x}{(2x+1)^2}  h(x)\, \begin{cases}
	>0, & \text{ if } 0<x<1; \\
	<0, & \text{ if } x>1. 
\end{cases}
$

 Meanwhile,  the probability of cutoff $\P(\calA^c)$ with $n=n_M$ is negligible compared to the above rate of $ M^{-\frac{2\beta r}{2\beta r+2r+1}}$. In fact, Lemma \ref{lemma:prob_cutoff} shows that 
 \[\P(\calA^c ) \leq  	\frac{\kappa\lambda_1^2}{M} + \exp\left(n\log\left(\frac{20(\lambda_1+1)}{\lambda_n}\right) - \frac{ M\lambda_n}{4\kappa(\lambda_n+\lambda_1)}\right)
 \]
 for all $n$. This left-tail probability is of order $O(\frac{1}{M})$. The exponential term with $n=n_M$ decays faster than $\frac{1}{M}$ since its two exponent terms are 
$\frac{M\lambda_{n_M}}{4\kappa(\lambda_{n_M}+\lambda_1)} \geq \frac{a M n_M^{-2r}}{8\kappa\lambda_1} \asymp M^{\frac{2\beta r + 1}{2\beta r + 2r +1}}
$ 
and $
n_M\log\left(\frac{20(\lambda_1+1)}{\lambda_{n_M}}\right) \leq n_M \log\left(\frac{20(\lambda_1+1)}{a n_M^{-2r}}\right) = O\left(M^{\frac{1}{2\beta r + 2r +1}}\log M\right)
$. 

As a result, $\limsup_{M\to\infty} \sup_{\phi_*\in H^\beta_\rho(L)}   \E_{\phi_*}[ M^{\frac{2\beta r}{2\beta r+2r+1}} \| \widehat{\phi}_{n_M}-\phi_{*} \|_{L^2_\rho} ^2 ]\leq  C_{\beta,r,L,a,b,\sigma}$ 
and the upper bound in \eqref{ineq:upper_main} follows.

\vspace{2mm}\noindent\textbf{Exponential spectral decay.} Consider next the case where $a\exp(-rn)\leq\lambda_n\leq b\exp(-rn)$ and $\mathcal{A} = \{\lambda_k(\Abar_{n,M})>\lambda_k/4,\forall k\leq n\}$. Note that $\lambda_n^{\beta}\sum_{k=1}^n\lambda_k\leq  \lambda_n^{\beta} b \sum_{k=1}^n e^{-rk}\leq \lambda_n^{\beta}\, \frac{b e^{-r}}{1-e^{-r}}$ converges to $0$ as $n\to \infty$, and $4\sigma^2 a^{-1} \sum_{k=1}^n e^{rk}= 4\sigma^2 a^{-1}\,\frac{e^{r(n+1)}-e^r}{e^r-1}< 4\sigma^2 a^{-1}\,\frac{e^{r(n+1)}}{e^r-1}=:c_1 e^{rn}$. Then, Lemma \ref{lemma:samplingError} implies  
\begin{align*}
    \E[\|\Abar_{n,M}^{-1}\bbar_{n,M}- \btheta_{n}^*\|^2\mathbf{1}_\mathcal{A}] &\leq \frac{16\kappa L^2}{ M}\lambda_n^{\beta-1}\sum_{k=1}^n\lambda_k + \frac{4\sigma^2}{M}\sum_{k=1}^n\lambda_k^{-1} 
    \leq  (c_1+o(1))  \frac{e^{rn}}{M}. 
\end{align*} 

Recall that the bias term is bounded by $\|\phi^*_{\calH_n^\bot}\|_{L^2_\rho}^2 \leq L^2\lambda_{n+1}^\beta\leq b^\beta L^2 e^{-\beta r(n+1)}$. Then, we select $n$ by minimizing the trade-off function $g(n):= c_1  \frac{e^{rn}}{M} + b^\beta L^2 e^{-\beta r(n+1)}$. Here, we regard $n$ and $n+1$ as the same variable $x\in[n,n+1]$. 
The solution is 
\[
n_M = \left\lfloor\frac{1}{\beta r + r}\log\left(c_1^{-1} b^\beta \beta L^2 M \right)\right\rfloor =  \frac{\log M} {\beta r+r} + O(1) .
\] 
Then, noting that $e^{ rn_M} \leq (c_1^{-1} b^\beta \beta L^2 M)^{\frac{1}{\beta+1}}  $ and $e^{-\beta r(n_M+1)} \leq (c_1^{-1} b^\beta \beta L^2 M)^{-\frac{\beta}{\beta+1}} $, we have 
\begin{align*}
    &\E[\|\Abar_{n,M}^{-1}\bbar_{n,M}- \btheta_{n}^*\|^2\mathbf{1}_\mathcal{A}] + \|\phi^*_{\calH_n^\bot}\|_{L^2_\rho}^2
    \leq  (c_1+o(1)) \frac{e^{rn_M} }{M} + b^\beta L^2 e^{- \beta r(n_M+1)}  \\
    \leq & \left( c_1 (c_1^{-1} b^\beta \beta L^2 )^{\frac{1}{\beta+1}}  +   b^\beta L^2  (c_1^{-1} b^\beta \beta L^2)^{-\frac{\beta}{\beta+1}} \right) M^{-\frac{\beta}{\beta+1}} \\
     \leq  & c_1^{\frac{\beta}{\beta+1} }(b^\beta L^2)^{\frac{1}{\beta+1} } (\beta^{\frac{1}{\beta+1} } + \beta^{-\frac{\beta}{\beta+1} } )  M^{-\frac{\beta}{\beta+1}}
     \leq C_{\beta,r,L,a,b,\sigma}  M^{-\frac{\beta}{\beta+1}},
\end{align*}
where $ C_{\beta,r,L,a,b,\sigma} = 2 c_1^{\frac{\beta}{\beta+1} }(b^\beta L^2)^{\frac{1}{\beta+1} } $ 
since $\sup_{\beta>0}(\beta^{\frac{1}{\beta+1} } + \beta^{-\frac{\beta}{\beta+1} } ) = 2$. 
The probability of cutoff $\P(\calA^c)$ with $n=n_M$ is negligible compared to the rate $M^{-\frac{\beta}{\beta+1}}$. In fact, by the second part of Lemma \ref{lemma:prob_cutoff},
\begin{align*}
    \P\left( \calA^c\right)&\leq \ \frac{n\kappa\lambda_1^2}{M} \ + \sum_{k=1}^n \exp{\left(k\log\left(\frac{20(\lambda_1+1)}{\lambda_k}\right) - \frac{ M\lambda_k}{4\kappa(\lambda_k+\lambda_1)}\right)}\\
    &\leq \frac{n\kappa\lambda_1^2}{M} \ + n \exp{\left(n\log\left(\frac{20(\lambda_1+1)}{\lambda_n}\right) - \frac{ M\lambda_n}{4\kappa(\lambda_n+\lambda_1)}\right)}.
\end{align*}
When $n=n_M = \frac{\log M}{\beta r+r} + O(1)$, the exponential term decays faster than $\frac{n_M}{M}$ since its positive exponent term, $n_M\log\left(\frac{20(\lambda_1+1)}{\lambda_{n_M}}\right)=O(n_M^2)=O((\log M)^2)$, is less significant than its negative exponent term $\frac{M\lambda_{n_M}}{4\kappa(\lambda_{n_M}+\lambda_1)}\geq \frac{aM\exp(-rn_M)}{8\kappa\lambda_1} \asymp M^{\frac{\beta}{\beta+1}}$. Therefore, $\P\left( \calA^c\right) =  O\left(\frac{n_M}{M}\right) = O\left(\frac{\log M}{M}\right)$ is negligible compared to $M^{-\frac{\beta}{\beta+1}}$. 

As a result, $\limsup_{M\to\infty} \sup_{\phi_*\in H^\beta_\rho(L)}   \E_{\phi_*}[ M^{\frac{\beta}{\beta+1}} \| \widehat{\phi}_{n_M}-\phi_{*} \|_{L^2_\rho} ^2 ]\leq  C_{\beta,r,L,a,b,\sigma}$, which gives the upper bound in \eqref{ineq:upper_exp}.
\end{proof}

\subsection{Probability of cutoff}\label{sec:prob_cutoff}
We prove the two left-tail probability inequalities in Lemma \ref{lemma:prob_cutoff} by the following lemma and its corollary. 
\begin{lemma}[Left-tail probability of the smallest eigenvalue]\label{Lem:LeftProb}
  Under {\rm Assumptions \ref{assum:model}} and {\rm \ref{assum:4thmom}} on the general model \eqref{eq:model_general}, the left-tail probability of the smallest eigenvalue $\widehat{\lambda}_{\min}:=\lambda_{\min}(\Abar_{n,M})$ of the empirical normal matrix $\Abar_{n,M}$ defined in \eqref{eq:Ab} satisfies
    \begin{equation}\label{ineq:LeftProb}
    \begin{aligned}
        \P\left\{ \widehat{\lambda}_{\min}\leq \frac{3-\varepsilon}{8} \lambda_n\right\}\leq& \ \frac{\kappa\lambda_1^2}{M} + \exp\left(n\log\bigg(\frac{20(\lambda_1+1)}{\lambda_n}\bigg) - \frac{\varepsilon M\lambda_n}{4\kappa(\lambda_n+\lambda_1)}\right)
    \end{aligned}
    \end{equation}
    for all $n\geq 1$. In particular, when $\varepsilon = 1$, one has the first inequality in \eqref{ineq:LeftProb4}. 
\end{lemma}

 We prove Lemma \ref{Lem:LeftProb} by the probably approximately correct (PAC) Bayesian inequality
\[
\P\left(\forall \mu,\int_\Theta Z(\theta)\mu(d\theta)\leq\kl{\mu,\pi}+t\right)\geq 1-e^{-t},\ \forall t>0,
\]
given $\E[\exp(Z(\theta))]\leq 1$ for all $\theta\in\Theta$,  where $\Theta$ is taken as the hypersphere $S^{n-1}$ and $\pi$ is the uniform distribution on it. Innovating the methodology developed in \cite{Mourtada2022,wang2023optimal}, we choose
\begin{align*}
    Z(\theta)&:=-\frac{2}{\kappa(\lambda_1+\lambda_n)}\sum_{m=1}^M\|R_{\phi_\theta}[u^m]\|_\spaceY^2+\frac{2\lambda_1 \lambda_n}{\kappa(\lambda_1+\lambda_n)^2}M\\
    &=-\frac{2}{\kappa(\lambda_1+\lambda_n)}\innerp{\Abar_{n,M}\theta,\theta}M+\frac{2\lambda_1 \lambda_n}{\kappa(\lambda_1+\lambda_n)^2}M,
\end{align*}
where $\kappa$ is from the fourth-moment condition \eqref{ineq:4thmom}, and $\phi_\theta = \sum_{k=1}^n \theta_k\psi_k$ for $\theta=(\theta_1,\cdots,\theta_n)\in S^{n-1}.$ By letting the probability $\mu$ run over all uniform distributions $\pi_{v,\gamma}$ on a cap centered at $v\in S^{n-1}$ with radius $\gamma\leq 1/2$, we characterize $\sup_{v\in S^{n-1}}\int_\Theta Z(\theta)\pi_{v,\gamma}(d\theta)$ by $\inf_{v\in S^{n-1}}\innerp{\Abar_{n,M}v,v}$ and the trace $\tr{(\Abar_{n,M})}$. The infimum term corresponds to the smallest eigenvalue, giving a PAC bound for the left-tail probability. The trace term is typically controlled by truncation (see \cite{Mourtada2022}), and we only need a rough bound 
\[
\P\left\{\frac{\tr(\Abar_{n,M})}{n} \geq \lambda_1 + 1\right\} \leq \P\left\{\tr(\Abar_{n,M}) \geq \sum_{k=1}^n\lambda_k + n\right\}\leq  \frac{\kappa  \lambda_1^2}{M}. 
\]
This corresponds to the first term of the left-tail probability \eqref{ineq:LeftProb}. We postpone the detailed proof to the Appendix.

The second left-tail probability inequality in Lemma \ref{lemma:prob_cutoff} considers all the eigenvalues and is used for the case of exponential spectral decay. 
\begin{corollary}[Left-tail probability of all eigenvalues]\label{cor:LeftProb}
     Under {\rm Assumptions \ref{assum:model}} and {\rm \ref{assum:4thmom}} on the general model \eqref{eq:model_general}, the left-tail probability of all eigenvalues $\widehat{\lambda}_k := \lambda_{k}(\Abar_{n,M})$ of the empirical normal matrix $\Abar_{n,M}$ defined in \eqref{eq:Ab} satisfies
    \begin{equation}\label{ineq:LeftProb_2}
    \begin{aligned}
      \P \bigg(\exists k,\ \widehat{\lambda}_k\leq \frac{\lambda_k}{4}\bigg) \leq& \ \frac{n\kappa\lambda_1^2}{M} \ + \sum_{k=1}^n \exp{\left(k\log\bigg(\frac{20(\lambda_1+1)}{\lambda_k}\bigg) - \frac{ M\lambda_k}{4\kappa(\lambda_k+\lambda_1)}\right)}.
    \end{aligned}
    \end{equation}
\end{corollary}

\begin{proof}[Proof of Corollary \ref{cor:LeftProb}]
    The Cauchy interlacing theorem (see, e.g., \cite{Horn_Johnson_2012}) implies
    \begin{equation*}
        \widehat{\lambda}_k=\lambda_k(\Abar_{n,M})\geq\lambda_{\min} (\Abar_{k,M}).
    \end{equation*}
    Therefore, we have
    \begin{align*}
        \P\left\{\lambda_{k}(\Abar_{n,M})\leq \frac{\lambda_k}{4} \right\}&\leq \P\left\{\lambda_{\min}(\Abar_{k,M})\leq \frac{\lambda_k}{4} \right\}
            \leq \frac{\kappa\lambda_1^2}{M} + \exp\left(k\log\left(\frac{20(\lambda_1+1)}{\lambda_k}\right) - \frac{  M\lambda_k}{4\kappa(\lambda_k+\lambda_1)}\right).
    \end{align*}
     Then, \eqref{ineq:LeftProb_2} follows from 
     $   \P\left( \bigcup_{k=1}^n\left\{\lambda_{k}(\Abar_{n,M})\leq \frac{\lambda_k}{4} \right\}\right)\leq \sum_{k=1}^n \P\left\{\lambda_{k}(\Abar_{n,M})\leq \frac{\lambda_k}{4} \right\}
$. 
    \end{proof}


\section{Minimax lower rates}\label{sec:lower}

We show next that the minimax lower rates match the minimax upper rates in Theorem \ref{thm:L2rho_upper}, thus confirming the optimality of the rates. 
\begin{theorem}[Minimax lower rate]\label{thm:L2rho_lower}
Under {\rm Assumption \ref{assum:model}} on the model and {\rm Assumption \ref{assump:noise-X-general}} on the noise, we have the following minimax rates. 
\begin{itemize}
    \item[(i)] If the normal operator has polynomial spectral decay {\rm\ref{assum:polydecay}}, 
    \begin{equation}\label{ineq:lower_poly}
\begin{aligned}
	& \liminf_{M\to\infty} \inf_{\widehat \phi \in L^2_\rho }\sup_{\phi_*\in H^\beta_\rho(L)}   \E_{\phi_*}[ M^{\frac{2\beta r}{2\beta r+2r+1}} \| \widehat{\phi}-\phi_* \|_{L^2_\rho} ^2 ] \geq C_{\beta,r,L,a,b,\tau} 
\end{aligned}
\end{equation}
with $C_{\beta,r,L,a,b,\tau} = 2^{-2\beta r -4}a^\beta (\tau b^{\beta+1})^{-\frac{2\beta r}{2\beta r+2r+1}}L^\frac{4r+2}{2\beta r+2r+1}.$
    \item[(ii)] If the normal operator has exponential spectral decay {\rm\ref{assum:expdecay}}, 
    \begin{equation}\label{ineq:lower_exp}
    \liminf_{M\to\infty} \inf_{\widehat \phi \in L^2_\rho}\sup_{\phi_*\in H^\beta_\rho(L)}   \E_{\phi_*}\left[ M^{\frac{\beta}{\beta+1}} \| \widehat{\phi}-\phi_* \|_{L^2_\rho} ^2 \right] \geq C_{\beta,r,L,a,b,\tau}
\end{equation}
with $C_{\beta,r,L,a,b,\tau}=\frac{1}{16}e^{-\beta r}\left(\frac{a}{b}\right)^{\beta}\tau^{-\frac{\beta}{\beta+1}}L^{\frac{2}{\beta+1}}$.
\end{itemize}
\end{theorem}

Notably, the constant is $\tau = 1/\sigma^2$ when the space $\spaceY$ is finite-dimensional and the noise is standard Gaussian with variance $\sigma^2$. Then, the orders of $\sigma^2$ and $L$ in the above constants $C_{\beta,r,L,a,b,\tau}$ match those in the constants $C_{\beta,r,L,a,b,\sigma}$ in the upper bounds. That is, the constants $C_{\beta,r,L,a,b,\sigma}$ are sharp in the orders of $\sigma$ and $L$.

\subsection{The reduction scheme and innovations}

We establish these minimax lower rates by the Assouad method \cite{assouad1983deux,yu97Assouad} that reduces the estimation to a test of $2^{L_M}$ hypotheses indexed by $L_M$ binary coefficients in the eigenfunction expansion. We summarize it in the following three steps. 

\begin{itemize}
\item \emph{Reduction to average probability of test errors over finite sets.} We first reduce the infimum over all estimators and the supremum over $H^\beta_\rho(L)$ to the infimum and supremum over a finite subset $\Phi_M\subset H^\beta_\rho(L)$, that is, 
    \begin{equation}\label{ineq:lowerb_finite}
    \begin{aligned}
        \inf_{\widehat \phi\in L^2_\rho}\sup_{\phi_*\in H^\beta_\rho(L)}   \E_{\phi_*}[ \| \widehat{\phi}-\phi_* \|_{L^2_\rho} ^2 ]
        & \geq \inf_{\widehat \phi\in L^2_\rho}\sup_{\phi_*\in \Phi_M}   \E_{\phi_*}[ \| \widehat{\phi}-\phi_* \|_{L^2_\rho} ^2 ]  \\
        & \geq \frac{1}{4}  \inf_{\widehat \phi\in \Phi_M}\sup_{\phi_*\in \Phi_M}   \E_{\phi_*}[ \| \widehat{\phi}-\phi_* \|_{L^2_\rho} ^2 ],
        \end{aligned}
    \end{equation}
    where the second inequality follows from Lemma \ref{lem:lowerb_finite_est} below. Recall the positive eigenvalues $\{\lambda_k\}_{k=1}^\infty$ of the normal operator $\LGbar$ in \eqref{eq:normal_opt} and their orthonormal eigenfunctions $\{\psi_k\}_{k=1}^\infty$, the set $\Phi_M$ is given by
    \begin{equation}\label{def:test_fct}
        \Phi_M:= \left\{ \sum_{k=n_M}^{n_M+L_M-1}  \theta_k \psi_k:\theta_k\in \left\{0,L_M^{-1/2}L\lambda_k^{\beta/2}\right\}\right\},
    \end{equation}
    where $n_M$ and $L_M$ are to be determined according to $M$ and the spectral decay. 
Then, writing $\widehat{\phi} = \sum_{k=n_M}^{n_M+L_M-1}  \widehat{\theta}_k \psi_k$ and $\phi_* =  \sum_{k=n_M}^{n_M+L_M-1}  \theta_k^* \psi_k$, we bound the supreme of the expectations from below by the average test error (see Section \ref{sec:proofs-lower} for its proof):  
    \begin{equation}\label{ineq:infsup_lowerb}
        \begin{aligned}
           &  \inf_{\widehat \phi\in  \Phi_M}\sup_{\phi_*\in \Phi_M}   \E_{\phi_*}[  \| \widehat{\phi}-\phi_* \|_{L^2_\rho} ^2 ]  =\inf_{\widehat \phi\in  \Phi_M} \sup_{\phi_*\in \Phi_M} \sum_{k=n_M}^{n_M+L_M-1} \E_{\phi_*}\left[\left|\widehat{\theta}_k-\theta_k^*\right|^2\right]\\
            \geq & \, \, \big(L_M^{-1} L^2  \sum_{k=n_M}^{n_M+L_M-1} \lambda_k^{\beta }\big) 2^{-L_M}  \inf_{\widehat \phi\in \Phi_M}\min_{k} \sum_{\phi_*\in \Phi_M}\P_{\phi_*}\big(\widehat{\theta}_k\ne\theta_k^*\big). 
        \end{aligned}
    \end{equation}
   
   \item \emph{Lower bound for the average probability of test errors over $\Phi_M$.} We write the sum over all probability of test errors in terms of the total variational distance, which we control by the Kullback--Leibler divergence between restricted measures (see Lemma \ref{lem:tvbd}). By doing so, we obtain a lower bound for the average probability of test errors as in Lemma \ref{lem:testerror}.  
    \item \emph{Selection of $n_M$ and $L_M$ to achieve the optimal rates.} We first select $n_M$ and $L_M$ such that the average test error is bounded from below for all $M$, then verify the minimax lower rate: 
    \begin{equation}\label{ineq:lowerbrate}
         L_M^{-1} L^2\sum_{k=n_M}^{n_M+L_M-1} \lambda_k^{\beta } \asymp 
         \begin{cases}
         	M^{-\frac{2\beta r}{2\beta r + 2r + 1}}\,, & \text{polynomial decay}; \\
         	M^{-\frac{\beta}{\beta+1}}\,, & \text{exponential decay}.
         \end{cases}
    \end{equation}
\end{itemize}
   
      The next two lemmas are the key steps in the scheme, and we postpone their proofs to Section \ref{sec:proofs-lower}.  
\begin{lemma}\label{lem:lowerb_finite_est}
Let $\Phi_M$ be the finite set of functions defined in \eqref{def:test_fct}. Then, 
    \begin{equation}\label{ineq:lowerb_finite_est}
        \inf_{\widehat \phi\in L^2_{\rho}}\, \sup_{\phi_*\in \Phi_M}   \E_{\phi_*}[ \| \widehat{\phi}-\phi_* \|_{L^2_\rho} ^2 ]\geq \frac{1}{4}  \inf_{\widehat \phi \in  \Phi_M}\sup_{\phi_*\in \Phi_M}   \E_{\phi_*}[ \| \widehat{\phi}-\phi_* \|_{L^2_\rho} ^2 ],
    \end{equation}
\end{lemma}

\begin{lemma}\label{lem:testerror}
Let $\P_{\phi_*}$ be the measure of samples $\{(u^m,f^m)\}_{m=1}^M$ from {\rm Model} \eqref{eq:model_general} with kernel $\phi_*$ and with noise satisfying {\rm Assumption \ref{assump:noise-X-general}}. Let $\Phi_M$ be the set in \eqref{def:test_fct}. Then, we have 
        \begin{equation}\label{eq:lb_avg_ProbErr}
            2^{-L_M}\inf_{\widehat \phi\in \Phi_M} \min_{k}  \sum_{\phi_* \in \Phi_M}\P_{\phi_*}\left(\widehat{\theta}_k\ne\theta_k^*\right)\geq 2^{-1}\left(1-\frac{1}{2} \sqrt{\tau M L^2 L_M^{-1} \lambda_{n_M}^{\beta + 1}}\right).
        \end{equation}
    \end{lemma}

The main idea of the scheme is to reduce the infimum over all estimators and the supremum over all functions in $H^\beta_\rho(L)$ to a finite set $\Phi_M$ consisting of functions with binary coefficients in the eigenfunction expansion, and reduce the expectations to the probability of hypothesis test errors. Then, similar to Assouad \cite{assouad1983deux} and Le Cam \cite{lecam1973convergence}, we use the Neyman-Pearson lemma to control the average probability of test errors $  2^{-L_M}\sum_{\phi_*\in \Phi_M}\P_{\phi_*}\left(\widehat{\theta}_k\ne\theta_k^*\right)$ through the total variation distance, which is bounded by the Kullback--Leibler divergence by the Pinsker's inequality $ d_{\mathrm{tv}}\left(\P_{\phi},\P_{\psi}\right)\leq \sqrt{\frac{1}{2}\kl{\P_{\phi},\P_{\psi}}} $. This approach is similar to the Fano method (see, e.g., \cite[Theorem 2.5]{tsybakov2008introduction} and \cite[Chapter 15]{Wainwright2019}), and we refer to \cite{yu97Assouad} for a comparison of these methods. It has been used to establish minimax rates in \cite{yuan2010reproducing,hall2007methodology} for functional linear regression and in \cite{BM18,Helin2024} for statistical inverse problems.

Our main innovation from \cite{hall2007methodology,yuan2010reproducing} is extending the results from scalar-valued output (i.e., $\spaceY=\R$) to separable Hilbert space-valued output. The main difficulty lies in controlling the distance between $\P_\phi$ and $\P_{\psi}$ when $\spaceY$ is infinite-dimensional, because to define the Radon–Nikodym derivative $\frac{d\P_\phi}{d\P_\psi}$ in the KL divergence, one needs additional conditions on the noise and tools in infinite-dimensional analysis, e.g., the Cameron-Martin space \cite{Nua06,da2006introduction,Hu2016analysis} when noise is induced by Gaussian measure. We overcome the difficulty by using their filtered approximations $\P_\phi\big|_{\calF_N}$ and $\P_\psi\big|_{\calF_N}$ over the filtration $ \calF_N := \sigma\left(\left\{u^m\right\}_{m=1}^M,\left\{\innerp{\varepsilon^m,y_1},\cdots\innerp{\varepsilon^m,y_N}\right\}_{m=1}^M \right) $ for $N\geq 1$. Their Radon-Nikodym derivatives are on finite-dimensional subspaces and their KL divergence can be controlled through the conditions on noise in {\rm Assumption \ref{assump:noise-X-general}}; see {\rm Section \ref{sec:tv_meas_inftyD}}. 
    
\begin{proof}[Proof of Theorem \ref{thm:L2rho_lower}]
First, combining \eqref{ineq:lowerb_finite} (which follows from Lemma \ref{lem:lowerb_finite_est}),  \eqref{ineq:infsup_lowerb}, and Lemma \ref{lem:testerror}, we obtain 
        \begin{align*}
                   & \inf_{\widehat \phi\in L^2_\rho}\sup_{\phi_*\in H^\beta_\rho(L)}   \E_{\phi_*}[ \| \widehat{\phi}-\phi_* \|_{L^2_\rho} ^2 ] 
            \geq  \frac{1}{4}\inf_{\widehat{\phi}\in\Phi_M}\sup_{\phi_*\in \Phi_M}   \E_{\phi_*}[ \| \widehat{\phi}-\phi_* \|_{L^2_\rho} ^2 ] \\
            = & \frac{1}{4} \inf_{\widehat \phi\in  \Phi_M} \sup_{\phi_*\in \Phi_M} \sum_{k=n_M}^{n_M+L_M-1} \E_{\phi_*}\left[\left|\widehat{\theta}_k-\theta_k^*\right|^2\right]\\
            \geq & \frac{1}{4} \big(L_M^{-1} L^2  \sum_{k=n_M}^{n_M+L_M-1} \lambda_k^{\beta }\big) 2^{-L_M}  \inf_{\widehat \phi\in \Phi_M}\min_{k} \sum_{\phi_*\in \Phi_M}\P_{\phi_*}\big(\widehat{\theta}_k\ne\theta_k^*\big) \\
             \geq & 2^{-3}\big(L_M^{-1} L^2  \sum_{k=n_M}^{n_M+L_M-1} \lambda_k^{\beta }\big)  \left(1-\frac{1}{2} \sqrt{\tau M L^2 L_M^{-1} \lambda_{n_M}^{\beta + 1}}\right). 
            \end{align*}
            
Next, we appropriately choose $L_M$ and $n_M$ based on the exponential or polynomial spectral decay, ensuring that $ \left(1-\frac{1}{2} \sqrt{\tau M L^2 L_M^{-1} \lambda_{n_M}^{\beta + 1}}\right)$ is bounded from below. Additionally, these choices guarantee that the term $\big(L_M^{-1} L^2  \sum_{k=n_M}^{n_M+L_M-1} \lambda_k^{\beta }\big)$ in \eqref{ineq:lowerbrate} gives the desired rates specified in \eqref{ineq:lower_poly} and \eqref{ineq:lower_exp}.
 
 \vspace{2mm}\noindent\textbf{Exponential spectral decay.} When the spectrum decays exponentially, i.e., $ a\exp(-rk)\leq \lambda_k\leq b\exp(-rk)$ for all $k\geq 1$, we take 
  $  L_M=1, \,   n_M=\left\lceil\frac{\log\left(\tau M L^2 b^{\beta + 1}\right)}{(\beta+1)r}\right\rceil\,
   $ to obtain 
    \begin{align*}
       \left(1-\frac{1}{2} \sqrt{\tau M L^2 L_M^{-1} \lambda_{n_M}^{\beta + 1}}\right) & \geq  1 - \frac{1}{2}\sqrt{\tau M L^2 b^{\beta + 1}\exp(-(\beta+1)rn_M)} \geq 1-\frac{1}{2}=\frac{1}{2};  \\
       L_M^{-1} L^2  \sum_{k=n_M}^{n_M+L_M-1} \lambda_k^{\beta }
                  &  = L^2 \lambda_{n_M}^\beta \geq L^2 a^\beta e^{-\beta rn_M}. 
    \end{align*}
Thus, noting that $ e^{-\beta rn_M} \geq e^{-\beta r} e^{-\beta r\frac{\log\left(\tau M L^2 b^{\beta + 1}\right)}{(\beta+1)r} } = e^{-\beta r} \left(\tau M L^2 b^{\beta + 1}\right)^{-\frac{\beta}{\beta+1}}$, we have 
        \begin{align*}
            \inf_{\widehat \phi\in L^2_\rho}\sup_{\phi_*\in H^\beta_\rho(L)}   \E_{\phi_*}[ \| \widehat{\phi}-\phi_* \|_{L^2_\rho} ^2 ]\geq \frac{\lambda_{n_M}^\beta L^2}{16} \geq \frac{a^\beta L^2}{16}e^{-\beta rn_M} 
            \geq C_{\beta,r,L,a,b,\tau}
            M^{-\frac{\beta}{\beta+1}}.
        \end{align*}

\noindent\textbf{Polynomial spectral decay.} When the spectrum decays polynomially, i.e., $ak^{-2r}\leq \lambda_k\leq bk^{-2r}$ for all $k\geq 1$, we take $L_M=n_M = \left\lceil\left(\tau M L^2 b^{\beta + 1}\right)^{\frac{1}{2\beta r + 2r + 1}}\right\rceil $ to obtain 
\[ 
\tau M L^2 L_M^{-1} \lambda_{n_M}^{\beta + 1} \leq \tau M L^2 n_M^{-1} b^{\beta + 1} n_M^{-2\beta r - 2r} = \tau M L^2b^{\beta + 1} n_M^{-2\beta r - 2r -1} \leq 1, 
\]
 so that $ \left(1-\frac{1}{2} \sqrt{\tau M L^2 L_M^{-1} \lambda_{n_M}^{\beta + 1}}\right)  \geq  1 - \frac{1}{2}=\frac{1}{2}$ and 
            \begin{align*}
       \big(L_M^{-1} L^2  \sum_{k=n_M}^{n_M+L_M-1} \lambda_k^{\beta }\big)
                  &  \geq L^2 \left(a(n_M+L_M-1)^{-2r}\right)^\beta \geq a^\beta L^2 (2n_M)^{-2\beta r}. 
    \end{align*}
  Then, noting that $(2n_M)^{-2\beta r} = \left(2^{-2\beta r} \left(\tau  L^2 b^{\beta + 1}\right)^{- \frac{2\beta r}{2\beta r + 2r + 1}}+o(1)\right)M^{- \frac{2\beta r}{2\beta r + 2r + 1}}$, we have 
          \begin{align*}
            \inf_{\widehat \phi\in L^2_\rho}\sup_{\phi_*\in H^\beta_\rho(L)}   \E_{\phi_*}[ \| \widehat{\phi}-\phi_* \|_{L^2_\rho} ^2 ]             \geq \frac{2^{-2\beta r}}{16} a^\beta L^2  \left(\tau L^2 b^{\beta + 1}\right)^{- \frac{2\beta r}{2\beta r + 2r + 1}} M^{- \frac{2\beta r}{2\beta r + 2r + 1} }.
        \end{align*}  
Thus, we have obtained the rates in \eqref{ineq:lower_exp} and \eqref{ineq:lower_poly}, respectively. 
 \end{proof}

\subsection{A bound for total variation distance}\label{sec:tv_meas_inftyD}
To prove Lemma \ref{lem:testerror}, we will use the Neyman-Pearson lemma to control the average probability of test errors by the total variation distances. Then we use the Pinsker's inequality to control the total variation distances by the Kullback--Leibler divergence, i.e., $ d_{\mathrm{tv}}\left(\P_{\phi},\P_{\psi}\right)\leq \sqrt{\frac{1}{2}\kl{\P_{\phi},\P_{\psi}}} $, which uses the Radon-Nikodym derivative $\frac{d\P_\phi}{d\P_\psi}$.

 However, a major difficulty arises when $\spaceY$ is infinite-dimensional, as the Radon-Nikodym derivative $\frac{d\P_\phi}{d\P_\psi}$ is difficult to compute from the conditions on the noise in {\rm Assumption \ref{assump:noise-X-general}}, which only provides a bound for the KL divergence between finite-dimensional marginal distributions. An exception is when the noise is an isonormal Gaussian process, for which $\frac{d\P_\phi}{d\P_\psi}$ is given by the Cameron-Martin formula; see, e.g., \cite{da2006introduction,Hu2016analysis}.  
    
We solve this issue by considering restricted measures $\left\{\P_\phi\big |_{\calF_N}\right\}_{N= 1}^\infty$ on a filtration $\{\calF_N\}_{N=1}^\infty$ generated by finite-dimensional projections, which connect the measure $\P_\phi$ with the condition on the noise in {\rm Assumption \ref{assump:noise-X-general}}, along with a lemma that characterizes the total variation between $\P_\phi$ and $\P_\psi$ as the limit of the restricted measures. As a result, we can bound the total variation by the limit of the KL divergence between the restricted measures, which we state in the next lemma and postpone its proof to Appendix \ref{sec_append:tv_bd}. 

    \begin{lemma}\label{lem:tvbd}
   Let $\P_\phi, \P_\psi$ be the probability measures induced by samples $\{(u^m,f^m)\}_{m=1}^M$ from {\rm Model} \eqref{eq:model_general} with kernels $\phi$ and $\psi$, respectively.   
 Under {\rm Assumption \ref{assump:noise-X-general}} on the noise, let $\P_{\phi, N}:= \P_\phi\big |_{\calF_N}$ and $\P_{\psi, N}:= \P_\psi\big |_{\calF_N}$ be the restricted measures on the filtration $$   
   \calF_N := \sigma\big(\big\{u^m,  (\innerp{\varepsilon^m,y_i})_{1\leq i\leq N}  \big\}_{m=1}^M 
   \big) 
  $$ for all $N\geq 1$. 
Then, we have 
 \begin{equation}\label{eq:KLbd0}
 	\kl{\P_{\phi,N},\P_{\psi,N}} \leq \frac{\tau M}{2} \E\left[ \left\|R_{\psi-\phi}[u]\right\|_\spaceY^2 \right], 
 \end{equation}
and when $\spaceY$ is either finite- or infinite-dimensional, we have 
       \begin{equation}\label{ineq:tvbd0}
           d_{\mathrm{tv}}\left(\P_{\phi},\P_{\psi}\right)            \leq \frac{1}{2} \sqrt{\tau M\, \E\left[ \left\|R_{\psi-\phi}[u]\right\|_\spaceY^2 \right]}. 
                  \end{equation}
   \end{lemma}

\subsection{Proofs of the key lemmas in minimax lower rate}\label{sec:proofs-lower}

\begin{proof}[Proof of \eqref{ineq:infsup_lowerb}]
We reduce the supremum to the average over the set $\Phi_M$ (which has $2^{L_M}$ elements), 
\begin{align*}
	\sup_{\phi_*\in \Phi_M}   \E_{\phi_*}[  \| \widehat{\phi}-\phi_* \|_{L^2_\rho} ^2 ] 
            &\geq 2^{-L_M} \sum_{\phi_*\in \Phi_M} \E_{\phi_*}[  \| \widehat{\phi}-\phi_* \|_{L^2_\rho} ^2 ].  
\end{align*}
 Meanwhile, by orthogonality of $\{\psi_k\}$ and definition of the set $\Phi_M$, we get
\begin{align*}
	\E_{\phi_*}[  \| \widehat{\phi}-\phi_* \|_{L^2_\rho} ^2 ] &=\sum_{k=n_M}^{n_M+L_M-1} \E_{\phi_*}\left[\left|\widehat{\theta}_k-\theta_k^*\right|^2\right] =\sum_{k=n_M}^{n_M+L_M-1} L_M^{-1}L^2\lambda_k^{\beta} \P_{\phi_*}\left(\widehat{\theta}_k\ne\theta_k^*\right)\,.
\end{align*}
Thus, we obtain that, for each $\widehat{\phi}\in \Phi_M$, 
    \begin{equation}\label{eq:supsum}
        \begin{aligned}
            \sup_{\phi_*\in \Phi_M}  & \E_{\phi_*}[  \| \widehat{\phi}-\phi_* \|_{L^2_\rho} ^2 ] \geq L_M^{-1} L^2 2^{-L_M} \sum_{k=n_M}^{n_M+L_M-1} \lambda_k^{\beta } \sum_{\phi_*\in \Phi_M} \P_{\phi_*}\left(\widehat{\theta}_k\ne\theta_k^*\right)\\
            \geq &\ \left( L_M^{-1} L^2\sum_{k=n_M}^{n_M+L_M-1} \lambda_k^{\beta } \right) 2^{-L_M} \min_{k}\sum_{\phi_*\in \Phi_M} \P_{\phi_*}\left(\widehat{\theta}_k\ne\theta_k^*\right). 
        \end{aligned}
    \end{equation}
Taking the infimum over $\widehat{\phi}\in \Phi_M$, we obtain \eqref{ineq:infsup_lowerb}. 
\end{proof}

    \begin{proof}[Proof of Lemma \ref{lem:testerror}]
    For each $\phi\in \Phi_M$ and $n_M\leq k\leq n_M+L_M-1$, we define $\phi^{(-k)} \in \Phi_M$ to be identical to $\phi$ in all coefficients except for the $k$-th coefficient $\theta_k(\phi)$, which is either flipped from $0$ to $L_M^{-1/2} L \lambda_k^{\beta/2}$ or from $L_M^{-1/2} L \lambda_k^{\beta/2}$ to $0$, i.e., 
    \begin{equation*}
        \phi^{(-k)} :=\phi - \theta_k(\phi)\psi_k + \left[  L_M^{-1/2}L\lambda_k^{\beta/2} - \theta_k(\phi) \right]\psi_k. 
    \end{equation*}
  We have $\sum_{\phi_*\in \Phi_M} \P_{\phi_*}\left(\widehat{\theta}_k\ne\theta_k^*\right) = \sum_{\phi_*\in \Phi_M} \P_{\phi_*^{(-k)}}\left(\widehat{\theta}_k\ne\theta_k(\phi_*^{(-k)})\right) $ by symmetry of opposition over the set $\Phi_M$. Hence, 
        \begin{align*}
        \sum_{\phi_* \in \Phi_M} \P_{\phi_*}\left(\widehat{\theta}_k\ne\theta_k^*\right) &= \frac{1}{2}\sum_{\phi_*\in \Phi_M} \left[ \P_{\phi_*}\left(\widehat{\theta}_k\ne\theta_k^*\right) + \P_{\phi_*^{(-k)}}\left(\widehat{\theta}_k\ne \theta_k(\phi_*^{(-k)})\right)\right]  \\
          & \geq \frac{1}{2}\sum_{\phi_*\in \Phi_M} \left( 1 - d_{\mathrm{tv}}\big (\P_{\phi_*},\P_{\phi_*^{(-k)}}\big)\right),
              \end{align*}
    where the inequality follows from applying Lemma \ref{lemma:NP} in the Appendix (Neyman–Pearson) to the hypothesis testing of $\widehat \theta_k = \theta_k(\phi_*^{(-k)}) $ against $\widehat \theta_k = \theta_k^* $.

Meanwhile, since $\E\left[ \left\|R_{\phi^{(-k)}-\phi}[u^m]\right\|_\spaceY^2 \right]= \innerp{\LGbar(\phi^{(-k)}-\phi),\phi^{(-k)}-\phi}_{L_\rho^2}=  L^2 L_M^{-1} \lambda_k^{\beta + 1}$ for all $\phi\in \Phi_M$, Lemma \ref{lem:tvbd} implies that  
    \begin{equation*}
        d_{\mathrm{tv}}\left(\P_{\phi},\P_{\phi^{(-k)}}\right) \leq \frac{1}{2} \, \sqrt{\tau M \E\left[ \left\|R_{\phi^{(-k)}-\phi}[u^m]\right\|_\spaceY^2 \right] } = \frac{1}{2}\sqrt{ \tau M L^2 L_M^{-1} \lambda_k^{\beta + 1} }.
    \end{equation*}
    Thus, $1-  d_{\mathrm{tv}}\left(\P_{\phi_*},\P_{\phi_*^{(-k)}}\right) \geq 1- \frac{1}{2}\sqrt{\tau M L^2 L_M^{-1} \lambda_{n_M}^{\beta + 1}} $ for $n_M\leq k <n_M+L_M$ and $\phi_*\in \Phi_M$. Therefore, we have
    \begin{align*}
              2^{-L_M}\inf_{\widehat \phi\in L^2_\rho}\min_{k}  \sum_{\phi_*\in \Phi_M}\P_{\phi_*}\left(\widehat{\theta}_k\ne\theta_k^*\right) & \geq 
             2^{-L_M-1}\inf_{\widehat \phi\in L^2_\rho}\min_{k}   \sum_{\phi_*\in \Phi_M}\left(1 - d_{\mathrm{tv}}\left(\P_{\phi_*},\P_{\phi_*^{(-k)}}\right) \right)\\
             &  \geq 
              2^{-1}\left(1 - \frac{1}{2}\sqrt{\tau M L^2 L_M^{-1} \lambda_{n_M}^{\beta + 1}}\right), 
    \end{align*}
   which gives \eqref{eq:lb_avg_ProbErr}. 
    \end{proof}

 \begin{proof}[Proof of Lemma \ref{lem:lowerb_finite_est}]
For each estimator $\widehat{\phi}$, we construct $P_M(\widehat{\phi})\in\Phi_M$ such that 
    \begin{equation}\label{eq:tilde_est}
        \|P_M(\widehat{\phi})-\widehat{\phi}\|_{L_\rho^2} = \min_{\phi'\in\Phi_M}\|{\phi'}-\widehat{\phi}\|_{L_\rho^2}.
    \end{equation}
    Then, for $\phi_*\in\Phi_M$, by the triangle inequality,
     \begin{equation*}
     \|P_M(\widehat{\phi})-{\phi_*}\|_{L_\rho^2}\leq \|P_M(\widehat{\phi})-\widehat{\phi}\|_{L_\rho^2} + \|\widehat{\phi}-\phi_*\|_{L_\rho^2}\leq 2\|\widehat{\phi}-\phi_*\|_{L_\rho^2}.
    \end{equation*}
    Taking the expectation and then supremum over $\phi_*\in\Phi_M$, we have 
    $\sup\limits_{\phi_*\in \Phi_M}  \E_{\phi_*}[\|\widehat{\phi}-\phi_*\|_{L_\rho^2}^2] \geq \frac{1}{4} \sup\limits_{\phi_*\in \Phi_M}  \E_{\phi_*}[\|P_M(\widehat{\phi}) -\phi_*\|_{L_\rho^2}^2]$.  
    As a result, taking the infimum over all $\widehat\phi \in L^2_\rho$, we obtain 
\[
\inf_{\widehat \phi\in L^2_\rho}\sup_{\phi_*\in \Phi_M}   \E_{\phi_*}[ \| \widehat{\phi}-\phi_* \|_{L^2_\rho} ^2 ]\geq \frac{1}{4}\inf_{\widehat{\phi}\in\Phi_M}\sup_{\phi_*\in \Phi_M}   \E_{\phi_*}[ \| \widehat{\phi}-\phi_* \|_{L^2_\rho} ^2 ].
\]

To construct $P_M(\widehat{\phi})$ satisfying \eqref{eq:tilde_est} for every $\widehat{\phi} = \sum_{k=1}^\infty \widehat{\theta}_k\psi_k\in L^2_\rho$, we first project $\widehat{\phi}$ on $\mathspan\{\psi_{k}\}_{k=n_M}^{n_M+L_M-1}$ and then map its coefficients $\{\widehat{\theta}_k\}_{k=n_M}^{n_M+L_M-1}$ to the binary sets:

\begin{equation*}
P_M(\widehat{\phi}) = \sum_{k=n_M}^{n_M+L_M-1} \widetilde{\theta}_k\psi_k, \quad \text{ with }\widetilde{\theta}_k = 
\begin{cases}
    0,&\text{ if }\widehat{\theta}_k\leq L_M^{-1}L\lambda_k^{\beta/2}/2;\\
    L_M^{-1}L\lambda_k^{\beta/2},&\text{ if }\widehat{\theta}_k> L_M^{-1}L\lambda_k^{\beta/2}/2.
\end{cases}
\end{equation*}
It is direct to verify that for every $\phi'=\sum_{k=n_M}^{n_M+L_M-1} \theta'_k\psi_k\in\Phi_M$, 
\begin{align*}
    \|{\phi'}-\widehat{\phi}\|_{L_\rho^2}^2 &= \sum_{k=n_M}^{n_M+L_M-1}(\theta'_k-\widehat{\theta}_k)^2 + \bigg(\sum_{k=1}^{n_M-1} + \sum_{k=n_M+L_M}^\infty\bigg)\widehat{\theta}_k^2\\
    &\geq \sum_{k=n_M}^{n_M+L_M-1}(\widetilde{\theta}_k-\widehat{\theta}_k)^2 + \bigg(\sum_{k=1}^{n_M-1} + \sum_{k=n_M+L_M}^\infty\bigg)\widehat{\theta}_k^2 = \|P_M(\widehat{\phi})-\widehat{\phi}\|_{L_\rho^2}^2, 
\end{align*}
and it implies \eqref{eq:tilde_est}. 
\end{proof}

    \appendix
\section{Proof of Lemma \ref{lemma:samplingError} (the bound of the sampling error)}\label{sec:samplingError} 
The tight bound on the sampling error, established in {\rm Lemma \ref{lemma:samplingError}}, is crucial for achieving the minimax upper rate. To derive this result, we decompose the error into two components: the error arising from the orthogonal component and the noise-induced error. The noise-induced error demands careful treatment, which we address in Lemma \ref{Lem:noise_up} below, after the proof of {\rm Lemma \ref{lemma:samplingError}}.

\begin{proof}[Proof of Lemma \ref{lemma:samplingError}]
Recall that $\phi_*=\sum_{k=1}^\infty \theta_k^*\psi_k = (\sum_{k=1}^n+\sum_{k=n+1}^\infty)\theta_k^*\psi_k = \phi^*_{\calH_n} + \phi^*_{\calH_n^\bot}$, and that $\btheta_n^* = (\theta_1^*,\cdots,\theta_n^*)^\top$. Using $f =R_{\phi^*_{\calH_n}}[u] + R_{\phi^*_{\calH_n^\bot}}[u] + \varepsilon $, we decompose $\bbar_{n,M}$ as 
    \begin{align*}
        \bbar_{n,M}(k) &= \frac{1}{M}\sum_{m=1}^M\innerp{f^{m}, R_{\psi_k}[u^{m}]}
        = \frac{1}{M}\sum_{m=1}^M\innerp{R_{\phi^*_{\calH_n}}[u^{m}] + R_{\phi^*_{\calH_n^\bot}}[u^{m}] + \varepsilon^m, R_{\psi_k}[u^{m}]}\\
        & = [\Abar_{n,M}\btheta_{n}^*](k) + \Bar{c}_{n,M}(k) + \Bar{d}_{n,M}(k), 
    \end{align*}
where we denote, for $1\leq k\leq n$,  
\begin{equation}\label{eq:cd}
    \Bar{c}_{n,M}(k) := \frac{1}{M}\sum_{m=1}^M\innerp{R_{\phi^*_{\calH_n^\bot}}[u^{m}], R_{\psi_k}[u^{m}]}_{\spaceY},\ \Bar{d}_{n,M}(k) : = \frac{1}{M}\sum_{m=1}^M\innerp{\varepsilon^m, R_{\psi_k}[u^{m}]}.
    \end{equation}
Therefore, the variance term becomes 
    \begin{align*}
        \E[\|\Abar_{n,M}^{-1}\bbar_{n,M}- \btheta_{n}^*\|^2\mathbf{1}_\mathcal{A}] &= \E[\|\Abar_{n,M}^{-1}(\Bar{c}_{n,M} + \Bar{d}_{n,M})\|^2\mathbf{1}_\mathcal{A}]\\
        &\leq \E[\|\Abar_{n,M}^{-1}\Bar{c}_{n,M}\|^2\mathbf{1}_\mathcal{A}] + \E[\|\Abar_{n,M}^{-1} \Bar{d}_{n,M}\|^2\mathbf{1}_\mathcal{A}].
    \end{align*}
The second term is bounded by Lemma \ref{Lem:noise_up} below. Thus, we only need to show the following bounds for the first term: 
\begin{equation}\label{eq:Ac}
\E[\|\Abar_{n,M}^{-1}\Bar{c}_{n,M}\|^2\mathbf{1}_\mathcal{A}]\leq \frac{16\kappa L^2}{ M}\lambda_n^{\beta-1}\sum_{k=1}^n\lambda_k.
\end{equation}

Since $\|\Abar_{n,M}^{-1}\|\leq 4 \lambda_n^{-1}$ for either case of $\mathcal{A}$, we have 
    \begin{align*}
        \E[\|\Abar_{n,M}^{-1}\Bar{c}_{n,M}\|^2\mathbf{1}_\mathcal{A}]
        & 
        \leq 16 \lambda_n^{-2} \E[\|\Bar{c}_{n,M}\|^2] = 16 \lambda_n^{-2} \sum_{k=1}^n \E[|\Bar{c}_{n,M}(k)|^2]. 
    \end{align*}
By sample independence and that $\E \left[\innerp{R_{\phi^*_{\calH_n^\bot}}[u^{m}], R_{\psi_k}[u^{m}]}_{\spaceY}\right] = \innerp{\LGbar\phi^*_{\calH_n^\bot}, \psi_k}_{L_\rho^2} = 0$, we obtain  
        \begin{align*}
        \E[|\Bar{c}_{n,M}(k)|^2]  = \E\left[\bigg|\frac{1}{M}\sum_{m=1}^M\innerp{R_{\phi^*_{\calH_n^\bot}}[u^{m}], R_{\psi_k}[u^{m}]}_{\spaceY}\bigg|^2\right]
        &= \frac{1}{M} \E \left[\innerp{R_{\phi^*_{\calH_n^\bot}}[u^{m}], R_{\psi_k}[u^{m}]}_{\spaceY}^2\right].
    \end{align*}
Meanwhile, the fourth-moment condition in {\rm Assumption \ref{assum:4thmom}} implies 
    \begin{align*}
        \E \left[\innerp{R_{\phi^*_{\calH_n^\bot}}[u^{m}], R_{\psi_k}[u^{m}]}_{\spaceY}^2\right] &\leq \E \left[\|R_{\phi^*_{\calH_n^\bot}}[u^{m}]\|_{\spaceY}^2\|R_{\psi_k}[u^{m}]\|_\spaceY^2 \right]\\
        &\leq \left( \E \left[\|R_{\phi^*_{\calH_n^\bot}}[u^{m}]\|_{\spaceY}^4\right]\right)^{\frac{1}{2}}\left( \E \left[\|R_{\psi_k}[u^{m}]\|_\spaceY^4\right]\right)^{\frac{1}{2}}\\
        &\leq \kappa\  \E \left[\|R_{\phi^*_{\calH_n^\bot}}[u^{m}]\|_{\spaceY}^2\right]\E \left[\|R_{\psi_k}[u^{m}]\|_\spaceY^2\right]\\
       & \leq \kappa  L^2 \lambda_k \lambda_n^{\beta + 1}, 
    \end{align*}
where the last inequality follows from that $\E \left[\|R_{\psi_k}[u^{m}]\|_\spaceY^2\right]=\innerp{\LGbar\psi_k,\psi_k}_{L^2_\rho} =  \lambda_k$ and similarly, $ \E \left[\|R_{\phi^*_{\calH_n^\bot}}[u^{m}]\|_{\spaceY}^2\right]= \innerp{\LGbar \phi^*_{\calH_n^\bot},\phi^*_{\calH_n^\bot}}_{L_\rho^2}\leq \lambda_{n+1} \|\phi^*_{\calH_n^\bot}\|^2_{L_\rho^2} \leq \lambda_{n+1}^{1+\beta} L^2 $ by {\rm Lemma \ref{lemma:SSS}}. Combining these results, we obtain \eqref{eq:Ac}, completing the proof. 
\end{proof}

The noise-induced error term, $\E[\|\Abar_{n,M}^{-1}\Bar{d}_{n,M}\|^2\mathbf{1}_\mathcal{A}]$, requires careful handling. The conventional approach in \cite{wang2023optimal,tsybakov2008introduction}, which uses the smallest eigenvalue to bound the operator norm of $\Abar_{n,M}^{-1}$ in the well-posed setting, results in a bound that is too loose to achieve the optimal rate (see Remark \ref{rmk:looseBd}). To address this, we employ a singular value decomposition (SVD) to leverage the trace of the normal matrix, enabling a tight bound that ensures the optimal rate.

\begin{lemma}[Tight bound for the noise-induced error]\label{Lem:noise_up}
   Under {\rm Assumptions \ref{assum:model}} and {\rm \ref{assump:noise-X-general}}, 
   the noise-induced error term with $\Bar{d}_{n,M}$ defined in \eqref{eq:cd} satisfies
\begin{equation}
\begin{aligned}
\E[\|\Abar_{n,M}^{-1}\Bar{d}_{n,M}\|^2\mathbf{1}_\mathcal{A}] 
& \leq \sigma^2 \sum_{k=1}^n \E\left[\frac{\lambda_k^{-1}(\Abar_{n,M})}{M} \mathbf{1}_\mathcal{A}\right] \\
& \leq 
\begin{cases} 
\frac{4\sigma^2 n}{M \lambda_n}, & \text{if } \mathcal{A} = \{\lambda_{\min}(\Abar_{n,M}) > \lambda_n / 4\};  \\
\frac{4\sigma^2}{M} \sum_{k=1}^n \lambda_k^{-1}, & \text{if } \mathcal{A} = \{\lambda_k(\Abar_{n,M}) > \lambda_k / 4, \forall k \leq n\}. 
\end{cases}
\end{aligned}
\end{equation}
 In particular, when the noise $\varepsilon$ satisfies $\E[\langle \varepsilon, y \rangle^2] = \|y\|_{\spaceY}^2$ for each $y \in \spaceY$, the bound is tight in the sense that the first inequality becomes an equality.
\end{lemma}

\begin{remark}\label{rmk:looseBd}
The tight bound for noise-induced error is crucial for achieving the optimal rate since this term dominates the sampling error. We illustrate it when $\lambda_k\asymp k^{-2r}$ for $k\geq 1$. Recall that the other part in the sampling error, as bounded in \eqref{eq:Ac}, is of order $O\left(\frac{1}{M}\lambda_n^{\beta-1}\sum_{k=1}^n\lambda_k\right) = O\left(\frac{n}{M\lambda_n}\lambda_n^{\beta}\right)$. Then, the noise-induced error, which is of order $ O\left(\frac{n}{M\lambda_n}\right)= O\left(\frac{n^{1+2r}}{M}\right)$, dominates the sampling error and leads to the optimal rate. In contrast, the common approach of bounding the operator norm $\|\Abar_{n,M}^{-1} \|_{op}$ results in a suboptimal bound. Specifically, since $\|\Abar_{n,M}^{-1} \|_{op}\leq 4\lambda_n^{-1}$ on $\calA$ and $\E[|\Bar{d}_{n,M}(k)|^2] \leq \sigma^2M^{-1} \lambda_k$, this operator norm bound gives  
\begin{equation*}
	\E[\|\Abar_{n,M}^{-1}\Bar{d}_{n,M}\|^2\mathbf{1}_\mathcal{A}]\leq 16 \lambda_n^{-2} \E[\|\Bar{d}_{n,M}\|^2] = 16\sigma^2 \frac{1}{M\lambda_n^2} \sum_{k=1}^n\lambda_k, 
\end{equation*}
which is of order $O\left(\frac{n^{4r}}{M}\right)$, $O\left(\frac{n^2\log n}{M}\right)$, and $O\left(\frac{n^{1+2r}}{M}\right)$ for $r>\frac{1}{2}$, $r=\frac{1}{2}$, and $r<\frac{1}{2}$, respectively. Thus, the resulting order is strictly larger than our $O\left(\frac{n^{1+2r}}{M}\right)$ except when $r<\frac{1}{2}$.
 
Furthermore, the tight bound discussed above and the results in {\rm Lemma \ref{lemma:samplingError}} hold for all $r \geq 0$ because their proofs do not rely on any spectral decay condition of the normal operator. Notably, these results remain valid even when the normal operator is the identity operator, as in classical regression. Consequently, our sampling error estimation is directly applicable to classical regression.
\end{remark}

\begin{proof}[Proof of Lemma \ref{Lem:noise_up}]
    Apply SVD to $\Abar_{n,M}$, we obtain 
    \[
    \Abar_{n,M}=\Tilde{U}^\top\Tilde{\Sigma}\Tilde{U}
    \]
    with $\Tilde{\Sigma} = {\rm{diag}}(\lambda_1(\Abar_{n,M}),\cdots,\lambda_n(\Abar_{n,M}))$  and $\Tilde{U}=(\Tilde{u}_{kl})_{1\leq k,l\leq n}\in\R^{n\times n}$ being the real unitary matrix consisting of orthonormal eigenvectors of $\Abar_{n,M}$.
    
      Then, the noise-induced error term can be computed as
    \begin{align*}
        \E[\|\Abar_{n,M}^{-1} \Bar{d}_{n,M}\|^2\mathbf{1}_\mathcal{A}] 
        &= \E\left[\|\Tilde{U}^\top\Tilde{\Sigma}^{-1}\Tilde{U} \Bar{d}_{n,M}\|^2\mathbf{1}_\mathcal{A}\right] = \E\left[\|\Tilde{\Sigma}^{-1}\Tilde{U} \Bar{d}_{n,M}\|^2\mathbf{1}_\mathcal{A}\right] \\
       & = \E\left[ \sum_{k=1}^n \lambda_k^{-2}(\Abar_{n,M})\big(\Tilde{U} \Bar{d}_{n,M}\big)_k^2 \mathbf{1}_\mathcal{A} \right] \\
       & = \sum_{k=1}^n  \Ebracket{ \E[  \lambda_k^{-2}(\Abar_{n,M})\big(\Tilde{U} \Bar{d}_{n,M}\big)_k^2 \mathbf{1}_\mathcal{A} \mid \Abar_{n,M}] } \\
       & = \sum_{k=1}^n  \Ebracket{\lambda_k^{-2}(\Abar_{n,M})\mathbf{1}_\mathcal{A} \E[ \big(\Tilde{U} \Bar{d}_{n,M}\big)_k^2  \mid \Abar_{n,M}] }. 
    \end{align*}
To compute $\E[ \big(\Tilde{U} \Bar{d}_{n,M}\big)_k^2  \mid \Abar_{n,M}]$, note that 
    \begin{align*}
        \big(\Tilde{U} \Bar{d}_{n,M}\big)_k &= \sum_{l=1}^n \Tilde{u}_{kl}\Bar{d}_{n,M}(l)= \sum_{l=1}^n \Tilde{u}_{kl}  \frac{1}{M}\sum_{m=1}^M\innerp{\varepsilon^m, R_{\psi_l}[u^{m}]} = \frac{1}{M}\sum_{m=1}^M\innerp{\varepsilon^m, \sum_{l=1}^n \Tilde{u}_{kl}R_{\psi_l}[u^{m}]}. 
    \end{align*}
   Then, since $\varepsilon^m$ is independent of $u^m$, {\rm Assumption \ref{assump:noise-X-general} \ref{assum:noise_up}} implies 
       \begin{equation}\label{eq:Asvd_core}
       \begin{aligned}
        \E[ \big(\Tilde{U} \Bar{d}_{n,M}\big)_k^2  \mid \Abar_{n,M}] 
        & =   \frac{1}{M^2}\sum_{m=1}^M \Ebracket{ \innerp{\varepsilon^m, \sum_{l=1}^n \Tilde{u}_{kl}R_{\psi_l}[u^{m}]}^2\mid \Abar_{n,M}} \\
        & \leq \sigma^2 \frac{1}{M^2}\sum_{m=1}^M  \left\|\sum_{l=1}^n \Tilde{u}_{kl}R_{\psi_l}[u^{m}]\right\|_{\spaceY}^2  \\ 
        & = \sigma^2 \frac{1}{M}  [\Tilde{u}_{k1},\cdots,\Tilde{u}_{kn}]\Abar_{n,M}\begin{bmatrix}
            \Tilde{u}_{k1}\\
            \vdots\\
            \Tilde{u}_{kn}
        \end{bmatrix} = \sigma^2 \frac{1}{M} \lambda_{k}(\Abar_{n,M}) . 
       \end{aligned}
       \end{equation}
Thus, by collecting the above estimates, we arrive at
    \begin{align*}
        \E[\|\Abar_{n,M}^{-1} \Bar{d}_{n,M}\|^2\mathbf{1}_\mathcal{A}] 
        &\leq \sigma^2\sum_{k=1}^n\Ebracket{\frac{\lambda_k^{-1}(\Abar_{n,M})}{M}\mathbf{1}_\mathcal{A}} 
         \leq \frac{4\sigma^2 n}{M\lambda_n}\text{ or }\frac{4\sigma^2 }{M}\sum_{k=1}^n \lambda_k^{-1}
    \end{align*}
    for $\calA = \{\lambda_{\min}  (\Abar_{n,M})>\lambda_n/4 \}$ or $ \{\lambda_k(\Abar_{n,M})>\lambda_k/4,\forall k\leq n\}$, respectively. 
    
    In particular, when $\E[\innerp{\varepsilon,y}^2] = \|y\|_\spaceY^2$ for each $y\in \spaceY$, \eqref{eq:Asvd_core} becomes an equality, so does the first inequality above. This completes the proof. 
\end{proof}

\section{Left-tail probability of the smallest eigenvalue}\label{sec:PAC}
In this section, we employ a relaxed PAC-Bayesian inequality to establish the left-tail probability bound for the smallest eigenvalue of the normal matrix $\Abar_{n,M}$ in \eqref{eq:Ab}. Our approach follows the framework developed in \cite{wang2023optimal}, but with a notable relaxation of the entry-wise boundedness assumption on random matrices.  Specifically, instead of requiring the matrix entries to be almost surely bounded, we impose only the fourth-moment condition in {\rm Assumption \ref{assum:4thmom}}.  
This relaxation is made possible through the use of eigenfunctions of the normal operator to make $\Abar_{n,\infty}$  diagonal and through careful treatment of the random trace term that emerges when applying the PAC-Bayesian inequality.

\subsection{The PAC-Bayesian inequality and preliminaries}

Our primary tool is the following PAC-Bayesian inequality; see, e.g., \cite{Mourtada2022,Oliveira2016,AudibertCatoni2011}.
\begin{lemma}[PAC-Bayesian inequality]\label{lemma:PAC_Bay}
	Let $\Theta$ be a measurable space, and $\{Z(\theta):\theta\in \Theta\}$ be a real-valued measurable process. Assume that 
	\begin{equation}\label{Eq:PAC_Assu}
		\E[\exp(Z(\theta))]\leq 1\,,\quad \text{for every } \theta\in\Theta\,.
	\end{equation}
	Let $\pi$ be a probability measure on $\Theta$. Then,
	\begin{align}\label{Eq:PAC_Ineq}
		\P\left\{\forall \mu \in \mathcal{P}, \int_{\Theta} Z(\theta) \mu(\theta)\leq \textup{KL}(\mu,\pi)+t\right\}\geq 1-e^{-t}\,,
	\end{align}
	where $\mathcal{P}$ is the set of all probability measures on $\Theta$, and $\textup{KL}(\mu,\pi)$ is the Kullback-Leibler divergence between $\mu$ and $\pi$ defined by
$		\textup{KL}(\mu,\pi):=\begin{cases}
			\int_{\Theta} \log\Big[\frac{d\mu}{d\pi}\Big] d\mu & \text{if }\mu \ll \pi \,; \\
			\infty & \text{otherwise}\,.
		\end{cases}
$
\end{lemma}

We will apply the above PAC-Bayesian inequality to the process  
\[
 Z(\theta)= -\lambda \sum_{m=1}^M\|R_{\phi_\theta}[u^m]\|_\spaceY^2 +  \alpha M =M\lambda (- \theta^\top \overline{A}_{n,M}\theta + \alpha ), \quad  \theta\in \Theta=S^{n-1}, 
\]
 where $\lambda$ and $\alpha$ are properly selected constants, along with properly selected measures $\mu$ and $\pi$ such that the KL divergence $\textup{KL}(\mu,\pi)$ and the integral $\int_{\Theta} Z(\theta) \mu(\theta)$ can be controlled. We need the following lemmas on the exponential integrability of $Z(\theta)$, a lower bound for the trace of the normal matrix, and a control for the approximation term in the PAC-Bayesian inequality.   

 \begin{lemma}\label{lemma:mfg_4thmmt}
	Under the fourth-moment condition \eqref{ineq:4thmom} in {\rm Assumption \ref{assum:4thmom}}, we have 
	      \begin{align}\label{eq:mgf_4thmmt}
        \E\left[\exp\left(-\lambda \|R_{\phi}[u]\|_\spaceY^2\right)\right]\leq 
        \exp\left(-\lambda \E[\|R_{\phi}[u]\|_\spaceY^2] +\frac{\kappa\lambda^2}{2} \E[\|R_{\phi}[u]\|_\spaceY^2]^2 \right)
    \end{align}
for all $\phi\in H^\beta_\rho$ and $\lambda>0$. 
\end{lemma}
\begin{proof}
	 By using the inequalities $ e^{-y} \leq 1-y+ \frac{1}{2}y^2$ for all $y\geq 0$ and $1+y\leq e^y$ for all $y\in \R$, we have, for a square-integrable nonnegative random variable $X$ and $\lambda >0$, 
	  \[ \E[e^{-\lambda X}]\leq 1- \lambda \E[X] + \frac{\lambda^2}{2} \E[X^2] \leq  e^{-\lambda \E[X]+\frac{\lambda^2}{2}\E[X^2]} .
	  \] Applying it with $X=\|R_{\phi}[u]\|_\spaceY^2$ and the fourth-moment condition \eqref{ineq:4thmom}, we obtain \eqref{eq:mgf_4thmmt}. 
\end{proof}

\begin{lemma}\label{cor:tr}
    Under {\rm Assumption \ref{assum:4thmom}}, the normal matrix $\Abar_{n,M}$ defined in \eqref{eq:Ab} satisfies
    \begin{equation}\label{ineq:trlearning}
        \P\left\{\frac{\tr(\Abar_{n,M})}{n} \geq \lambda_1 + 1\right\} \leq \frac{\kappa\lambda_1^2}{M},
    \end{equation}
    where $\lambda_1$ is the largest eigenvalue of $\E[\Abar_{n,M}]$. 
\end{lemma}

\begin{proof}[Proof of Lemma \ref{cor:tr}]
Recall that $\Abar_{n,M} = \frac{1}{M}\sum_{m=1}^M A^m$, where $A^m$ are i.i.d.~matrices with entries $A^m (k,l) = \innerp{R_{\psi_k}[u^m],R_{\psi_l}[u^m]}$ for $1\leq k,l\leq n$. 

  The Fourth-moment {\rm Assumption \ref{assum:4thmom}} implies that 
  $$\var (A^m(k,k)) = \E\left[\|R_{\psi_{k}}[u^m]\|^4\right] - \left(\E\left[\|R_{\psi_{k}}[u^m]\|_\spaceY^2\right]\right)^2 \leq (\kappa-1 )\lambda_k^2. 
  $$
  Then, together with the fact that $\var ( \tr(\Abar_{n,M})) = \frac{1}{M}\var(\tr(A^m))$, we have 
    \begin{align*}
       \var ( \tr(\Abar_{n,M})) & = \frac 1M  \var(\tr(A^m)) = \frac 1M  \var\left(\sum_{k=1}^n A^m(k,k) \right) \leq \frac{n}{M} \sum_{k=1}^n \var (A^m(k,k))\\
        & \leq \frac n M  (\kappa - 1) \sum_{k=1}^n \lambda_k^2. 
    \end{align*}
Consequently, Chebyshev's equality and the fact that $ \E [\tr(\Abar_{n,M}) ]= \sum_{k=1}^n\lambda_k $ imply
    \begin{align*}
        \P\left\{\frac{\tr(\Abar_{n,M})}{n} \geq \lambda_1 + 1\right\} 
        &\leq \P\left\{\tr(\Abar_{n,M}) \geq \sum_{k=1}^n\lambda_k + n\right\} \\
        &\leq \frac{\var( \tr(\Abar_{n,M}) )}{n^2} \leq \frac{\kappa-1}{nM} \sum_{k=1}^n \lambda_k^2 
         \leq  \frac{\kappa  \lambda_1^2}{M}.
    \end{align*}
The proof is completed.
\end{proof}

The next lemma, from \cite[Supplementary Section 2.3]{Mourtada2022} (see also in \cite{wang2023optimal} for a constructive proof), controls the approximate term in the application of the PAC-Bayesian inequality. 
\begin{lemma}\label{Lem:F_Sigma} 
	For every $\gamma\in(0,1/2]$, $v\in S^{n-1}$ with $n\geq 2$, define 
	\begin{equation}\label{eq:prob_measures}
	\Theta_{v,\gamma}:=\{\theta\in S^{n-1}: \|\theta-v\|\leq \gamma \} \quad \text{  and } \quad	\pi_{v,\gamma}(d\theta):=\frac{\bf1_{\Theta_{v,\gamma}}(\theta)}{\pi(\Theta_{v,\gamma})}\pi(d\theta), 
	\end{equation}
	where $\pi$ is the uniform distribution on the sphere. That is, $\Theta_{v,\gamma}$ is a ``spherical cap'' or ``contact lens'' in $n$-dimensional space, and $\pi_{v,\gamma}$ is a uniform surface distribution on the spherical cap. Then, 
\begin{align*}
	F_{v,\gamma}(\Sigma):= \int_{\Theta}\innerp{\Sigma \theta, \theta} \pi_{v,\gamma}(d\theta)=[1-h(\gamma)]\innerp{\Sigma v, v}+h(\gamma) \frac{\textup{Tr}(\Sigma)}{n}, 
\end{align*} 
for every symmetric matrix $\Sigma$, where 
\begin{equation}\label{eq:h_gamma}
	h(\gamma)= \frac{n}{n-1}\int_{\Theta}[1-\innerp{\theta,v}^2]\pi_{v,\gamma}(d\theta) \in \bigg[0,\frac{n \gamma^2}{n-1}\bigg]\,.
\end{equation}
\end{lemma}

\subsection{Proof for the left-tail probability bound}
The following technical lemma establishes a parameterized bound for the left-tail probability by employing the PAC-Bayesian inequality. We will then select the optimal parameters to obtain the bounds stated in {\rm Lemma \ref{Lem:LeftProb}}.
\begin{lemma}\label{Lem:LeftProbGM}
    Under {\rm Assumption \ref{assum:4thmom}}, the matrix $\Abar_{n,M}$ in \eqref{eq:Ab} with $n\geq 2$ satisfies
    \begin{equation}\label{ineq:LeftProbGM}
        \P\left\{ \lambda_{\min}(\Abar_{n,M})\leq G^M(\gamma,t)\}\right\}\leq \frac{\kappa\lambda_1^2}{M} + \exp\left( -t \right)
    \end{equation}
 for all $\gamma\in(0,1/2]$ and $t>0$, where 
        \begin{equation}\label{eq:GM}
        G^M(\gamma,t):= c_\gamma \left[ -2\gamma^2(\lambda_1 + 1) + \frac{\lambda_1\lambda_n}{\lambda_1+\lambda_n} - \frac{\kappa(\lambda_1+\lambda_n)}{2 M}\left(n\log\left(\frac{5}{4\gamma^2}\right)+t\right)\right]
    \end{equation}
with constant $c_\gamma = 1/(1-h(\gamma))\in[1,2]$ and $h(\gamma)$ in \eqref{eq:h_gamma}. 
\end{lemma}

\begin{proof}[Proof of Lemma \ref{Lem:LeftProbGM}] We split the proof into two steps. 

\textbf{Step 1.} We define the process $Z(\theta)$ in the PAC-Bayesian inequality using $R_{\phi_\theta}[u^m]$, so that the bounds for $Z(\theta)$ can lead to the left-tail bounds of  $\lambda_{\min}(\Abar_{n,M})$. Here, we denote $\phi_\theta = \sum_{l=1}^n \theta_l \psi_l$ with $\theta = (\theta_1,\cdots,\theta_n)^\top\in S^{n-1}$, which gives $\theta^\top \Abar_{n,M} \theta = \frac{1}{M}\sum_{m=1}^M\|R_{\phi_\theta}[u^m]\|_\spaceY^2$.

 Note that $\E[\|R_{\phi_\theta}[u]\|_\spaceY^2]= \theta^\top \Abar_{n,\infty} \theta $. Under the fourth-moment condition in {\rm Assumption \ref{assum:4thmom}}, {\rm Lemma \ref{lemma:mfg_4thmmt}} implies that  
    \begin{align} \label{eq:mgf_R_phi_theta}
        \E\left[\exp\left( -\lambda \|R_{\phi_\theta}[u]\|_\spaceY^2
        \right)\right] 
        &\leq \exp\left(-\lambda \left(\theta^\top \Abar_{n,\infty} \theta\right) + \frac{\lambda^2}{2}\kappa\left(\theta^\top \Abar_{n,\infty} \theta\right)^2\right)
    \end{align}
    for all $\lambda>0$. 
    Note that $\theta^\top \Abar_{n,\infty} \theta$ runs over $[\lambda_n,\lambda_1]$ as $\theta$ running over the sphere. The quadratic function $ g_\lambda(x):= -\lambda x + \frac{\lambda^2\kappa}{2} x^2$ is maximized at either of the endpoints if its center $\frac{1}{\lambda\kappa}$ is 
    \begin{equation}\label{eq:lambda*}
        \frac{1}{\lambda\kappa} = \frac{\lambda_1 + \lambda_n}{2},\text{ i.e. }\lambda = \frac{2}{\kappa(\lambda_1+\lambda_n)}, 
    \end{equation}
   and the maximal value is $ -\frac{2}{\kappa(\lambda_1+\lambda_n)} \lambda_n + \frac{2}{\kappa(\lambda_1+\lambda_n)^2} \lambda_n^2 = -\frac{2\lambda_1\lambda_n}{\kappa(\lambda_1+\lambda_n)^2}= -\lambda \frac{\lambda_1\lambda_n}{\lambda_1+\lambda_n}$.
    Thus, with $\lambda$ in \eqref{eq:lambda*}, Eq.\eqref{eq:mgf_R_phi_theta} implies that 
        \begin{align*}
        \E\left[\exp\left( -\lambda \|R_{\phi_\theta}[u]\|_\spaceY^2
        \right)\right] 
        &\leq \exp\left(g_\lambda(\theta^\top \Abar_{n,\infty} \theta) \right) \leq \exp\left(- \lambda \frac{\lambda_1\lambda_n}{\lambda_1+\lambda_n}\right). 
    \end{align*}
Therefore, with $\lambda$ in \eqref{eq:lambda*}, the process 
    \begin{equation} \label{eq:Ztheta}
        \begin{aligned}
            Z(\theta)&:= -\lambda \sum_{m=1}^M\|R_{\phi_\theta}[u^m]\|_\spaceY^2 + \lambda \frac{\lambda_1\lambda_n}{\lambda_1+\lambda_n} M 
        \end{aligned}
    \end{equation}
    with $\theta\in S^{n-1}$ satisfying 
    \[ \sup_{\theta\in S^{n-1}}\E[e^{Z(\theta)}] \leq 1. 
    \] 
Then, the PAC-Bayesian inequality with $\Theta=S^{n-1}$ implies, 
    \begin{equation*}
        \P\left\{ \forall \mu\in\mathcal{P},\ \int_\Theta Z(\theta)\mu(d\theta)\leq \kl{\mu,\pi} + t \right\}\geq 1 - e^{-t},\quad \forall t > 0,
    \end{equation*}
for every fixed Borel probability measure $\pi$ on $S^{n-1}$, where $\mathcal{P}$ is the set of all Borel probability measures on $S^{n-1}$.

\textbf{Step 2.} We pass the bound for $Z(\theta)$ to the bound for $ \lambda_{\min}(\Abar_{n,M})$.   
Let $\pi$ be the uniform distribution on $S^{n-1}$, and only consider $\mu$ of the form $\pi_{v,\gamma}$ in \eqref{eq:prob_measures}. 
    Then, the above PAC-Bayesian inequality implies that for all $t>0$, there exists a measurable set $E_{t}$ with $\P(E_{t})\geq 1 - e^{-t}$ such that and for all $\omega\in E_{t}$ and all $v\in S^{n-1}$, $\gamma\in(0,1/2]$ (hence all $\pi_{v,\gamma}\in\mathcal{P}$),  
    \begin{equation}\label{eq:PAC-ineq}
        \int_{S^{n-1}} Z(\theta,\omega)\pi_{v,\gamma}(d\theta)\leq \kl{\pi_{v,\gamma},\pi} + t .	
    \end{equation}

The bound for the Kullback-Leibler divergence is straightforward. Since $\pi(\Theta_{v,\gamma})\geq (1+2/\gamma)^{-n}$ (see, e.g., \cite[Supplementary Section 2.4]{Mourtada2022} and \cite[Corollary 4.2.13]{Vershynin2018}), we have $$
        \kl{\pi_{v,\gamma},\pi} = \log(1/\pi(\Theta_{v,\gamma}))\leq n\log(1+2/\gamma) \leq n  \log (5/(4\gamma^2)), 
        $$ 
    where the last inequality holds because $1 + \frac{2}{\gamma}= \frac{\gamma^2+2\gamma}{\gamma^2} \leq \frac 5 {4\gamma^2}$ for $\gamma\in (0,\frac 1 2]$.
    Meanwhile, the definition of $Z(\theta)$ in \eqref{eq:Ztheta} and Lemma \ref{Lem:F_Sigma} imply that 
    \begin{align*}
        \frac{1}{M}\int_{S^{n-1}} Z(\theta)\pi_{v,\gamma}(d\theta)
        &= -\lambda \int_{S^{n-1}} \innerp{\Abar_{n,M} \theta,\theta} \pi_{v,\gamma}(d\theta) + \lambda \frac{\lambda_1\lambda_n}{\lambda_1+\lambda_n}\\
        &= -\lambda [1-h(\gamma)]\innerp{\Abar_{n,M} v,v} -\lambda h(\gamma)\frac{\tr(\Abar_{n,M})}{n} + \lambda \frac{\lambda_1\lambda_n}{\lambda_1+\lambda_n}.
    \end{align*}     
Thus, Eq.\eqref{eq:PAC-ineq} implies that for all $\omega\in E_{t}$, $v\in S^{n-1}$, $\gamma\in(0,1/2]$, 
    \begin{align*}
        &-\lambda [1-h(\gamma)]\innerp{\Abar_{n,M}(\omega) v,v} -\lambda h(\gamma)\frac{\tr(\Abar_{n,M}(\omega))}{n} + \lambda \frac{\lambda_1\lambda_n}{\lambda_1+\lambda_n}
        \leq  
         \frac{n\log (5/(4\gamma^2)) +t}{M}.
    \end{align*}
    Then, using the note notation $c_\lambda=1/(1-h(\gamma)) $, we have 
        \begin{align*}
        \innerp{\Abar_{n,M}(\omega) v,v}&\geq\frac{1}{1-h(\gamma)}\left[ -h(\gamma)\frac{\tr(\Abar_{n,M}(\omega))}{n} +  \frac{\lambda_1\lambda_n}{\lambda_1+\lambda_n} - \frac{n\log (5/(4\gamma^2))  +t}{\lambda M} \right]. 
    \end{align*}
    When $v$ runs over $S^{n-1}$, the left-hand side becomes $\lambda_{\min}(\Abar_{n,M}(\omega)) = \inf_{v\in\Theta}\innerp{\Abar_{n,M}(\omega) v,v}$, while the right-hand side is independent of $v.$ Therefore, for all $\omega\in E_{t}$, 
    \begin{equation}\label{ineq:eigen-lowerbound}
        \begin{aligned}
            \lambda_{\min}(\Abar_{n,M}(\omega)) \geq c_\gamma &\left[ -h(\gamma)\frac{\tr(\Abar_{n,M}(\omega))}{n} + \frac{\lambda_1\lambda_n}{\lambda_1+\lambda_n} -  \frac{n\log (5/(4\gamma^2))  +t}{\lambda M} \right]. 
        \end{aligned}
    \end{equation}
     
Lastly, we bound the random trace term. By Lemma \ref{cor:tr}, there exists an event $E_0$ with probability no less than $1-\kappa\lambda_1^2/M$ such that for all $\omega\in E_0$, one has $\tr(\Abar_{n,M}(\omega))/n < \lambda_1 + 1$. Meanwhile, the function $h(\gamma)$ in \eqref{eq:h_gamma}
satisfies $h(\gamma)\leq 2\gamma^2$. Then, for $\omega\in E_0 \cap E_{t}$,    
    \begin{equation}
        \label{def:GMgammalambdat}
        \begin{aligned}
            \lambda_{\min}(\Abar_{n,M}(\omega))  >  c_\gamma \left[ -2\gamma^2(\lambda_1 + 1) + \frac{\lambda_1\lambda_n}{\lambda_1+\lambda_n}  -  \frac{n\log (5/(4\gamma^2))  +t}{\lambda M}\right],
        \end{aligned}
    \end{equation}
    which equals $G^M(\gamma,t)$ defined in \eqref{eq:GM} by recall that the value of $\lambda$ in \eqref{eq:lambda*}. In other words, 

    \begin{align*}
        &\ \quad \P\left\{\lambda_{\min}(\Abar_{n,M})\leq G^M(\gamma,t)\right\}\leq \P\left\{\left(E_0 \cap E_{t}\right)^c\right\} = \P\left\{E_0^c \cup E_{t}^c\right\}\\
        &\leq \P\left(E_0^c\right) + \P\left(E_{t}^c\right)\leq \frac{\kappa\lambda_1^2}{M} + 1-(1-e^{-t}) = \frac{\kappa\lambda_1^2}{M} + \exp(-t), 
    \end{align*}
    which gives \eqref{Lem:LeftProbGM}. 
\end{proof}

\begin{proof}[Proof of Lemma \ref{Lem:LeftProb}]
We split the proof into two cases: $n\geq 2$ and $n=1$. The proof for the case $n\geq 2$ uses the PAC-Bayesian inequality-based bound in Lemma \ref{Lem:LeftProbGM}. The proof for case $n=1$ follows directly from a bound for the moment generating function of $\|R_\phi[u]\|_\spaceY^2$. 

 \textbf{Case $n\geq 2$.} By Lemma \ref{Lem:LeftProbGM},
    \begin{equation}\label{eq:left_tail_tgamma}
        \P\left\{ \lambda_{\min}(\Abar_{n,M})\leq G^M(\gamma,t)\}\right\}\leq \frac{\kappa\lambda_1^2}{M} + \exp\left( -t \right),
    \end{equation}
    where $G^M(\gamma,t)$ is defined in \eqref{eq:GM}: 
    \begin{equation*}
        G^M(\gamma,t):= c_\gamma \left[ -2\gamma^2(\lambda_1 + 1) + \frac{\lambda_1\lambda_n}{\lambda_1+\lambda_n} - \frac{\kappa(\lambda_1+\lambda_n)}{2 M}\left(n\log\left(\frac{5}{4\gamma^2}\right)+t\right)\right]
    \end{equation*}
with $c_\gamma \in[1,2]$ for all $\gamma\in(0,1/2]$ and $t>0$.

    Now that the bound $G^M(\gamma,t)$ is deterministic, what remains to do is to choose proper constants $\gamma\in(0,1/2]$ and $t>0$ such that $G^M(\gamma,t) \geq (3-\varepsilon) \lambda_n/8$. First, choose $\gamma= \sqrt{\lambda_n/(\lambda_1 + 1)}/4 < 1/4$ so that $2(\lambda_1 + 1)\gamma^2 = \lambda_n/8$.  Note that $ -2\gamma^2(\lambda_1 + 1) + \frac{\lambda_1\lambda_n}{\lambda_1+\lambda_n} \geq -\frac{\lambda_n}{8} + \frac{1}{2}\lambda_n =  \frac{3\lambda_n}{8}$. Then, 
        \begin{align*}
            G^M(\sqrt{\lambda_n/(\lambda_1 + 1)}/4,t) 
            &= c_\gamma \left[ \frac{3\lambda_n}{8} - \frac{\kappa(\lambda_1+\lambda_n)}{2 M}\left(n\log\left(\frac{5}{4\gamma^2}\right)+t\right)\right].
        \end{align*}
Next, set $t$ to satisfy $
            \frac{\kappa(\lambda_1+\lambda_n)}{2 M}\left(n\log\left(\frac{5}{4\gamma^2}\right)+t\right) = \frac{\varepsilon}{8}\lambda_n$ 
        so that (recall that $c_\gamma\in [1,2]$)
        \begin{equation*}
            c_\gamma \left[ \frac{3\lambda_n}{8} - \frac{\kappa(\lambda_1+\lambda_n)}{2 M}\left(n\log\left(\frac{5}{4\gamma^2}\right)+t\right)\right] = c_\gamma \frac{3-\varepsilon}{8}\lambda_n \geq \frac{3-\varepsilon}{8}\lambda_n. 
        \end{equation*}
        Solving for $t$, we get 
        \begin{equation}\label{eq:t}
            t = \frac{\varepsilon M}{4\kappa(\lambda_1+\lambda_n)}\lambda_n - n\log\left(\frac{5}{4\gamma^2}\right) = \frac{\varepsilon M}{4\kappa(\lambda_1+\lambda_n)}\lambda_n - n\log\left(\frac{20(\lambda_1+1)}{\lambda_n}\right).
        \end{equation}

    Therefore, by \eqref{eq:left_tail_tgamma}, we have
    \begin{align*}
        &\ \quad \P\left\{\lambda_{\min}(\Abar_{n,M})\leq  \frac{3-\varepsilon}{8} \lambda_n\right\}        \leq   \frac{\kappa\lambda_1^2}{M} + \exp\left(n\log\left(\frac{20(\lambda_1+1)}{\lambda_n}\right) - \frac{\varepsilon M \lambda_n}{4\kappa(\lambda_1+\lambda_n)}\right).
    \end{align*}

    \textbf{Case $n=1$.} Since $\lambda_{\min}(\Abar_{1,M}) = \Abar_{1,M}$, we have  
    \begin{align*}
         \P\left\{\lambda_{\min}( \Abar_{1,M})\leq \frac{3-\varepsilon}{8} \lambda_1\right\}
        &=\P\left\{\exp\left(-\lambda M\Abar_{1,M}\right)\geq \exp\left(-\frac{3-\varepsilon}{8} \lambda\lambda_1 M\right)\right\}\\
        &\leq \E\left[\exp\left(-\lambda M\Abar_{1,M}\right)\right]\exp\left(\frac{3-\varepsilon}{8} \lambda\lambda_1 M\right)\\
        &= \left(\E\left[\exp\left(-\lambda \|R_{\psi_1}[u^m]\|_\spaceY^2\right)\right]\right)^M\exp\left(\frac{3-\varepsilon}{8} \lambda\lambda_1 M\right)
    \end{align*}
    for all $\lambda>0$. Meanwhile, since $\E[\|R_{\psi_1}[u^m]\|_\spaceY^2]= \lambda_1 $,  Lemma \ref{lemma:mfg_4thmmt} implies that 
    \begin{equation*}
        \E\left[\exp\left(-\lambda \|R_{\psi_1}[u^m]\|_\spaceY^2\right)\right]\leq\exp\left(-\lambda\lambda_1+\frac{\kappa\lambda^2\lambda_1^2}{2}\right). 
    \end{equation*}

    Taking $\lambda = \frac{1}{\kappa\lambda_1}$, we obtain 
    \begin{align*}
        \P\left\{\lambda_{min}(\Abar_{1,M})\leq \frac{3-\varepsilon}{8} \lambda_1\right\}
        &\leq \exp\left(-\lambda\lambda_1 M+\frac{\kappa\lambda^2\lambda_1^2}{2}M+\frac{3-\varepsilon}{8} \lambda\lambda_1 M\right)\\
        &=\exp\left(-\frac{1+\varepsilon}{8\kappa} M\right)\\
        &\leq \frac{\kappa\lambda_1^2}{M} + \exp\left(\log\left(\frac{20(\lambda_1+1)}{\lambda_1}\right) - \frac{ \varepsilon M\lambda_1}{4\kappa(\lambda_1+\lambda_1)}\right).
    \end{align*}
    This completes the proof. 
\end{proof}

\section{Preliminaries and proofs for the lower rate}\label{sec:pre_proof_lower}
\subsection{Preliminaries} \label{sec:preliminiary-lower}

A key element in the lower bound is the probability of test errors for binary hypothesis testing. The following  Neyman-Pearson lemma connects the probability of test errors with the total variation distance, 
\[
d_{\rm tv}(\P_0,\P_1) := \sup_{A\in\calE}|\P_1(A)-\P_0(A)|,
\]
between two distributions $\P_0$ and $\P_1$ on the same measurable space $(E,\calE)$. 
\begin{lemma}[Neyman-Pearson]\label{lemma:NP}
        Let $\P_0$ and $\P_1$ be two probability measures defined on the same measurable space $(E,\calE)$. Then, among all tests $T:(E,\calE)\to\{0,1\}$,
        \begin{equation}\label{eq:NP}
            \inf_{T:(E,\calE)\to\{0,1\}}\left\{\P_0(T=1) + \P_1(T=0)\right\}=1 - d_{\mathrm{tv}}(\P_0,\P_1).
        \end{equation}
\end{lemma}

\begin{proof}[Proof of Lemma \ref{lemma:NP}]
    Note that for a test $T$, 
    \begin{align*}
        \P_0(T=1) + \P_1(T=0) &= \P_0(T=1) + 1 - \P_1(T=1) \\
        &= 1 - (\P_1(T=1) - \P_0(T=1)). 
    \end{align*}
Also, note that $T$ runs over all tests is equivalent to the set $A:=\{T=1\}$ runs over all the measurable sets. Thus, we have
    \begin{equation*}
        \begin{aligned}
            \inf_{T:(E,\calE)\to\{0,1\}}\left\{\P_0(T=1) + \P_1(T=0)\right\}&=1 - \sup_{T:(E,\calE)\to\{0,1\}}\left\{\P_1(T=1) - \P_0(T=1)\right\}\\
            &=1 - \sup_{A\in\calE}\left\{\P_1(A) - \P_0(A)\right\}\\
            &=1 - d_{\mathrm{tv}}(\P_0,\P_1)
        \end{aligned}
    \end{equation*}
since  $\sup_{A\in\calE}\left\{\P_1(A) - \P_0(A)\right\} = \sup_{A\in\calE}\left|\P_1(A) - \P_0(A)\right|$.
\end{proof}

To bound the total variation distance, we resort to Pinsker's inequality (see, e.g., \cite[Lemma 2.5]{tsybakov2008introduction}). It applies to probabilities on general measurable spaces, including finite- and infinite-dimensional spaces. 
   \begin{lemma}[Pinsker's inequality]\label{lemma:pinsker}
         Let $\P_0$ and $\P_1$ be two probability measures defined on the same measurable space $(E,\calE)$, then
        \begin{equation*}
       d_{\mathrm{tv}}\left(\P_{0},\P_{1}\right) \leq \sqrt{\frac{1}{2}\kl{\P_{0},\P_{1}}},
       \end{equation*}
    where the Kullback--Leibler divergence is    $    \kl{\P_0,\P_1} = \begin{cases}
            \E_0\left[\log\left(\frac{d\P_0}{d\P_1}\right)\right] &\text{if }\P_0\ll \P_1;\\
            +\infty&\text{otherwise.}
        \end{cases}
    $ 
    \end{lemma}

 However, the KL divergence requires the Radon-Nikodym derivative $\frac{d\P_\phi}{d\P_\psi}$ between $\P_\phi$ and $\P_\psi$, which can be measures on infinite-dimensional function spaces. But {\rm Assumption \ref{assump:noise-X-general}} on the noise only provides a bound for the KL divergence between two finite-dimensional shifted measures, and it does not even ensure the existence of $\frac{d\P_\phi}{d\P_\psi}$. The next lemma shows that the total variation between $\P_\phi$ and $\P_\psi$ is the limit of their restricted measures on a filtration, whose KL divergence can be controlled, avoiding the computation of $\frac{d\P_\phi}{d\P_\psi}$.    
    \begin{lemma}\label{lemma:tv}
  Let $\P_0$ and $\P_1$ be two probability measures defined on the same measurable space $(E,\calE)$. Consider a filtration $\calF_1\subset\calF_2\subset\cdots\subset\calF_N\cdots$ such that $\calE = \sigma\left(\bigcup_{N=1}^\infty\calF_N\right)$. Let $\{\P_{i}\big|_{\calF_N}\}$ be restricted measures on the filtration, i.e., $\left.\P_{i}\right|_{\calF_N}(A) = \P_{i}(A),\text{ for all }A\in\calF_N$ and $i=0,1$.
       Then, 
        \begin{equation}\label{eq:lim_tv}
            d_{\mathrm{tv}}\left(\P_{0},\P_{1}\right) = \lim_{N\to\infty}d_{\mathrm{tv}}\left(\left.\P_{0}\right|_{\calF_N},\left.\P_{1}\right|_{\calF_N}\right).
        \end{equation}
    \end{lemma}

\begin{proof}[Proof of Lemma \ref{lemma:tv}]
        Note that $\{d_{\mathrm{tv}}\left(\left.\P_{0}\right|_{\calF_N},\left.\P_{1}\right|_{\calF_N}\right)\}$ is non-decreasing and bounded, i.e., 
        \begin{equation*}
            \begin{aligned}
                d_{\mathrm{tv}}\left(\left.\P_{0}\right|_{\calF_N},\left.\P_{1}\right|_{\calF_N}\right) &= \sup_{A\in\calF_N}|\P_0(A)-\P_1(A)|\\
                &\leq \sup_{A\in\calF_{N+1}}|\P_0(A)-\P_1(A)| = d_{\mathrm{tv}}\left(\left.\P_{0}\right|_{\calF_{N+1}},\left.\P_{1}\right|_{\calF_{N+1}}\right)\\
                &\leq\sup_{A\in\calE}|\P_0(A)-\P_1(A)|= d_{\mathrm{tv}}\left(\P_{0},\P_{1}\right). 
            \end{aligned}
        \end{equation*}
        Therefore, the limit exists and 
        \begin{equation*}
            D:=\lim_{N\to\infty}d_{\mathrm{tv}}\left(\left.\P_{0}\right|_{\calF_N},\left.\P_{1}\right|_{\calF_N}\right)\leq d_{\mathrm{tv}}\left(\P_{0},\P_{1}\right).
        \end{equation*}
        The other half of the proof uses the monotone class theorem (see, e.g., \cite[Theorem 3.4]{Billingsley2012}). Consider the class
        \begin{equation*}
            \mathcal{C} := \{A\in\calE:|\P_0(A) - \P_1(A)|\leq D\}.
        \end{equation*}
        First, $\calF_N\subset\mathcal{C}$ for all $N\geq 1$ since $d_{\mathrm{tv}}\left(\left.\P_{0}\right|_{\calF_N},\left.\P_{1}\right|_{\calF_N}\right)\leq D$.
        Therefore, we have $\bigcup_{N=1}^\infty\calF_N\subset\mathcal{C}$. Next, we can check that 
        \begin{itemize}
            \item $\bigcup_{N=1}^\infty\calF_N$ is an algebra. This is because (i) $\emptyset\in\calF_1$;
(ii) if $A\in \bigcup_{N=1}^\infty\calF_N$, $A\in\calF_N$ for some $N\geq 1$, then $A^c\in\calF_N$;
(iii) if $A,B\in \bigcup_{N=1}^\infty\calF_N$, $A\in\calF_{N_1}$ and $B\in\calF_{N_2}$ for some $N_1,N_2\geq 1$, then $A,B\in\calF_N$ with $N = \max\{N_1,N_2\}$, and so does $A\cup B$.
  
            \item $\mathcal{C}$ is a monotone class. If $A_1\subset A_2\subset\cdots$ is a monotone non-decreasing sequence in $\mathcal{C}$, we have
                \begin{equation*}
                    \begin{aligned}
                        \left|\P_0\left(\bigcup_{n=1}^\infty A_n\right) - \P_1\left(\bigcup_{n=1}^\infty A_n\right)\right| & = \left|\lim_{n\to\infty}\P_0\left( A_n\right) - \lim_{n\to\infty}\P_1\left( A_n\right)\right|\\
                        &= \lim_{n\to\infty}\left|\P_0\left( A_n\right) - \P_1\left( A_n\right)\right|\leq D.
                    \end{aligned}
                \end{equation*}
Meanwhile, if $A_1\supset A_2\supset\cdots$ is a monotone non-increasing sequence in $\mathcal{C}$, we have
                \begin{equation*}
                    \begin{aligned}
                      \left|\P_0\left(\bigcap_{n=1}^\infty A_n\right) - \P_1\left(\bigcap_{n=1}^\infty A_n\right)\right|   &= \left|\lim_{n\to\infty}\P_0\left( A_n\right) - \lim_{n\to\infty}\P_1\left( A_n\right)\right|\\
                        &= \lim_{n\to\infty}\left|\P_0\left( A_n\right) - \P_1\left( A_n\right)\right|\leq D.
                    \end{aligned}
                \end{equation*}
        \end{itemize}
        Thus, by the monotone class theorem in \cite{Billingsley2012}, $\calE = \sigma\left(\bigcup_{N=1}^\infty\calF_N\right) \subset \mathcal{C}$.
Hence, $d_{\mathrm{tv}}\left(\P_{0},\P_{1}\right) \leq D$, which completes the proof.
    \end{proof}

\subsection{Proof of Lemma \ref{lem:tvbd} (the total variation bound)}\label{sec_append:tv_bd}

  To start, we explicitly define $\P_\phi$, the probability measure of the samples $\{(u^{m},f^m)\}_{m=1}^M$ with $f^m= R_\phi[u^m]+ \varepsilon^m$. We first introduce the filtrations using the following random variables. Let $\{y_i\}_{i\geq 1}$ be an orthonormal basis of $\spaceY$, and denote 
  \begin{equation}\label{eq:Zepsilon}
Z_{\phi,u^m,N} = (\innerp{R_\phi[u^m],y_i}_\spaceY)_{i=1}^N, \ 
Z_{f^m,N}=(\innerp{f^m,y_i})_{i=1}^N, \
Z_{\varepsilon^m,N}=(\innerp{\varepsilon^m,y_i})_{i=1}^N,  
\end{equation}
which are $\R^N$-valued random variables induced by the samples. Note that $Z_{f^m,N} = Z_{\phi,u^m,N}+ Z_{\varepsilon^m,N}$ for each $m$. Let   
  \begin{equation}\label{def:filtration_data}
   \calF_\infty:= \sigma\bigg(\bigcup_{N\geq 1} \calF_N\bigg), \quad   
   \calF_N     := \sigma\left(\left\{u^m, Z_{\varepsilon^m,N}    \right\}_{m=1}^M 
   \right) 
     \,, \forall N\geq 1. 
      \end{equation}
Note that a $\calF_N$-measurable set $A_N$ is in the form 
\begin{equation}\label{eq:AN_FN}
A_N = \{\omega\in \Omega: \{u^m(\omega), Z_{\varepsilon^m,N}(\omega)\}_{m=1}^M\in B_N\}
\end{equation} for some $B_N \in (\mathcal{B}(\spaceX)\otimes \mathcal{B}(\R^N))^{\otimes M}$. Also,  the set  
\[
A_N^\phi = \big\{\omega\in \Omega: \{u^m(\omega),Z_{\varepsilon^m,N}(\omega)+Z_{\phi,u^m,N} (\omega)\}_{m=1}^M \in B_N \big\} 
\]
is in $\calF_N$ since $R_\phi:\spaceX\to \spaceY$ is measurable. At last, if $\spaceY$ is infinite-dimensional, $\calF_\infty=\sigma\left(\bigcup_{N\geq 1}\calF_N\right)=\lambda\left(\bigcup_{N\geq 1}\calF_N\right)$. Here, $\bigcup_{N\geq 1}\calF_N$ is a $\pi$-system, and $\lambda\left(\bigcup_{N\geq 1}\calF_N\right)$ is the $\lambda$-system generated by $\bigcup_{N\geq 1}\calF_N$. It is equal to $\calF_\infty$ due to the $\pi$-$\lambda$ theorem. Moreover, a probability distribution on $(\Omega,\calF_\infty)$ is determined by its behavior on $\bigcup_{N\geq 1}\calF_N$. With these notations, the explicit description of $\P_\phi$ is as follows. 
\begin{itemize}
\item When $\spaceY$ is finite-dimensional (i.e., $\spaceY = \mathrm{span}\{y_1,\ldots,y_N\}$), $\P_\phi$ is a measure on $\calF_N$:
  \[\P_\phi(A_N): = \P (A_N^\phi ), \quad \forall A_N\in \calF_N.  \]	
 \item When $\spaceY$ is infinite-dimensional, $\P_\phi$ is a measure on $\calF_\infty$, determined by
    \[P_\phi(A_N): = \P (A_N^\phi ),\quad \forall A_N \in \calF_N,\, \forall N\geq 1 .\]
    In particular, we define the restricted measures of $\P_\phi$ on $\calF_N$ as  
   \begin{equation}\label{eq:P_FN}
      \P_{\phi,N}:=\P_{\phi}\big|_{\calF_N}, \, \text{ i.e., }\ P_{\phi}\big| _{\calF_N}(A_N) = \P_{\phi}(A_N),\text{ for all } A_N\in\calF_N. 
    \end{equation}
\end{itemize}

\begin{proof}[Proof of Lemma \ref{lem:tvbd}]
    The probability measures $\P_\phi$ and $\P_\psi$ are induced by samples $\{(u^m,f^m )\}_{m=1}^M$ when $f^m = R_\phi[u^m]+\varepsilon^m$ and $f^m =R_\psi[u^m]+\varepsilon^m$, respectively. In particular, note that $\P_0$ is the measure induced by $\{(u^m,\varepsilon^m)\}_{m=1}^M$ since $f^m = R_0[u^m] + \varepsilon^m= \varepsilon^m$.

Our goal is to prove that the next inequality for $\P_{\phi, N}$ and $\P_{\psi, N}$ defined in \eqref{eq:P_FN}:  
    \begin{equation}\label{eq:KLbd}
 	\kl{\P_{\phi,N},\P_{\psi,N}} \leq \frac{\tau M}{2} \E\left[ \left\|R_{\psi-\phi}[u]\right\|_\spaceY^2 \right],
 \end{equation}
 and prove the bound for the total variation distance: 
       \begin{equation}\label{ineq:tvbd}
           d_{\mathrm{tv}}\left(\P_{\phi},\P_{\psi}\right) \leq \frac{1}{2} \sqrt{\tau M\, \E\left[ \left\|R_{\psi-\phi}[u]\right\|_\spaceY^2 \right]}. 
       \end{equation}

    Recall that the $\R^N$-valued random vectors $Z_{\phi,u^m,N} $, $Z_{\psi,u^m,N} $, $Z_{f^m,N}$ and $Z_{\varepsilon^m,N}$ in \eqref{eq:Zepsilon} are induced by these samples. In particular, recall that $p_N$ is the probability density of $Z_{\varepsilon^m,N}$. 

 We show first the following change of measure:  
    \begin{equation}\label{eq:RNder_exp}
        \frac{d\P_{\phi,N}}{d\P_{0,N}} = \prod_{m=1}^M \frac{p_N\left(Z_{f^m,N}- Z_{\phi,u^m,N}   \right)}{p_N\left(Z_{f^m,N} \right)}. 
    \end{equation}
Note that under $\P_0$,  $Z_{f^m,N}$ has the same distribution as $ Z_{\varepsilon^m,N}$ and it is independent of $u^m$. By the independence of the samples, it suffices to consider \eqref{eq:RNder_exp} with $M=1$, which is reduced to verifying that for all $A_N\in \calF_N$, 
  \begin{align*}
     \E_0\left[\frac{d\P_{\phi,N}}{d\P_{0,N}} \mathbf{1}_{A_N} \right]
  & =  \E_0 \left[\frac{p_N\left(Z_{f^m,N}- Z_{\phi,u^m,N}   \right)}{p_N\left(Z_{f^m,N}\right)} \mathbf{1}_{A_N}\right] \\
  &= \E \left[\frac{p_N\left(Z_{\varepsilon^m,N}- Z_{\phi,u^m,N}   \right)}{p_N\left(Z_{\varepsilon^m,N}\right)} \mathbf{1}_{A_N}\right]
  = \P_{\phi,N}(A_N).  
  \end{align*}

As in \eqref{eq:AN_FN}, there exists $B_N\in\mathcal{B}(\spaceX)\otimes \mathcal{B}(\R^N)$ such that $A_N = \{\omega\in \Omega: (u^m(\omega), Z_{\varepsilon^m,N}(\omega))\in B_N \}
$.  Then, the independence between $u^m$ and $\varepsilon^m$ gives rise to
  \begin{align*}
  & \E \left[ \frac{p_N\left(Z_{\varepsilon^m,N}- Z_{\phi,u^m,N}   \right)}{p_N\left(Z_{\varepsilon^m,N} \right)}  \mathbf{1}_{B_N}(u^m,Z_{\varepsilon^m,N} ) \mid u^m \right]  \\
  = &\int_{\R^N} \frac{p_N\left(z  - Z_{\phi,u^m,N}   \right)}{p_N\left(z \right)}  \mathbf{1}_{B_N}(u^m,z ) p_N(z)\,dz = \int_{\R^N} p_N\left(z  - Z_{\phi,u^m,N}   \right)  \mathbf{1}_{B_N}(u^m,z ) \,dz\\
  = & \int_{\R^N} p_N\left(z'    \right)   \mathbf{1}_{B_N}(u^m,z'+ Z_{\phi,u^m,N})\, dz' = \E[\mathbf{1}_{B_N}( u^m,Z_{\varepsilon^m,N} + Z_{\phi,u^m,N} ) \mid u^m ].   
  \end{align*}
Then, we obtain \eqref{eq:RNder_exp} from
\begin{align*}
  & \E \left[\frac{p_N\left(Z_{\varepsilon^m,N}- Z_{\phi,u^m,N}   \right)}{p_N\left(Z_{\varepsilon^m,N} \right)} \mathbf{1}_{A_N} \right ] \\
 =  & \E \left[ \E \left[ \frac{p_N\left(Z_{\varepsilon^m,N}- Z_{\phi,u^m,N}   \right)}{p_N\left(Z_{\varepsilon^m,N} \right)}  \mathbf{1}_{B_N}(u^m,Z_{\varepsilon^m,N} +Z_{\psi,u^m,N} ) \mid u^m \right]  \right ] \\
 = & \E \Big[\E[\mathbf{1}_{B_N}(u^m,Z_{\varepsilon^m,N}+Z_{\phi,u^m,N} ) \mid u^m ]    \Big ] =\P(A_N^\phi)=  \P_{\phi,N}(A_N). 
 \end{align*}

To prove \eqref{eq:KLbd}, applying \eqref{eq:RNder_exp} with  $
\frac{d\P_{\phi,N}}{d\P_{\psi,N}} = \frac{d\P_{\phi,N}}{d\P_{0,N}}\cdot\left(\frac{d\P_{\psi,N}}{d\P_{0,N}}\right)^{-1}$, 
we obtain 
   \begin{equation}\label{eq:Pphi/Ppsi}
         \frac{d\P_{\phi,N}}{d\P_{\psi,N}} = \prod_{m=1}^M \frac{p_N\left(Z_{f^m,N}- Z_{\phi,u^m,N}   \right)}{p_N\left(Z_{f^m,N}- Z_{\psi,u^m,N}\right)}. 
    \end{equation}
Note that under $\P_\phi$, $Z_{f^m,N}$ has the same distribution as $Z_{\phi,u^m,N} + Z_{\varepsilon^m,N}$ under $\P_0$. Then, using $Z_{\phi,u^m,N} - Z_{\psi,u^m,N} = Z_{\phi-\psi,u^m,N}$ and the independence between $u^m$ and $\varepsilon^m$, we obtain 
\begin{align*}
   \E_{\phi}\left[\log \left(\frac{p_N\left(Z_{f^m,N}- Z_{\phi,u^m,N}   \right)}{p_N\left(Z_{f^m,N}- Z_{\psi,u^m,N}\right)}\right)\right] 
& = \E_{0}\left[\log \left(\frac{p_N\left(Z_{\varepsilon^m,N} \right)}{p_N\left(Z_{\varepsilon^m,N}+ Z_{\phi-\psi,u^m,N} \right)}\right)\right]\\ 
 & =   \E\bigg[\int_{\R^N}\log\left(\frac{p_N\left(z \right)}{p_N\left(z - Z_{\psi-\phi,u^m,N}\right)}\right) p_N(z)\,dz \bigg]. 
\end{align*}
Then, {\rm Assumption \ref{assump:noise-X-general} \ref{assum:noise_low}} and $\| Z_{\psi-\phi,u^m,N}\|_{\R^N}^2= \sum_{l=1}^N \innerp{R_{\psi-\phi}[u^m],y_l}_{\spaceY}^2$ imply 
    \begin{align*}
        &\kl{\P_{\phi,N},\P_{\psi,N}} = \E_{\phi}\left[\log\left(\frac{d\P_{\phi,N}}{d\P_{\psi,N}}\right)\right]=\sum_{m=1}^M \E_{\phi}\left[\log \left(\frac{p_N\left(Z_{f^m,N}- Z_{\phi,u^m,N}   \right)}{p_N\left(Z_{f^m,N}- Z_{\psi,u^m,N}\right)}\right)
        \right]\\
        =\, &\sum_{m=1}^M \E\bigg[\int_{\R^N}\log\left(\frac{p_N\left(z' \right)}{p_N\left(z' - Z_{\psi-\phi,u^m,N}\right)}\right) p_N(z')\,dz' \bigg]\\
        \leq&\sum_{m=1}^M \E\left[\frac{\tau}{2} \sum_{l=1}^N \innerp{R_{\psi-\phi}[u^m],y_l}_{\spaceY}^2 \right] \leq \frac{\tau}{2} M\, \E\left[ \left\|R_{\psi-\phi}[u]\right\|_\spaceY^2 \right]. 
    \end{align*}

   Additionally, when $\spaceY$ is finite-dimensional, Eq.\eqref{ineq:tvbd} follows directly from the Pinsker's inequality (i.e., $ d_{\mathrm{tv}}\left(\P_{0},\P_{1}\right) \leq \sqrt{\frac{1}{2}\kl{\P_{0},\P_{1}}}$ when $\P_0$ is absolutely continuous with respect to $\P_1$), 
       \begin{equation*}
        d_{\mathrm{tv}}\left(\P_{\phi},\P_{\psi}\right) = d_{\mathrm{tv}}\left(\P_{\phi,N},\P_{\psi,N}\right)\leq \sqrt{\frac{1}{2}\kl{\P_{\phi,N},\P_{\psi,N}}}\leq \frac{1}{2} \sqrt{\tau M\, \E\left[ \left\|R_{\psi-\phi}[u]\right\|_\spaceY^2 \right]}. 
    \end{equation*}
   When $\spaceY$ is infinite-dimensional, Eq.\eqref{ineq:tvbd} follows from 
    \begin{align*}
        &\ d_{\mathrm{tv}}\left(\P_{\phi},\P_{\psi}\right) = \lim_{N\to\infty}d_{\mathrm{tv}}\left(\P_{\phi,N},\P_{\psi,N}\right)\\
        \leq &\ \limsup_{N\to\infty}\sqrt{\frac{1}{2}\kl{\P_{\phi,N},\P_{\psi,N}}}
        \leq \frac{1}{2} \sqrt{\tau M\, \E\left[ \left\|R_{\psi-\phi}[u]\right\|_\spaceY^2 \right]}, 
    \end{align*}
  where the first equality follows from Lemma \ref{lemma:tv}.
\end{proof}

\begin{remark}[Loss function and likelihood]
\label{rmk:CM-space}
When the noise $\varepsilon$ is an isonormal Gaussian process, the loss function leading to the least squares estimator is a scaled log-likelihood of the data, i.e., $\calE_M(\phi)= -\frac{2}{M}\log \frac{d\P_{\phi}}{d\P_{0}}$.  When $\spaceY$ is finite-dimensional with dimension $N$, this follows directly from \eqref{eq:RNder_exp} with $\P_{\phi,N}= \P_\phi$ and $p_N(x)= \frac{1}{\sqrt{2\pi}}\exp(-\|x\|_{\R^N}^2/2)$:  
\begin{align*}
   -\frac{2}{M}\log\frac{d\P_{\phi,N}}{d\P_{0,N}} 
   & =  \frac{1}{M} \sum_{m=1}^M \left( \|Z_{f^m,N}- Z_{\phi,u^m,N}\|_{\R^N}^2 - \|Z_{f^m,N}\|_{\R^N}^2 \right) 
\\ & =  \frac{1}{M} \sum_{m=1}^M \left( \|Z_{\phi,u^m,N}\|_{\R^N}^2 - 2 \innerp{Z_{f^m,N},Z_{\phi,u^m,N}}_{\R^N} \right)  
\\ & =  \frac{1}{M} \sum_{m=1}^M \left(\|R_\phi[u^m]\|_\spaceY^2 - 2\innerp{f^m,R_\phi[u^m]} \right) = \calE_M(\phi). 
\end{align*}
 When $\spaceY$ is infinite-dimensional, we obtain $\calE_M(\phi)= -\frac{2}{M}\log \frac{d\P_{\phi}}{d\P_{0}}$ by sending $N\to+\infty$ and using the facts that $\|Z_{\phi,u^m,N}\|_{\R^N}^2\to \|R_\phi[u^m]\|_\spaceY^2 $ and $\innerp{Z_{f^m,N},Z_{\phi,u^m,N}}_{\R^N}\to  \innerp{f^m,R_\phi[u^m]}$ as $N\to \infty$. 
 
 In fact, the limit of \eqref{eq:RNder_exp} is the Girsanov change of measure {\rm\cite[Section 8.6]{oksendal2013_sde}} when we interpret the model as a stochastic differential equation, 
   \[
  \frac{d\P_\phi}{d\P_0} = \exp\left(- \frac{1}{2} \sum_{m=1}^M \big( \|R_{\phi}[u^m]\|_\spaceY^2  - 2\innerp{f^m,R_{\phi}[u^m]}\big) \right),
 \] 
which is closely related to the Cameron-Martin formula, e.g., {\rm \cite[Section 2.3]{da2006introduction}}.  
\end{remark}

\section{Examples}\label{sec:eg}

\subsection{Non-Gaussian noise} \label{sec:non-gaussian noise}
This section shows that the logistic distribution $\varepsilon$ on $\R^N$, which is non-Gaussian, satisfies {\rm Assumption \ref{assump:noise-X-general}}. 

\begin{example}\label{eg:logistic}
 Let $\varepsilon$ be an $\R^N$-valued random variable with i.i.d.~logistic-distributed marginal entries that have a probability density function $p(x) = \frac{e^{-x}}{(1+e^{-x})^2}$. 
Then, $\varepsilon$ satisfies {\rm Assumption \ref{assump:noise-X-general}}. 
   \end{example}
   
   \begin{proof}[Derivation for Example \ref{eg:logistic}]
   We start with the $1$-dimensional case. The logistic distribution with density $p(x) = \frac{e^{-x}}{(1+e^{-x})^2} = \frac{d}{dx}\left(\frac{1}{1+e^{-x}}\right)$ is of mean $0$ and variance $\pi^2/3$, which makes it a centered and square-integrable noise. To show \eqref{ineq:noise_low}, we consider
\[
\kl{p,p(\cdot +v)} = \int_{\R}\log\left(\frac{p(x)}{p(x+v)}\right)p(x)\,dx 
\]
for $|v|\leq 1$ and $|v|>1$ separately. When $v>1$,
\begin{align*}
    \kl{p,p(\cdot+v)} &= \int_\R \left(v+2\log\left(\frac{1+e^{-(x+v)}}{1+e^{-x}}\right)\right)\frac{e^{-x}}{(1+e^{-x})^2}\,dx\\
    &\leq \int_\R \left(v+2\log\left(\frac{1+e^{-x}}{1+e^{-x}}\right)\right)\frac{e^{-x}}{(1+e^{-x})^2}\,dx = v \leq v^2.
\end{align*}
When $v<-1$,
\begin{align*}
    \kl{p,p(\cdot+v)} &= \int_\R \left(v+2\log\left(\frac{1+e^{-(x+v)}}{1+e^{-x}}\right)\right)\frac{e^{-x}}{(1+e^{-x})^2}\,dx\\
    &\leq \int_\R \left(v+2\log\left(\frac{e^{-v}+e^{-(x+v)}}{1+e^{-x}}\right)\right)\frac{e^{-x}}{(1+e^{-x})^2}\,dx = -v \leq v^2.
\end{align*}
When $|v|\leq 1$, we claim that 
\begin{equation}\label{ineq:logistic_small_v}
    \kl{p,p(\cdot +v)} \leq \frac{25}{12}v^2.
\end{equation}
Thus, the KL divergence can be bounded by $25v^2/12$ for all $v\in\R$.

The $N$-dimensional case follows directly since the entries are i.i.d. The conditions in {\rm Assumption \ref{assump:noise-X-general}} hold because (i) the mean is zero and the covariance matrix is $(\pi^2/3) I_N$; and (ii), the joint distribution has a density $p_N(x_1,\ldots,x_N) = \Pi_{i=1}^N p(x_i)$, so for all $v\in \R^N$,       
\begin{align*}
    \kl{p_N,p_N(\cdot+v)} &=\int_{\R^N}\log\left(\prod_{i=1}^N\frac{p(x_i)}{p(x_i+v_i)}\right)\prod_{i=1}^N p(x_i)\,dx\\
    &=\int_{\R^N}\sum_{i=1}^N\log\left(\frac{p(x_i)}{p(x_i+v_i)}\right)\prod_{i=1}^N p(x_i)\,dx\\
    &= \sum_{i=1}^N\int_{\R}\log\left(\frac{p(x_i)}{p(x_i+v_i)}\right)p(x_i)\,dx_i\leq \frac{25}{12}\sum_{i=1}^N v_i^2 = \frac{25}{12}\|v\|^2.
\end{align*}

To prove Eq.\eqref{ineq:logistic_small_v} with $|v|\leq 1$, we resort to Lemma \ref{lemma:regularity} below, for which we need to verify the regularity conditions. Note that
\begin{align*}
	\frac{dp}{dx}(x)    & = \frac{-e^{-x}(1+e^{-x}) - e^{-x}\cdot(-2e^{-x})}{(1+e^{-x})^3} = \frac{e^{-2x}-e^{-x}}{(1+e^{-x})^3};\\ 
   \frac{d^2p}{dx^2}(x) & = \frac{(e^{-x}-2e^{-2x})(1+e^{-x}) - (e^{-2x}-e^{-x})(-3e^{-x})}{(1+e^{-x})^4} = \frac{e^{-x}-4e^{-2x}+e^{-3x}}{(1+e^{-x})^4}.
\end{align*}
Both $\left|\frac{dp}{dx}(x)\right|$ and $\left|\frac{d^2p}{dx^2}(x)\right|$ are continuous and of the order $e^{-|x|}$ as $x\to \pm \infty$. The same is thus true for $\sup_{|y-x|\leq 1}\left|\frac{dp}{dx}(y)\right|$ and $\sup_{|y-x|\leq 1}\left|\frac{d^2p}{dx^2}(y)\right|$. Thus, the first regularity condition holds with $v_0=1$.
 Also, the second regularity condition holds with $C_3 = 1/2$ by using $
\frac{d\log p}{dx}(x) = 1 - \frac{2}{1+e^{-x}}$  and $p(x) = \frac{d}{dx}\left(\frac{1}{1+e^{-x}}\right)$ to get 
\[
\left|\frac{d^3\log p}{dx^3}(x)\right| = 2\left|\frac{d p}{dx}(x)\right|\leq 2\left|(e^{-|x|})^2-e^{-|x|}\right|\leq \frac{1}{2}=:C_3.
\] 
Therefore, since the Fisher information is bounded by (recall that $
\frac{d\log p}{dx}(x) = -1+\frac{2e^{-x}}{1+e^{-x}}$)
\[
I(0) = \int_\R \left(-1+\frac{2e^{-x}}{1+e^{-x}}\right)^2p(x)\,dx\leq \int_\R 2^2p(x)\,dx = 4, 
\]
 Lemma \ref{lemma:regularity} implies that for $|v|\leq 1$, 
\[
\kl{p,p(\cdot+v)} \leq  \frac{1}{2} I(0)v^2 + \frac{C_3}{3!}|v|^3
\leq 2v^2 + \frac{1}{12}v^2\cdot 1 = \frac{25}{12}v^2, 
\]
which verifies Eq.\eqref{ineq:logistic_small_v}. 
   \end{proof}

The next lemma is reworded from \cite[Section 2.6]{kullback1978information}, which shows that 
\[
\kl{p_N,p_N(\cdot+v)} = \frac{1}{2}v^\top I(0)v + o(\|v\|^2) \text{ as }v\to 0.
\]
if $p_N$ is regular, where $I(v)$ is the Fisher information. This equation is used as a noise assumption in \cite[page 91]{tsybakov2008introduction}.

\begin{lemma}\label{lemma:regularity}
    Suppose $p_N:\R^N\to (0,+\infty)$ is a positive probability density function satisfying the following regularity conditions.
    \begin{itemize}
        \item $p_N$ is twice continuously differentiable at all $x\in\R^N$ and there are two Lebesgue integrable functions $F_1$ and $F_2$ such that for $1\leq i,j\leq N$, all $x\in\R^N$ and some $v_0>0$, 
        \[
        \sup_{\|y-x\|\leq v_0}\left|\frac{\partial p_N}{\partial x_i}(y)\right|\leq F_1(x),\ \sup_{\|y-x\|\leq v_0}\left|\frac{\partial^2 p_N}{\partial x_i\partial x_j}(y)\right|\leq F_2(x);
        \]
        \item Third-order directional derivatives exist for all $x\in\R^N$ and all directions, and are uniformly bounded:
               \[
        \left|\frac{d^3}{dt^3}\Big|_{t=0}\log p_N(x+tv)\right|\leq C_3 <+\infty\quad \text{ for all }x\in\R^N,\, v\in S^{N-1}.
        \]
    \end{itemize}
    Then, for $\|v\|\leq v_0$,
    \[
    \kl{p_N,p_N(\cdot+v)} \leq  \frac{1}{2}v^\top I(0)v + \frac{C_3}{3!}\|v\|^3.
    \]
    Here, $I(v) := \int_{\R^N}\nabla_v(\log(p_N(x+v)))(\nabla_v(\log(p_N(x+v))))^\top p_N(x+v)\,dx$ is the Fisher information.
\end{lemma}

\subsection{Proof of Proposition \ref{prop:compactLG} and Derviations for Examples \ref{example:nonlocal_opt}--\ref{example:Aggr_opt}}\label{sec:comop}
This section uses Proposition \ref{prop:compactLG} to show operators in Examples \ref{example:nonlocal_opt} and \ref{example:Aggr_opt} are compact.

\begin{proof}[Proof of Proposition \ref{prop:compactLG}] 
Since $\langle\LGbar\phi,\psi \rangle_{L_\rho^2} 
  = \E [\langle R_\phi[u], R_\psi[u]\rangle_{\spaceY} ]$ for all $\phi,\psi\in L^2_\rho$, we have 
\begin{align*}
    \langle\LGbar\phi,\psi \rangle_{L_\rho^2} 
  & = \E\left[\int_\calX\left(\int_\calS \phi(s)g[u](x,s)ds\right)\left(\int_\calS \psi(s)g[u](x,s)ds\right)\nu(dx)\right]\\
  &= \int_{\calS\times\calS}\phi(s)\psi(s')\E\left[\int_\calX g[u](x,s)g[u](x,s')\nu(dx)\right]ds\, ds'\\
  &= \int_{\calS\times\calS}\phi(s)\psi(s') \overline{G}(s,s') \dot{\rho}(s)\dot{\rho}(s') ds\, ds',
\end{align*}
where the integrand $\overline{G}(s,s')$ is  
\[
 \overline{G}(s,s') := \frac{G(s,s')}{\dot{\rho}(s)\dot{\rho}(s')}\,  \text{ with }  G(s,s'): =\E\left[\int_\calX g[u](x,s)g[u](x,s')\nu(dx)\right]. 
\]
The operator $\LGbar$ is self-adjoint since $\overline{G}(s,s')$ is symmetric, and it is nonnegative by definition since $\innerp{\LGbar\phi,\phi}=\E[\|R_\phi[u]\|^2_\spaceY]\geq 0$ for all $\phi\in L^2_\rho$. Thus, we only need to show that $\overline{G}\in L^2(\rho\otimes\rho)$, which implies that $\LGbar$ is compact. By Cauchy-Schwartz inequality, we have  
\begin{align*}
  G^2(s,s') =  &\ \left(\E\left[\int_\calX g[u](x,s)g[u](x,s')\nu(dx)\right]\right)^2   \\
    \leq &\  \E\left[\int_\calX g^2[u](x,s)\nu(dx)\right] \E\left[\int_\calX g^2[u](x,s')\nu(dx)\right] = Z^2\dot{\rho}(s)\dot{\rho}(s')
\end{align*}
with $Z= \int_\calS \E\left[\int_\calX | g[u](x,s)|^2 \nu(dx)\right]ds$. 
Then, 
$$   \int_{\calS\times\calS}\overline{G}^2(s,s')\dot{\rho}(s)\dot{\rho}(s') ds\, ds' 
     \leq \int_{\calS\times\calS} Z^2 \, ds\, ds'  =Z^2{\rm{vol}}^2(\calS)<+\infty,
$$ that is, $\overline{G}\in L^2(\rho\otimes\rho)$.  
\end{proof}

\begin{proof}[Derivation for Example \ref{example:nonlocal_opt}] 
Recall that the input functions are 
$ u(x)=\sum_{k=1}^\infty X_k \cos(2\pi k x)$,  
 where $\{X_k\}_{k=1}^\infty$ is a sequence of independent $\calN(0,\sigma_k^2)$ random variables with $\sum_{k=1}^\infty \sigma_k^2<+\infty$.
 Then,  using 
	 $\cos\bigl(2\pi k\,(x+s)\bigr) + \cos\bigl(2\pi k\,(x-s)\bigr)=  2 \cos\bigl(2\pi k x \bigr) \cos(2\pi k s)$, we have 
\begin{align*}
    g[u](x,s) & = u(x+s)+u(x-s)-2u(x) = 2\sum_{k=1}^\infty X_k \,\cos(2\pi  k \,x)\bigl(\cos(2\pi k \,s)-1\bigr).   
\end{align*}
Since $X_k$'s are independent with $\E[X_k X_j]=\sigma_k^2 \delta_{k,j}$, we have 
\[\E\Bigl[\,g[u](x,s)\,g[u](x,s')\Bigr]
=4 \sum_{k=1}^\infty \sigma_k^2 \,\bigl(\cos(2\pi k \,s)-1\bigr) \bigl(\cos(2\pi k \,s')-1\bigr)\,\cos^2(2\pi k\,x).
\]
Then, integrating in $x$ with the fact that $\int_0^1 \cos^2(2\pi k\,x) dx=\frac{1}{2}$, we obtain 
\[
\begin{aligned}
G(s,s') =  \int_0^1 \E\Bigl[\,g[u](x,s)\,g[u](x,s')\Bigr]\, dx =2\sum_{k=1}^\infty \sigma_k^2 \,\bigl(\cos(2\pi k \,s)-1\bigr) \bigl(\cos(2\pi k \,s')-1\bigr)\,.
\end{aligned}
\]
The series converges uniformly since $\sum_{k=1}^\infty \sigma_k^2<+\infty$, and $G$ is a continuous function $\calS\times \calS$. Then, $Z= \int_\calS G(s,s)ds<+\infty$, the exploration measure has density $\dot{\rho} = \frac{1}{Z}G(s,s)= \frac{1}{Z} 2\sum_{k=1}^\infty \sigma_k^2 \,\bigl(\cos(2\pi k \,s)-1\bigr)^2$. As a result, Proposition \ref{prop:compactLG} implies that the normal operator $\LGbar:L^2_\rho\to L^2_\rho$ is compact. 
\end{proof}

\begin{proof}[Derivation for Example \ref{example:Aggr_opt}] 
First, we show that $u(\cdot,\omega)$ is a random probability density function. 
Then, $u(x,\omega)\in (0,2)$
since $     \biggl|\sum_{n=1}^\infty a_n\zeta_n\cos(2\pi n x)\biggr|   \;\le\;\sum_{n=1}^\infty a_n<1$, 
and $\displaystyle\int_0^1u(x,\omega)\,dx = 1$ since each $\cos(2\pi n x)$ integrates to zero over $[0,1]$. 
Thus, $u(\cdot,\omega)$ is a probability density for each $\omega$. 

Next, we compute $ g[u](x,s)$ and show that 
\begin{equation}\label{eq:gu_double_series}
	g[u](x,s)
=\sum_{n=1}^\infty [\alpha_n(x,s) \zeta_n  + h_n(x,s)] + \sum_{n\neq m}\zeta_n\zeta_m R_{nm}(x,s), 
\end{equation}
where the deterministic functions $\alpha_n$ and $h_n$'s are   
$$
\alpha_n(x,s)= -4\pi n\,a_n \cos(2\pi n x)\sin(2\pi n s), \quad
h_n(x,s)
=-4\pi n a_n^2 \sin(2\pi n s) \cos(4\pi n x), 
$$
and $R_{mn}$'s are deterministic functions collecting off-diagonal terms. 

Recall that $ g[u](x,s)
=\bigl[u'(x+s)-u'(x-s)\bigr]\,u(x)
\;+\;\bigl[u(x+s)-u(x-s)\bigr]\,u'(x)
$ and 
$$
u(x) =1+\sum_n a_n\zeta_n\cos(2\pi n x),
\qquad
u'(x) =-\sum_n a_n\zeta_n\,(2\pi n)\,\sin(2\pi n x). 
$$
We have 
   $$
   \begin{aligned}
   u'(x+s)-u'(x-s)
   &= \sum_n -a_n\,(2\pi n)\,\zeta_n\bigl[\sin(2\pi n(x+s)) - \sin(2\pi n(x-s)) \bigr]\\
   &= \sum_n -a_n\,(2\pi n)\,\zeta_n\bigl[2\cos(2\pi n x)\,\sin(2\pi n s)\bigr]=  \sum_{n}\alpha_n(x,s)\,\zeta_n, \\
   u(x+s)-u(x-s)
   &= \sum_n a_n\,\zeta_n\bigl[\cos(2\pi n(x+s))-\cos(2\pi n(x-s))\bigr]\\
   &= \sum_n -a_n\,\zeta_n\bigl[2\sin(2\pi n x)\,\sin(2\pi n s)\bigr] 
    = \sum_{n}\beta_n(x,s)\,\zeta_n,
   \end{aligned}
   $$
 where $\displaystyle\alpha_n(x,s)=-4\pi n\,a_n \cos(2\pi n x)\sin(2\pi n s)$ and $\displaystyle\beta_n(x,s)=-2\,a_n\sin(2\pi n x)\sin(2\pi n s)$.

Then, we have  
\begin{align*}
g[u](x,s) & = \bigl[\sum_{n}\alpha_n\zeta_n\bigr]\bigl[1+\sum_{m}a_m\zeta_m\cos(2\pi m x)\bigr] +  
\bigl[\sum_{n}\beta_n\zeta_n\bigr]\bigl[-\sum_{m}a_m(2\pi m)\zeta_m\sin(2\pi m x)\bigr] \\
& =\sum_{n}\alpha_n\zeta_n
 + \, \sum_{n,m}\zeta_n\,\zeta_m\,
   \Bigl[\alpha_n\,a_m\cos(2\pi m x)\;-\;\beta_n\,a_m(2\pi m)\sin(2\pi m x)\Bigr].
\end{align*}

Splitting the double sum into the diagonal and off-diagonal terms, we write
\begin{align*}
g[u](x,s)
&=\sum_{n}
   \Bigl[ \alpha_n(x,s) \zeta_n 
    + h_n(x,s) + \sum_{m\neq n}\zeta_n\zeta_m R_{nm}(x,s)\Bigr], 
\end{align*}
where the diagonal term $h_n(x,s)$ and the off-diagonal term $R_{nm}$ with $m\neq n$ are
$$
\begin{aligned}
h_n(x,s)
&=\alpha_n(x,s)\,a_n\cos(2\pi n x)  - \beta_n(x,s)\,a_n(2\pi n)\sin(2\pi n x) \\
&=-4\pi n a_n^2 \sin(2\pi n s)[\cos^2(2\pi n x) - \sin^2 (2\pi n x)]  \\
&= -4\pi n a_n^2 \sin(2\pi n s) \cos(4\pi n x)\,, \\
R_{nm}(x,s)
&=\alpha_n(x,s)\,a_m\cos(2\pi m x)  - \beta_n(x,s)\, 2\pi m a_m \sin(2\pi m x) \\
&=  - 4\pi  a_n a_m \sin(2\pi ns) [ n\cos(2\pi n x)  \cos(2\pi m x) - m \sin(2\pi n x)  \sin(2\pi m x)]. 
\end{aligned}
$$

Lastly, we compute $G(s,s')=\int_0^1\E\bigl[g[u](x,s)\,g[u](x,s')\bigr]\,dx$ using \eqref{eq:gu_double_series}. The only nonzero contributions in the expectations $\E[\zeta_n\zeta_m\zeta_{n'}\zeta_{m'}]$ come from those terms with $(n',m')= (n,m)$ or $(n',m')= (m,n)$ since $\zeta_n$ are i.i.d.~Rademacher and  $\zeta_n^2=\zeta_n^2\zeta_m^2=1$, $\E[\zeta_n]=0$. We end up with
\begin{align*}
    G(s,s')
=\sum_{n=1}^\infty
  \int_0^1 \bigg[\alpha_n(x,s)\,\alpha_n(x,s')&+ h_n(x,s)\,h_n(x,s')\\
  &\quad+ \sum_{m\neq n}R_{mn}(x,s)\big(R_{mn}(x,s')+ R_{nm}(x,s')\big)\bigg]\,dx. 
\end{align*}
The $n$-th $\alpha$-term and the $n$-th $h$-term are  
$$
\begin{aligned}
\int_0^1 \alpha_n(x,s)\,\alpha_n(x,s')\,dx
&=\bigl[4\pi n a_n\bigr]^2\,\sin(2\pi n s)\,\sin(2\pi n s')\,
 \int_0^1\cos^2(2\pi n x)\,dx\\
&=8\pi^2 n^2\,a_n^2\;\sin(2\pi n s)\,\sin(2\pi n s'),  \\
\int_0^1 h_n(x,s)\,h_n(x,s')\,dx
&=\bigl[4\pi n\,a_n^2\bigr]^2\,\sin(2\pi n s)\,\sin(2\pi n s')\,
 \int_0^1\cos^2(4\pi n x)\,dx\\
&=8\pi^2 n^2\,a_n^4\;\sin(2\pi n s)\,\sin(2\pi n s'). 
\end{aligned}
$$
To compute $\int_0^1 \sum_{m\neq n}R_{mn}(x,s) R_{mn}(x,s') dx$, we denote
$
C_{nm}(x)=\cos(2\pi n x)\cos(2\pi m x),
$
$S_{nm}(x)=\sin(2\pi n x)\sin(2\pi m x),$ 
and write $R_{n m}(x,s)= -4\pi  a_n a_m \sin(2\pi ns) \bigl[n\,C_{nm}(x)-m\,S_{nm}(x)\bigr]$. 
Thus, 
$$
\int_0^1R_{n m}(x,s)\,R_{n m}(x,s')\,dx
=(4\pi\,a_n a_m)^2 \,\sin(2\pi n s)\,\sin(2\pi n s') \int_0^1 \bigl[n\,C_{nm}(x)-m\,S_{nm}(x)\bigr]^2\,dx.
$$
But for integers $n\neq m$, 
$\int_0^1C_{nm}(x)^2\,dx
=\int_0^1S_{nm}(x)^2\,dx
=\frac14$, and $\int_0^1C_{nm}(x)\,S_{nm}(x)\,dx=0$, so
$$
\int_0^1 \bigl[n\,C_{nm}(x)-m\,S_{nm}(x)\bigr]^2 \,dx =\frac{n^2+m^2}{4}.
$$
Putting it all together, the total off-diagonal contribution for mode $n$ is
$$
\int_0^1\sum_{m\neq n}R_{n m}(x,s)\,R_{n m}(x,s')\,dx
=4\pi^2\,a_n^2\,\sin(2\pi n s)\,\sin(2\pi n s')\;
\sum_{m\neq n}a_m^2\,(n^2+m^2).
$$
Likewise,
$$
\int_0^1\sum_{m\neq n}R_{n m}(x,s)\,R_{m n}(x,s')\,dx
=8\pi^2\,a_n^2\,\sin(2\pi n s)\;
\sum_{m\neq n}a_m^2\,\sin(2\pi m s')\,nm.
$$

Summing over $n$ gives
$$
\begin{aligned}
G(s,s')
=\sum_{n=1}^\infty 
   \bigg\{& \Bigl[8\pi^2 n^2 a_n^2 + 8\pi^2 n^2 a_n^4 +4\pi^2\,a_n^2 \sum_{m\neq n}a_m^2\,(n^2+m^2) \Bigr]\, \sin(2\pi n s)\,\sin(2\pi n s') \\
&\qquad +8\pi^2 \sum_{m\neq n} a_n^2 a_m^2mn\sin(2\pi ns)\sin(2\pi ms')\bigg\}.
\end{aligned}
$$
The series converges absolutely since $\sum_{n\geq 1} n a_n<1$, which implies that $\sum_{n\geq 1} n^2 a_n^2\leq C \sum_{n\geq 1} n a_n \leq C$ with $C= \sup_n ( na_n ) <+\infty$ . 	
\end{proof}

 \bibliographystyle{plain}
\bibliography{ref_FeiLU2025_1,ref_IPS_stochastic,ref_minimax,ref_regularization24_12,ref_kernel_methods25_1,ref_nonlocal_kernel25_02}  
\end{document}